\documentclass{amsart}
\usepackage{amssymb, graphics}

\usepackage[all]{xy}
\newcommand\datver[1]{\def\datverp%
 {\par\boxed{\boxed{\text{#1; Run: \today}}}}}

\numberwithin{equation}{section}
\newtheorem{lemma}{Lemma}[section]
\newtheorem{proposition}[lemma]{Proposition}
\newtheorem{corollary}[lemma]{Corollary}
\newtheorem{Theorem}[lemma]{Theorem}
\newtheorem{theorem}[lemma]{Theorem}
\newtheorem{non-theorem}[lemma]{Non-Theorem}

\newtheorem{example}[lemma]{Example}
\newtheorem{definition}[lemma]{Definition}
\newtheorem{remark}[lemma]{Remark}
\newtheorem{assumption}{Assumption}
\newtheorem*{acknowledgements}{Acknowledgements}

\newcommand\cf{cf\@. }

\newcommand\pa{ \partial}
\newcommand{\eq}[1][r]
   {\ar@<-3pt>@{-}[#1]
    \ar@<-1pt>@{}[#1]|<{}="gauche"
    \ar@<+0pt>@{}[#1]|-{}="milieu"
    \ar@<+1pt>@{}[#1]|>{}="droite"
    \ar@/^2pt/@{-}"gauche";"milieu"
    \ar@/_2pt/@{-}"milieu";"droite"}

\newcommand\bbB{\mathbb B}
\newcommand\bbC{\mathbb C}

\newcommand\bbN{\mathbb N}

\newcommand\bbR{\mathbb R}
\newcommand\bbS{\mathbb S}
\newcommand\bbT{\mathbb T}

\newcommand\bbZ{\mathbb Z}

\newcommand\cS{\mathcal{S}}
\newcommand\cU{\mathcal{U}}
\newcommand\tcU{\widetilde{\mathcal{U}}}
\newcommand\cV{\mathcal{V}}

\newcommand\tV{\widetilde{V}}
\newcommand\tM{\widetilde{M}}
\newcommand\tX{\widetilde{X}}
\newcommand\tE{\widetilde{E}}
\newcommand\tF{\widetilde{F}}

\newcommand\tf{\widetilde{f}}
\newcommand\tP{\widetilde{P}}
\newcommand\tQ{\widetilde{Q}}
\newcommand\tS{\widetilde{S}}
\newcommand\tnu{\widetilde{\nu}}
\newcommand\tp{\widetilde{p}}
\newcommand\tih{\widetilde{h}}

\newcommand\End{\operatorname{End}}
\newcommand\cF{\mathcal{F}}
\newcommand\cG{\mathcal{G}}
\newcommand\cI{\mathcal{I}}
\newcommand\CI{\mathcal{C}^{\infty}}
\newcommand\Diff{\operatorname{Diff}}
\newcommand\pr{\operatorname{pr}}
\newcommand\inv{\operatorname{inv}}
\newcommand\Id{\operatorname{Id}}
\newcommand\SN{\operatorname{SN}}
\newcommand\ff{\operatorname{ff}}
\newcommand\ph{\operatorname{ph}}
\newcommand\ad{\operatorname{ad}}
\newcommand\cusp{\operatorname{cu}}
\newcommand\sus{\operatorname{sus}}
\newcommand\Hom{\operatorname{Hom}}
\newcommand\supp{\operatorname{supp}}

\newcommand\cl{\operatorname{cl}}
\newcommand\SU{\operatorname{SU}}
\newcommand\Ind{\operatorname{Ind}}
\newcommand\Ch{\operatorname{Ch}}
\newcommand\pt{\operatorname{pt}}
\newcommand\Cl{\mathbb{C}\ell}
\renewcommand\vert{\operatorname{vert}}
\newcommand\inte[1]{ {#1}\setminus\pa {#1}}
\newcommand\cL{\mathcal{L}}
\newcommand\cE{\mathcal{E}}
\newcommand\cA{\mathcal{A}}
\newcommand\cB{\mathcal{B}}
\newcommand\cP{\mathcal{P}}
\newcommand\cC{\mathcal{C}}
\newcommand\cK{\mathcal{K}}
\newcommand\ALF{\operatorname{ALF}}
\newcommand\sign{\operatorname{sign}}
\newcommand\SO{\operatorname{SO}}

\datver{0.1B; Revised:  November 3 2011}
\begin{document}
\title[Foliated cusp operators]
{Pseudodifferential operators on manifolds with foliated boundaries}

\author{Fr\'ed\'eric Rochon}
\address{Department of mathematics, Australian National University}
\email{Frederic.Rochon@anu.edu.au}
\dedicatory{\datverp}
\thanks{The author was partially supported by a NSERC discovery grant}

\begin{abstract}   
Let $X$ be a smooth compact manifold with boundary.  For smooth foliations on the boundary of $X$ admitting a `resolution' in terms of a fibration, we construct a pseudodifferential calculus generalizing the fibred cusp calculus of Mazzeo and Melrose.  In particular, we introduce certain symbols leading to a simple description of the Fredholm operators inside the calculus.  When the leaves of the fibration `resolving' the foliation are compact, we also obtain an index formula for Fredholm perturbations of Dirac-type operators.  Along the way, we obtain a formula for the adiabatic limit of the eta invariant for invertible perturbations of Dirac-type operators, a result of independent interest generalizing the well-known formula of Bismut  and Cheeger.
\end{abstract}
\maketitle

\tableofcontents

\section*{Introduction}

Let $X$ be a smooth compact manifold with non-empty boundary $\pa X$.  Assume that $\pa X$ is the total space of a smooth fibration
\begin{equation*}
\xymatrix{
                 Z  \ar@{-}[r]  & \pa X \ar[d]^{\Phi} \\
                                              & Y
                                              }
\end{equation*}
where the base $Y$ and the typical fibre $Z$ are smooth closed manifolds.  Let also $x\in \CI(X)$ be a choice of boundary defining function.  On the interior of $X$, consider a complete Riemannian metric $g_{\Phi}$ which in a collar neighborhood $c: \pa X\times [0,\epsilon)_{x}\to X$    of $\pa X$ takes the form
\[
   g_{\Phi}= \frac{dx^{2}}{x^{4}} + \frac{\Phi^{*}h}{x^{2}} + \kappa
\]
where $\kappa$ is a symmetric 2-tensor restricting to a Riemannian metric on each fibre of $\Phi$, and $h$ is a Riemannian metric on $Y$.  When $Y=\pa X$ and $\Phi$ is the identity map, this means $g_{\Phi}$ is a conical metric near $\pa X$.  The Euclidean metric on $\bbR^{n}$ is the prototypical example of a metric of this form.   When instead $Y$ is a point, $(X\setminus \pa X,g_{\Phi})$ is a manifold with a cylindrical end.  In general, when $0<\dim Y < \dim \pa X$, the metric $g_{\Phi}$ behaves like a cone in the base of the fibration and like a cylindrical end in the fibres.  Important examples of metrics of this form are given by gravitational instantons in the ALF and ALG classes such as the multi-Taub-NUT metric and the reduced 2-monopole moduli space metric. 

To study geometric operators arising in this context, like the Laplacian and the Dirac operator, Mazzeo and Melrose introduced in \cite{Mazzeo-Melrose} the fibred cusp calculus (also called $\Phi$-calculus), a pseudodifferential calculus naturally adapted to this sort of metrics.  When $Y=\pa X$ and $\Phi$ is the identity map, this reduces to the scattering calculus of \cite{MelroseGST}, while when $Y$ is a point, this is the cusp calculus, a calculus intimately related to the $b$-calculus of \cite{MelroseAPS}.  

As for other related calculi of pseudodifferential operators, like the $b$-calculus, the $0$-calculus \cite{Mazzeo-Melrose0}, the Edge Calculus \cite{MazzeoEdge} or the $\Theta$-calculus \cite{EMM}, the starting point of the construction of \cite{Mazzeo-Melrose} is a certain  space of vector fields, namely
\[
  \cV_{\Phi}(X)= \{\xi\in \Gamma(TX) \quad | \quad \exists \, C>0 \
  \mbox{such that} \; g_{\Phi}(\xi(p),\xi(p))< C\; \forall \ p\in 
 X\setminus \pa X \},
\]  
the space of fibred cusp vector fields.
If $(x,y,z)$ are coordinates in a small neighborhood near $\pa X$ where the fibration $\Phi$ is trivial with $y$ and $z$ being local coordinates on the base $Y$ and the fibre $Z$ respectively, then in that neighborhood, a fibred cusp vector field $\xi\in \cV_{\Phi}(X)$ is necessarily of the form
\[
   \xi= a x^{2}\frac{\pa}{\pa x} + \sum_{i} b^{i} x\frac{\pa}{\pa y^{i}} + \sum_{j} c^{j} \frac{\pa}{\pa z^{j}}
\]
where $a,b^{i},c^{j}\in \CI(X)$ are smooth functions up to the boundary.  In fact, the space 
$\cV_{\Phi}(X)$ can be defined alternatively by
\[
 \cV_{\Phi}(X)= \{ \xi \in \Gamma(TX) \quad | \quad 
  \xi x\in x^{2}\CI(X) \; \mbox{and} \; 
\Phi_{*}(\left.\xi\right|_{\pa X})=0 \}.
\]
The space $\cV_{\Phi}(X)$ is easily seen to be a Lie subalgebra of $\Gamma(TX)$.  This means we can define unambiguously the space 
$\Diff^{k}_{\Phi}(X)$ of $\Phi$-differential operators of order $k$ as those operators on $\CI(X)$ which can be written as a finite sum of products of smooth functions  with at most $k$ elements of $\cV_{\Phi}(X)$.  To define more generally the space $\Psi^{m}_{\Phi}(X)$ of $\Phi$-pseudodifferential operators of order $m$, Mazzeo and Melrose describe their Schwartz kernels as conormal distributions on the manifold $X\times X$ suitably blown up at the corner $\pa X\times \pa X$.  The $\Phi$-pseudodifferential operators map smooth functions to smooth functions and they map the subspace $\dot{\cC}^{\infty}(X)\subset \CI(X)$ of functions vanishing to infinite order at $\pa X$ to itself.  They are closed under composition and the $\Phi$-operators of order $0$ induce bounded linear operators on the space $L^{2}_{\Phi}(X\setminus \pa X)$ of square integrable functions with respect to the density defined by $g_{\Phi}$.  More generally, there are natural Sobolev spaces on which $\Phi$-pseudodifferential operators act.  

The notions of symbol and ellipticity have a natural generalization in this context.  However, ellipticity is not enough to insure an operator is Fredholm.  Another `symbol', called the normal operator, which encodes the asymptotic behavior of the operator at infinity, must also be invertible.  In that case, one says the operator is fully elliptic.  The criterion of Mazzeo and Melrose is that a $\Phi$-operator is Fredholm (when acting on suitable Sobolev spaces)  if and only if it is fully elliptic.  One of the nice features of the $\Phi$-calculus is that, in contrast with other types of pseudodifferential calculi, the inverse of an invertible fully elliptic $\Phi$-operator is automatically in the calculus.  

In \cite{Mazzeo-Melrose}, Mazzeo and Melrose raised the problem of finding a nice formula for the index of fully elliptic $\Phi$-operators.  When $Y$ is a point, a formula for the index of fully elliptic $\Phi$-operators is essentially given by the Atiyah-Patodi-Singer index theorem \cite{APS1}.
When $Y=\pa X$, then the problem can be reduced to the Atiyah-Singer index theorem \cite{Atiyah-Singer0} as explained in \cite{MelroseGST}.  In the intermediate cases, obtaining a satisfactory formula is more delicate.  When $X\setminus \pa X= \bbS^{1}\times \bbR^{3}$, a formula was obtained by \cite{Nye-Singer} for some operators coming from gauge theory.  In \cite{Lauter-Moroianu}, Lauter and Moroianu computed the Hochschild homology of the algebra of $\Phi$-operators and obtained an index formula in this framework.  In \cite{HHM}, the $L^{2}$-cohomology of $\Phi$-metrics is computed in terms of intersection homology.  An index in $K$-theory for families was obtained in \cite{Melrose-Rochon06}.  In \cite{Moroianu}, Moroianu used the index theorem of Vaillant \cite{Vaillant} to obtain one for Dirac-type $\Phi$-operators.  This formula was also obtained in \cite{LMP} using the adiabatic calculus of \cite{Melrose-Rochon06}.  This was subsequently generalized in \cite{Albin-Rochon1} to include Fredholm perturbations of Dirac-type $\Phi$-operators and to deal with families.

In the present paper, we generalize the $\Phi$-calculus to situations where the fibration $\Phi$ on the boundary is replaced by a smooth foliation $\cF$.  Since a foliation locally looks like a fibration, the notions of $\Phi$-vector fields and $\Phi$-differential operators have obvious generalizations.  However, the passage from a fibration to a foliation is much more delicate for pseudodifferential operators.  This is because pseudodifferential operators are not local, so the global aspects of the foliation have to be taken into account in their definition.  One way to proceed is to use the general construction of \cite{ALN07} for manifolds with Lie structure at infinity.  The calculi obtained in this way have many of the usual properties, but they are typically smaller than the calculi constructed \`a la Melrose.   This makes certain constructions, like the one of a parametrix, more delicate.  To be able to use some of the known results about $\Phi$-operators, notably about the index of fully elliptic ones, we will proceed differently.  

More precisely, we will assume the foliation $\cF$ can be `resolved' into a fibration with possibly non-compact fibres but with compact base (see Assumption~\ref{tfb.2b} at the beginning of \S~\ref{rf.0}).   This certainly impose a restriction.  Still, as we indicate in \S~\ref{rf.0}, a wide variety of natural examples arise in this way.  For this type of foliations, one can then define $\cF$-pseudodifferential operators using $\Phi$-pseudodifferential operators on the fibration `resolving' the foliation.  Standard mapping properties and the fact the $\cF$-calculus is closed under composition follow without too much effort.  The notion of symbol and ellipticity also has an obvious generalization.  The introduction of a normal operator requires more work, but eventually leads to simple criteria describing the $\cF$-operators which are compact or Fredholm.  One important new feature is that, contrary to what happens for $\Phi$-operators, the inverse of an invertible fully elliptic $\cF$-operator is not necessarily in the calculus.  In fact, the construction of a parametrix in the spirit of \cite{Mazzeo-Melrose} only works in certain special cases (see Theorem~\ref{pc.14} and Corollary~\ref{pcf.1}).

When the fibres of the fibration `resolving' the foliation are compact (see Assumption~\ref{fini.1} at the beginning of \S~\ref{it.0}), we are able to obtain an index formula for (Fredholm perturbations of) Dirac-type $\cF$-operators  (see Theorem~\ref{ad.10}).  The formula is quite similar to the one  of \cite{Albin-Rochon1} for Dirac-type $\Phi$-operators, except that it has a new term, a $\rho$-invariant encoding how the normal operator is lifted to the fibration `resolving' the foliation.  To obtain this formula, our strategy, inspired from \cite{LMP}, is to start with a cusp metric, that is, a $\Phi$-metric with $Y=\pt$, and to compute the limit of the known formula when the metric is deformed to an $\cF$-metric.  The main step is to compute the adiabatic limit of some eta invariant, which we can do using the index formula of \cite{Albin-Rochon1}.  This gives in this way a generalization of the adiabatic limit of Bismut and Cheeger to invertible perturbations of Dirac-type operators (see Theorem~\ref{ad.13}), a result of independent interest.   This should be compared with the recent work \cite{Goette2011} of Goette, who independently obtained an adiabatic limit for the eta invariant using different methods.    

The paper is organized as follows.  In \S~\ref{mfb.0}, we introduce the relevant geometric structures arising on a manifold with foliated boundaries.  It is followed in \S~\ref{mfdo.0} by a brief description of the construction of pseudodifferential operators using groupoids.  In section~\ref{rf.0}, we focus our attention on foliations that can be `resolved' by a fibration and we define the algebra of $\cF$-pseudodifferential operators.  The notions of symbol and of normal operator are introduced in \S~\ref{symb.0}.  We then define the natural Sobolev spaces on which $\cF$-operators act and provide a compactness criterion.  Before obtaining a Fredholm criterion in \S~\ref{fc.0}, we need in \S~\ref{sss.0} to develop an adequate notion of Sobolev spaces for $\cF$-suspended operators.  In \S~\ref{it.0}, we state and prove our index theorem for Fredholm perturbations of Dirac-type $\cF$-operators.  Finally, in \S~\ref{qmtn.0}, we use our results to compute the index of Dirac operators in some natural examples.

\begin{acknowledgements}
The author is grateful to Paolo Piazza for helpful discussions and to an anonymous referee for useful comments.
\end{acknowledgements}

\section{Manifolds with foliated boundaries} \label{mfb.0}

Let $X$ be a smooth manifold of dimension $n$ with non-empty boundary
$\pa X$.  Suppose that $\cF$ is a smooth foliation on the boundary 
$\pa X$ and denote by $\ell$ the dimension of the leaves.  To ease the presentation, we are assuming the boundary $\pa X$ is connected, but the reader should keep in mind that the results presented in this paper admit straightforward generalizations to situations where $\pa X$ has more than one connected component (with the dimension of the leaves of the foliation possibly varying from one connected component to another).     Fix also once
and for all a boundary defining function $x\in \CI(X)$, that is, $x$ is 
a function such that $x>0$ in the interior of $X$, $x=0$ on $\pa X$ and 
the differential of $x$ is nowhere zero on $\pa X$.  

With this data, we can define the space of \textbf{foliated cusp vector fields} (or \textbf{$\cF$-vector fields}) by
\begin{equation}
 \cV_{\cF}(X)= \{ \xi \in \Gamma(TX) \quad | \quad 
  \xi x\in x^{2}\CI(X) \; \mbox{and} \; 
\left.\xi\right|_{\pa X}\in \Gamma(T\cF) \}
\label{mfb.1}\end{equation}
where $T\cF$ is the distribution associated to the foliation $\cF$.  The 
condition $\left.\xi\right|_{\pa X}\in \Gamma(T\cF)$ simply means that 
$\xi$ is required to be tangent to the leaves of the foliation.  As a 
particular case, we recover the space of fibred cusp vector fields 
introduced in \cite{Mazzeo-Melrose} when the 
leaves of $\cF$ are given by the fibres of a smooth fibration.  

As can be seen directly from the definition, the space $\cV_{\cF}(X)$ 
is a Lie subalgebra of the Lie algebra of smooth vector fields $\Gamma(TX)$.  
Thanks to this property, we can define the space $\Diff^{m}_{\cF}(X)$ of
\textbf{foliated cusp differential operators of order $m$} or \textbf{$\cF$-differential operators of order $m$} as the space of operators on
$\CI(X)$ generated by $\CI(X)$ and products of up to $m$ elements of 
$\cV_{\cF}(X)$. 

By a theorem of Serre and Swan, the Lie algebra $\cV_{\cF}(X)$ can be 
identified with the space of sections
of a certain vector bundle, namely the \textbf{ $\cF$-tangent bundle}
${}^{\cF}TX$.  This vector bundle is defined as follows.  For 
$p\in X$, let $\cI_{p}(X)\subset \CI(X)$ be the ideal of functions vanishing
at $p$ and set
\begin{equation}
    {}^{\cF}T_{p}X = \cV_{\cF}(X)/ (\cI_{p}(X)\cV_{\cF}(X)).
\label{mfb.2}\end{equation}
The $\cF$-tangent vector bundle ${}^{\cF}TX$ is then the vector bundle whose
fibre above $p\in X$ is given by \eqref{mfb.2}.  The theorem of 
Serre and Swan insures that it  has a natural smooth
structure and is such that there is a canonical identification
$\cV_{\cF}(X)= \Gamma({}^{\cF}TX)$.  The restriction of ${}^{\cF}TX$ 
to the interior $\inte{X}$ of $X$ is naturally isomorphic to 
$\left. TX\right|_{\inte{X}}$.  This isomorphism extends to a natural
map of vector bundles
\begin{equation}
  \rho : {}^{\cF}TX\to TX.
\label{mfb.3}\end{equation}
On the boundary however, this map fails to be an isomorphism.  In fact, the
kernel of its restriction to the boundary $\left.\rho\right|_{\pa X}:
\left.{}^{\cF}TX\right|_{\pa X} \to \left. TX\right|_{\pa X}$ forms a smooth
vector bundle ${}^{\cF}N\pa X \to \pa X$ of rank $n-\ell$.  Notice however that
despite the fact the map $\rho$ is not an isomorphism of vector bundles, the 
bundles ${}^{\cF}TX$ and  $TX$ are nevertheless isomorphic, although not
in a natural way.  The map $\rho$ also induces a map of sections 
\begin{equation}
  \rho_{\Gamma}: \Gamma({}^{\cF}TX)\to \Gamma(TX)
\label{mfb.4}\end{equation}
which is just the natural inclusion $\cV_{\cF}(X)\subset \Gamma(TX)$ 
under the identification of $\cV_{\cF}(X)$ with $\Gamma({}^{\cF}TX)$.
This discussion can be conveniently summarized by the notion of Lie algebroid
which we now recall.  

\begin{definition}
A \textbf{Lie algebroid} $E$ over a manifold $M$ (possibly with corners) is a vector 
bundle $E$ over $M$ together with a Lie algebra structure on its space of
smooth sections $\Gamma(E)$ and a bundle map $a:E\to TM$ extending to
a map of sections $a_{\Gamma}: \Gamma(E)\to \Gamma(TM)$ such that 
\begin{itemize}
\item[(i)] $a_{\Gamma}([X,Y])= [ a_{\Gamma}(X),a_{\Gamma}(Y)]$ for
 all $X,Y \in \Gamma(E)$,
\item[(ii)] $[X,fY]= f[X,Y] + (a_{\Gamma}(X)f) Y$ for all
$X,Y\in \Gamma(E)$ and $f\in \CI(M)$.
\end{itemize}
The map $a_{\Gamma}$ is called the \textbf{anchor map} of $E$.  
\label{mfb.5}\end{definition}
Clearly, the $\cF$-tangent bundle ${}^{\cF}TX$ is a Lie algebroid with 
anchor map given by \eqref{mfb.4}.  Since all the vector fields of 
$\cV_{\cF}(X)$ are tangential to the boundary of $X$, the Lie
algebroid ${}^{\cF}TX$ is also said to be boundary tangential in the
terminology of \cite{ALN04}.  

If $g_{\cF}$ is a choice of smooth metric for the vector bundle ${}^{\cF}TX$,
then its restriction to $\left. {}^{\cF}TX\right|_{\inte{X}}=\left.TX\right|_{\inte{X}}$ induces a Riemannian metric on $\inte{X}$ also denoted $g_{\cF}$.
We will say the metric $g_{\cF}$ is a \textbf{$\cF$-metric} or \textbf{foliated boundary metric}.  The metric $g_{\cF}$ gives a particular example of a 
Riemannian manifold with Lie structure at infinity, a notion extensively
studied in \cite{ALN04}.  In our case, the Lie structure at infinity is
specified by the foliation $\cF$ and the boundary defining function $x$.  Among other things, we know from
the general results of \cite{ALN04} that $(\inte{X},g_{\cF})$ is complete.
Notice that the Lie algebra of foliated cusp vector fields can also be 
defined in terms of the metric $g_{\cF}$,
\begin{equation}
\cV_{\cF}(X)=\{ \xi \in \Gamma(TX) \quad | \quad \exists \, C>0 \
  \mbox{such that} \; g_{\cF}(\xi(p),\xi(p))< C\; \forall \ p\in 
\inte{X}\}.
\label{mfb.6}\end{equation}  

The Laplacian $\Delta_{g_{F}}$ associated to an $\cF$-metric $g_{\cF}$ is a 
natural example of foliated cusp operator.  More generally, any
reasonable differential operator geometrically constructed from the 
metric $g_{\cF}$, for instance the Dirac operator when $(\inte{X},g_{\cF})$ has
a spin structure, will give an example of foliated cusp operator.  

The motivation behind the terminology \emph{foliated cusp} comes from another
type of metrics, namely foliated cusp metrics.  A \textbf{foliated cusp metric}
$g_{\cF-c}$ is a metric of the form
\begin{equation}
   g_{\cF-c}= x^{2}g_{\cF}
\label{mfb.7}\end{equation} 
for some $\cF$-metric $g_{\cF}$.  In particular, a foliated cusp metric is
always conformal to an $\cF$-metric.  In a local chart 
$\cU= [0,\epsilon)\times F\times B$, near $\pa X$ on which the foliation looks
like a fibration with its leaves given by $\{0\}\times F\times \{b\}$ for
$b\in B$, a simple example of such metric is given by 
\begin{equation}
      \frac{dx^{2}}{x^{2}}+ g_{B}+ x^{2}g_{F}
\label{mfb.8}\end{equation} 
where $g_{B}$ and $g_{F}$ are (pull-backs of) metrics on $B$ and $F$ 
respectively.  When we restrict this metric to a leaf 
$(0,\epsilon)\times F\times \{b\}$ for some $b\in B$, we get a cusp metric
$\frac{dx^{2}}{x^2}+ x^{2}g_{F}$, suggesting the metric \eqref{mfb.8}
as a whole is a foliated cusp.  The corresponding $\cF$-metric is of the form
\begin{equation}
      \frac{dx^{2}}{x^{4}}+ \frac{g_{B}}{x^{2}}+ g_{F}
\label{mfb.9}\end{equation}
in this local chart. 
A particular feature of foliated cusp metrics is that they give non-compact
complete Riemannian manifolds of finite volume.  As for an $\cF$-metric, one
can use a foliated cusp metric $g_{\cF-c}$ to define a space of vector 
fields
\begin{equation}
\cV_{\cF-c}(X)=\{ \xi \in \Gamma(TX) \quad | \quad \exists \, C>0 \
  \mbox{such that} \; g_{\cF-c}(\xi(p),\xi(p))< C\; \forall \ p\in 
\inte{X}\}.
\label{mfb.10}\end{equation}  
As can be check directly however, this space is not a Lie subalgebra of the 
Lie algebra $\Gamma(TX)$, which is not so convenient to define a corresponding
algebra of differential operators.  Instead, one can notice that 
$\cV_{g_{\cF-c}}(X)= x^{-1}\cV_{\cF}(X)$, which suggests the following natural
definition for differential operators of order $m$ associated to \eqref{mfb.10},
\begin{equation}
    \Diff_{g_{\cF-c}}^{m}(X)= x^{-m}\Diff^{m}_{\cF}(X).
\label{mfb.11}\end{equation}
This definition is consistent with the fact the Laplacian of a foliated cusp
metric is an element of $x^{-2}\Diff^{2}_{\cF}(X)$.  

\section{Microlocalization using groupoids}\label{mfdo.0}

As was shown in many circumstances, even if one is solely interested in 
differential operators, it is often useful to have a corresponding 
calculus of pseudodifferential operators to study them, for instance to determine
if a given differential operator is Fredholm.  For the algebra
$\Diff^{*}_{\cF}(X)$, such a microlocalization could also be useful to 
study operators arising from a foliated cusp metric as in \eqref{mfb.11}.  For instance, when
the foliation $\cF$ comes from a fibration, this point of view was used very successfully by Vaillant in his thesis \cite{Vaillant} to study certain Dirac-type
operators associated to fibred cusp metrics.  
As shown in \cite{ALN07} and \cite{NWX}, one very general way of 
microlocalizing an algebra of differential operators
such as $\Diff^{*}_{\cF}(X)$ comes from groupoid theory.  To describe this
construction, we will first recall briefly the definition of a Lie groupoid.

A \textbf{groupoid} $\cG$ is a category whose morphisms are invertible and form
a set.  Let us denote by $\cG^{(1)}$ its set of morphisms and by $M$ its set
of objects.  An element $g\in \cG^{(1)}$ has a domain $d(g)\in M$ and 
a range $r(g)\in M$.  This defines maps
\begin{equation}
  d: \cG^{(1)}\to M, \quad r: \cG^{(1)}\to M.
\label{mfdo.1}\end{equation}  
Since two morphisms compose when the range of one is the domain of the other,
there is  a composition map 
\begin{equation}
  \mu: \{(g,h)\in \cG^{(1)}\times \cG^{(1)}\; | \; r(h)=d(g)\}\to M,
\quad \mu(g,h)= g\circ h.
\label{mfdo.1b}\end{equation}
Since every morphism is invertible, there is also an inverse map
\begin{equation}
\begin{array}{lccc}
  \inv:& \cG^{(1)} & \to & \cG^{(1)}  \\
       & g & \mapsto & g^{-1}. 
\end{array}
\label{mfdo.2}\end{equation}
In particular, this tacitly assume that for each $m$, there is a unique unit 
element $e_{m}$ such that $g\circ g^{-1}=e_{m}$ whenever $r(g)=m$.
We denote the space of units by $\cG^{(0)}$ and remark that it is 
canonically identified with the space of objects $M$.  
\begin{definition}
A \textbf{Lie groupoid} is a groupoid $\cG$ such that $\cG^{(1)}$ and 
$\cG^{(0)}$ are smooth manifolds with corners and such that the structural
maps $d,r,\inv$ and $\mu$ are smooth with $d$ and $r$ submersions (in the 
sense of Definition 1 in \cite{NWX}).  With this definition, the fibres
of $d$ and $r$ are smooth manifolds without boundary or corner.
\label{mfdo.3}\end{definition}

The \textbf{$d$-vertical tangent space} of a Lie groupoid $\cG$, denoted
$T_{\vert}\cG$, is the vector bundle defined by the kernel of the differential
$d_{*}:T\cG^{(1)}\to T\cG^{(0)}$.  The \textbf{Lie algebroid of $\cG$},
denoted $A(\cG)$, is defined to be the restriction of $T_{\vert}\cG$ to
$\cG^{(0)}$.  Its space of sections is naturally identified with the 
space of sections of $T_{\vert}\cG$ which are right invariant with respect
to the action of $\cG$ and its Lie algebra structure is identified with
the corresponding one for right invariant sections of $T_{\vert}\cG$.  The
anchor map of $A(\cG)$ is induced by the differential of the range map
$r$:
\begin{equation}
  r_{*}: A(\cG)\to T\cG^{(0)}= TM.
\label{mfdo.4}\end{equation}

In \cite{NWX}, a calculus of pseudodifferential operators was defined on such
Lie groupoids providing a unifying point of view for the description of 
pseudodifferential operators arising in various contexts.  In particular, as
described in \cite{ALN07}, this gives a systematic way of constructing an
algebra of pseudodifferential operators for Riemannian manifolds with Lie 
structure at infinity.  One of the key ingredients for this approach is an 
integrability criterion of Debord \cite{Debord} (see also Corollary 5.9
in \cite{CrainicFernandes}).
\begin{theorem}[Debord] A Lie algebroid with injective anchor map on a 
dense open set is integrable, that is, it is the Lie algebroid of a Lie
groupoid.
\label{mfdo.5}\end{theorem}
This criterion certainly applies to the Lie algebroid of a Riemannian manifold
with Lie structure at infinity.  In particular, it applies to the Lie
algebroid ${}^{\cF}TX$.  Thus there is a Lie groupoid $\cG$ with Lie
algebroid canonically identified with ${}^{\cF}TX$.  From \cite{NWX}, we
get a corresponding algebra $\Psi^{\infty}(\cG)$ of pseudodifferential 
operators which acts on $\CI_{c}(\inte{X})$ via the \emph{vector representation} $\pi_{X}: \Psi^{\infty}(\cG) \to \End(\CI_{c}(\inte{X}))$.  This 
representation gives a corresponding algebra of pseudodifferential operators
on $\inte{X}$.  As described in \cite{ALN07}, if we can find an $\cF$-metric
$g_{\cF}$ with positive injective radius, then this algebra can also be 
described directly without referring to the groupoid.  The 
advantage of using groupoids, besides making the construction systematic, is 
that the composition property follows for free.    

However, as can be seen from previous works about 
pseudodifferential operators associated to  various types of Lie structure at infinity (see for instance \cite{Mazzeo-Melrose0}, \cite{EMM}, \cite{MelroseAPS},
\cite{MelroseGST}, 
\cite{Mazzeo-Melrose}, \cite{MazzeoEdge},\cite{Vaillant}, \cite{Lauter-Moroianu2}, \cite{Krainer07} and \cite{Grieser-Hunsicker})
the algebra of pseudodifferential constructed in \cite{NWX} and \cite{ALN07}
is usually not big enough to allow certain natural construction, for instance 
the construction of a parametrix.  Of
course, in all the papers just mentioned, there is always a groupoid hidden 
behind
the definition of the corresponding calculus of pseudodifferential operators.
The groupoid is hidden because what is usually used in not the groupoid, but a
natural compactification of it, typically a manifold with corners on which 
the pseudodifferential operators are defined in terms of conormal 
distributions.  Notice that the Lie groupoids used in \cite{NWX} and 
\cite{ALN07} are also manifolds with corners, but they are usually not compact.

\section{Microlocalization for certain types of foliated boundaries}\label{rf.0}
     
Since we are interested in analytical constructions like the construction
of a parametrix, we will construct our algebra of pseudodifferential operators
more in the spirit of \cite{Mazzeo-Melrose}.  In order to be able to 
do so, we will narrow down the type of foliations we will consider on 
the boundary $\pa X$.

To define our algebra of pseudodifferential operators on $X$, we will
make the following assumption on the foliation $\cF$.

\begin{assumption}
The foliation $\cF$ on $\pa X$ arises as follows:
\begin{itemize}
\item[(i)] $\pa X= (\pa\widetilde{X})/\Gamma$ where $\pa\widetilde{X}$ is a smooth manifold (not necessarily compact) on which a countable discrete group $\Gamma$ acts on the right by diffeomorphisms freely and
properly discontinuously.  The manifold $\pa\widetilde{X}$ is the total space of a 
fibration 
\begin{equation}
\xymatrix{
                 \widetilde{Z} \ar[r] & \pa\widetilde{X} \ar[d]^{\Phi} \\
                                              & Y
                                              }
\label{rf.1}\end{equation}
where the base $Y$ is a closed manifold and the typical fibre $\widetilde{Z}$ is 
a smooth manifold.  The group $\Gamma$ acts smoothly and locally freely\footnote{Recall that $\Gamma$ acts locally freely on $Y$ if given $\gamma\in \Gamma$ and an open set $\cU\subset Y$ such that $\gamma(y)=y$ for any 
$y\in \cU$, then $\gamma=1$.} on
$Y$ in such a way that 
\[
   \Phi(m\cdot\gamma) =  \Phi(m)\cdot\gamma,  \quad \forall \ m\in \pa\widetilde{X}, \ \gamma\in \Gamma.  
\]
\item[(ii)] The leaves of the foliation $\cF$ are given by the images of the fibres of the fibration $\Phi: \pa\widetilde{X}\to Y$ under
the quotient map $q: \pa\widetilde{X}\to\pa X$.  Thus the leaves of the foliation
are given by $q( \Phi^{-1}(y))$ for $y\in Y$.
\end{itemize} 
\label{tfb.2b}\end{assumption}
\begin{remark}
Since $\Gamma$ acts smoothly and locally freely on $Y$, notice that the subset of $Y$ where the action is free is a countable intersection of dense open sets, so in particular it is dense in $Y$ by the Baire category theorem. 
\label{lf.1}\end{remark}
\begin{remark}
The assumption that the action of $\Gamma$ on $Y$ is locally free is not necessary in many parts of the paper, notably in \S~\ref{it.0}.  This assumption is however useful to give a simple description of the holonomy groupoid of the foliation and is very helpful in defining the normal operator (see Lemma~\ref{faith.1} and Definition~\ref{symb.19}).\label{lf.1b}\end{remark}

\begin{example}
Here is a simple way to construct such a foliation.  
Let $\widetilde{W}$ denote the universal cover of  a smooth closed manifold $W$.  Then $\Gamma=\pi_{1}(W)$ acts freely and properly discontinuously on $\widetilde{W}$.  Suppose also that $\Gamma$ acts smoothly and locally freely on another closed manifold $V$.  Then the diagonal action of $\Gamma$ on the total space of the trivial fibration $\pr_{R}: \widetilde{W}\times V\to V$ satisfies all the requirements in Assumption~\ref{tfb.2b}, so that there is an induced foliation on the quotient space $(\widetilde{W}\times V)/ \Gamma$. 
\label{uc.1}\end{example}

An important special case is given by the following.  

\begin{example}[Kronecker foliation]  For $\theta$ an irrational number, the Kronecker foliation on the torus $\bbT^{2}= \bbR^{2}/\bbZ^{2}$ is the foliation whose leaves are obtained by projecting the straight lines of slope $\theta$ in $\bbR^{2}$ onto the torus.  Each leaf is then diffeomorphic to the real line and  is dense in the torus.  The Kronecker foliation  is an example of foliation satisfying Assumption~\ref{tfb.2b}.  Indeed, consider the fibration $\Phi: \bbR\times (\bbR/\bbZ)\to \bbR/\bbZ$ given by the projection on the right factor.  Let $\Gamma= \bbZ$ act on the total space by
\begin{equation}
 (x,[y])\cdot k= (x+k, [y-\theta k]), \quad (x,[y])\in \bbR\times (\bbR/\bbZ), \; k\in\bbZ,
\label{kr.2}\end{equation}
where the brackets denote equivalence classes modulo the action of $\bbZ$.
It induces an action of $\bbZ$ on $\bbR/\bbZ$, 
\begin{equation}
 [y]\cdot k= [y-\theta k], \quad [y]\in\bbR/\bbZ, \; k\in\bbZ.
\label{kr.3}\end{equation}
This fibration and group actions satisfy part (i) of Assumption~\ref{tfb.2b}, so that there is an induced foliation on the quotient
\[
        (\bbR\times (\bbR/\bbZ))/\bbZ.
\]
One can then easily check that the diffeomorphism
\begin{equation}
\begin{array}{lccc}
  \psi: & (\bbR\times (\bbR/\bbZ))/\bbZ & \to & \bbT^{2}= \bbR/\bbZ\times \bbR/\bbZ \\
          &   [x,[y]] & \mapsto & ([x], [\theta x +y])
\end{array}
\label{kr.4}\end{equation}
identifies the foliation on $(\bbR\times (\bbR/\bbZ))/\bbZ$ with the Kronecker foliation on $\bbT^{2}$.
\label{kr.1}\end{example}

\begin{example} [Seifert fibrations]
Let $\cF$ be a foliation of a closed smooth 3-manifold $M$ by circles.  By a result of Epstein \cite{Epstein1972}, this foliation is then diffeomorphic to a Seifert fibration.  Thus, its leaves are the fibres of a circle bundle $\hat{N}$ over some compact orbifold surface without boundary $\hat{\Sigma}$ (the space of leaves).  If this orbifold is bad, which means its universal cover is not a smooth manifold,  then the foliation $\cF$ cannot come from the quotient of a fibration.  As described in \cite{Thurston}, this can only happen if $\hat{\Sigma}$ is the teardrop (a $2$-sphere with one cone point) or a $2$-sphere with two cone points having different cone angles.  Otherwise, $\hat{\Sigma}$ is a good orbifold, which means its universal cover is a smooth manifold.  In that case, one can show (see for instance Theorem~2.5 in \cite{Scott1983}) that $\hat{\Sigma}$ can be covered by a smooth closed surface $\Sigma$.  If $N\to \Sigma$ is the pull-back of $\hat{N}$ to $\Sigma$ and $\Gamma$ is the group of deck transformations of the cover $\Sigma\to \hat{\Sigma}$, then $N\to \Sigma$ is naturally a $\Gamma$-equivariant circle bundle satisfying all the hypotheses in Assumption~\ref{tfb.2b}.  The foliation we obtain by passing to the quotient is then precisely $\cF$.      
\label{sm.1}\end{example}

\begin{example}
As special case of the previous example, consider $\bbR^3\times \bbS^1$ with its natural flat metric and let $\bbZ_k$ act on $\bbR^3\times \bbS^1$ by letting its generator act by rotation by an angle $\frac{2\pi}{k}$ around the $z$-axis in $\bbR^3$ and by rotation by an angle $\frac{2\pi}{k}$ on $\bbS_1$.  This action preserves the metric, so we get a corresponding metric $g_{\cF}$ on the quotient space $(\bbR^3\times \bbS^1)/\bbZ_k$.  If we think of $\bbR^3$ as the interior of the unit ball $\bbB^3$, then $g_{\cF}$ can be thought as an $\cF$-metric associated to the manifold with foliated boundary $X= (\bbB^3\times \bbS^1)/\bbZ_k$.  This example generalizes to multi-Taub-NUT metrics of type $A_{k-1}$ admitting an action of $\bbZ_k$ by isometries, see \cite{Suvaina2011} and \cite{Wright2011}.  In this case, the circle fibration at infinity is replaced by a circle foliation when one passes to the quotient.  More details are provided in \S~\ref{qmtn.0} below.       
\label{multiT.1}\end{example}

\begin{example}
As in Example~\ref{sm.1}, take $\hat{N}\to \hat{\Sigma}$ to be a Seifert fibration over a good compact orbifold surface without boundary $\hat{\Sigma}$.  Since $\bbS^{1}\subset \SU(2)$,  we can enlarge the structure group to form a principal $\SU(2)$-bundle $\hat{P}$ over $\hat{\Sigma}$.  The fibres of $\hat{P}$ then define a foliation $\cF$ on the total space of this bundle.  Except for finitely many leaves diffeomorphic to lens spaces, the leaves of $\cF$ are diffeomorphic to $\bbS^{3}$.  If $\Sigma$ is a smooth closed surface covering $\hat{\Sigma}$ and $\Gamma$ is the group of deck transformations, then the pull-back $P$ of $\hat{P}$ to $\Sigma$ is a $\Gamma$-equivariant $\SU(2)$-bundle satisfying the hypotheses of Assumption~\ref{tfb.2b} and the foliation obtained by passing to the quotient is precisely $\cF$.  This construction still works if instead of $\SU(2)$ we have more generally a smooth compact manifold admitting a free $\bbS^{1}$-action.    
\label{sm.2}\end{example}

We refer to  \S2.1 of \cite{Benameur-Piazza} for further examples.  
For a foliation $\cF$ arising as in Assumption~\ref{tfb.2b}, the holonomy groupoid admits a simple construction.  It is given by 
\[
      \cG= \pa\widetilde{X}\times_{\Phi}\pa\widetilde{X}/\Gamma
\]
with space of units given by $\cG^{(0)}= \pa\widetilde{X}/\Gamma= \pa X$ and with
domain and range maps given by
\[
   d[(\widetilde{m}, \widetilde{m}')] = [\widetilde{m}'], \quad 
    r[(\widetilde{m}, \widetilde{m}')] = [\widetilde{m}]
\]
where the brackets denote equivalence classes modulo the action of the group $\Gamma$.  If $\Gamma_{y}$ is the isotropy group of $y\in Y$, then there is a canonical identification of the leaf $L_{y}= q(\Phi^{-1}(y))$ with the quotient 
$\widetilde{Z}_{y}/\Gamma_{y}$, where $Z_{y}= \Phi^{-1}(y)$.  Moreover, the holonomy cover of $L_{y}$ is then given by the quotient map
\[
      \widetilde{Z}_{y}\to  \widetilde{Z}_{y}/\Gamma_{y}= L_{y}.
\]

The quotient map $q: \pa\widetilde{X}\to \pa X$ should be understood as 
a `resolution' of the foliation $\cF$ into a fibration.  This resolution will
allow us to describe our algebra of pseudodifferential operators in terms
of the fibred cusp operators of Mazzeo and Melrose \cite{Mazzeo-Melrose}.

Let $c: \pa X\times [0,\epsilon)_{x}\hookrightarrow X$ be a collar neighborhood of $\pa X$ compatible with the boundary defining function
$x$, that is, such that $c(\pa X\times \{r\})= x^{-1}(r)$ for all $r\in [0,\epsilon)$.  The quotient map $q:\pa\widetilde{X}\to \pa X$ extends to a map
\begin{equation}
  q_{c}= q\times \Id_{[0,\epsilon)_{x}}: \pa\widetilde{X}\times
 [0,\epsilon)_{x}\to \pa X\times [0,\epsilon)_{x}.
\label{tfb.3}\end{equation}
Then $\widetilde{M}= \pa\widetilde{X}\times [0,\epsilon)_{x}$ is a non-compact 
manifold with boundary, the boundary $\pa \widetilde{M}= \pa\widetilde{X}\times \{0\}$ being also possibly non-compact.  The boundary is equipped with
a fibration structure
\begin{equation}
  \Phi: \pa \widetilde{M}= \pa\widetilde{X}\to Y.
\label{tfb.4}\end{equation}
Although $\pa \widetilde{M}$ and $\widetilde{M}$ are possibly non-compact, we can proceed as in \cite{Mazzeo-Melrose} to 
define the corresponding $\Phi$-double space.  One first blows up the corner
of $\widetilde{M}\times \widetilde{M}$ to obtain the $b$-double space
\begin{equation}
  \widetilde{M}^{2}_{b}= [\widetilde{M}\times \widetilde{M}; \pa \widetilde{M}\times \pa \widetilde{M}]
\label{tfb.5}\end{equation}
with blow-down map $\beta_{b}: \widetilde{M}^{2}_{b}\to \widetilde{M}^{2}$.  
If $x$ and $x'$ are the boundary defining functions for the left
and right factors in $\widetilde{M}\times \widetilde{M}$, then this blow-up amounts
to the introduction of polar coordinates
\begin{equation}
r=\sqrt{ x^{2}+ (x')^{2}}, \; \omega= \frac{x}{r}, \; \omega'= \frac{x'}{r},
\label{tfb.5}\end{equation}
where  $r$ is the boundary defining function of the `new' hypersurface
\begin{equation}
         B= \beta_{b}^{-1}(\pa \widetilde{M}\times \pa \widetilde{M}) \subset \widetilde{M}_{b}^{2},
\label{tfb.6}\end{equation}
while $\omega$ and $\omega'$ are boundary defining functions of the 
`old' hypersurfaces.  Notice that the `new' hypersurface 
\begin{equation}
   B= \SN^{+}(\pa \widetilde{M}\times \pa\widetilde{M})
\label{tfb.7}\end{equation}
is by definition a quarter circle bundle over $\pa\widetilde{M}\times \pa\widetilde{M}$,
therefore giving the natural decomposition
\begin{equation}
 B= \pa\widetilde{M}\times \pa \widetilde{M}\times [-1,1]_{s}, \quad 
   s=\omega-\omega'.
\label{tfb.8}\end{equation}
From this decomposition, we can define the lift of the fibre diagonal
\begin{equation}
  D_{\Phi}= \{ (h,h')\in \pa\widetilde{M}\times \pa\widetilde{M}\; | \; 
     \Phi(h)=\Phi(h') \}
\label{tfb.9}\end{equation}
to the hypersurface $B$ by
\begin{equation}
  \tilde{D}_{\Phi}= \{ (h,h',0)\in \pa \widetilde{M}\times \pa\widetilde{M} \times [-1,1]_{s}
  \; | \; \Phi(h)=\Phi(h') \}.
\label{tfb.10}\end{equation}
The $\Phi$-double space of $\widetilde{M}$ can then be defined by
\begin{equation}
 \widetilde{M}_{\Phi}^{2}= [\widetilde{M}_{b}^{2}; \tilde{D}_{\Phi}]=
     [\widetilde{M}^{2}; \pa\widetilde{M}\times \pa\widetilde{M}; \tilde{D}_{\Phi} ] 
\label{tfb.11}\end{equation}
with blow-down map  $\beta_{\Phi-b}: \widetilde{M}^{2}_{\Phi}
     \to \widetilde{M}^{2}_{b}$
and total blow-down map $\beta_{\Phi}= \beta_{b}\circ \beta_{\Phi-b}$.  Let 
\begin{equation}
\ff_{\Phi}= \overline{ \beta_{\Phi-b}^{-1}(\tilde{D}_{\Phi})}
\label{tfb.12}\end{equation}
be the `new' face created by this blow-up.  It is called the \textbf{front face}
of $\widetilde{M}^{2}_{\Phi}$.  Let also $\widetilde{\Delta}_{\Phi}= \overline{\beta^{-1}_{\Phi}(\Delta_{\widetilde{M}}\setminus \Delta_{\pa \widetilde{M}})}$ be 
 the lift of the diagonal of $\widetilde{M}^{2}$ to the $\Phi$-double space
$\widetilde{M}^{2}_{\Phi}$.    Fibred cusp operators on $\widetilde{M}$ can be defined
as distributions on $\widetilde{M}_{\Phi}^{2}$ which are conormal to the lifted
diagonal $\Delta_{\Phi}$ and decay rapidly at each boundary face except
maybe at $\ff_{\Phi}$.  It is worth pointing out that there is a natural underlying Lie groupoid $\cG_{\Phi}$ given by 
\begin{equation}
  \cG_{\Phi}^{(1)}= \overset{\circ}{\tM^{2}_{\Phi}} \cup \overset{\circ}{\ff}_{\Phi} \subset \tM^{2}_{\Phi}, \quad 
    \cG^{(0)}_{\Phi}= \widetilde{\Delta}_{\Phi}\cong \tM,
    \label{ug.1}\end{equation}
with domain and range maps defined by $d=\pr_{R}\circ \beta_{\Phi}$ and $r=\pr_{L}\circ \beta_{\Phi}$
where $\pr_{R}: \widetilde{M}\times \tM\to \tM$ and $\pr_{L}: \tM\times \tM\to \tM$ are projections on the right and left factors respectively.  The map 
\begin{equation}
   \iota: \overset{\circ}{\tM}\times  \overset{\circ}{\tM} \ni (\widetilde{m}, \widetilde{m}') \mapsto
    (\widetilde{m}', \widetilde{m}) \in     \overset{\circ}{\tM}\times \overset{\circ}{\tM}
 \label{ug.2}\end{equation}
interchanging the two factors extends uniquely to a smooth map $\iota: \cG^{(1)}_{\Phi}\to \cG^{(1)}_{\Phi}$ defining the inverse map of $\cG_{\Phi}$.  In the same way, the composition map on the pair groupoid $\overset{\circ}{\tM}\times \overset{\circ}{\tM}$ extends uniquely to give the composition map $\mu: \cG^{(2)}_{\Phi}\to \cG^{(1)}_{\Phi}$ with 
\begin{equation}
   \cG^{(2)}_{\Phi}= \{ (\alpha,\beta)\in \cG^{(1)}_{\Phi}\quad | \quad r(\beta)=d(\alpha) \}.  
\label{ug.3}\end{equation} 

Notice that the diagonal action of $\Gamma$ on $\widetilde{M}\times \widetilde{M}$
naturally extends to an action on the $\Phi$-double space $\widetilde{M}^{2}_{\Phi}$.  Let $R: \Gamma \to \Diff(\widetilde{M}^{2}_{\Phi})$ denote this action.  
Consider the density bundle ${}^{\Phi}\Omega_{R}'= \beta^{*}_{\Phi} {}^{\Phi}\Omega_{R}$ where ${}^{\Phi}\Omega_{R}= \pr_{R}^{*}{}^{\Phi}\Omega$ is the pull-back from the projection on the right factor of the $\Phi$-density
bundle ${}^{\Phi}\Omega\to \widetilde{M}$, which is the density bundle associated
to the $\Phi$-tangent bundle ${}^{\Phi}T\widetilde{M}$.  The $\Gamma$-action on $\widetilde{M}$ naturally induces a $\Gamma$-action on the $\Phi$-density bundle ${}^{\Phi}\Omega$ making it a $\Gamma$-equivariant vector bundle over $\widetilde{M}$.  Since the map $\pr_{R}\circ \beta_{\Phi}$ is equivariant with respect to the $\Gamma$-actions on $\widetilde{M}^{2}_{\Phi}$ and
$\widetilde{M}$, we see that the $\Gamma$-action on ${}^{\Phi}\Omega$ lifts to a $\Gamma$-action giving ${}^{\Phi}\Omega_{R}'$ the structure of a $\Gamma$-equivariant vector bundle over $\widetilde{M}_{\Phi}^{2}$.  The action of $\Gamma$ on $\tM^{2}_{\Phi}$ restricts to give an action of $\Gamma$ on $\cG_{\Phi}$ compatible with the groupoid structure, that is to say, $\Gamma$ acts smoothly on $\cG^{(0)}_{\Phi}$, $\cG^{(1)}_{\Phi}$ and $\cG^{(2)}_{\Phi}$ in such a way that 
the structure maps $(d,r,\iota, \mu)$ are equivariant with respect to these actions.  We can therefore get a new groupoid $\cG_{\Phi,\Gamma}$ by taking the quotient,
\begin{equation}
   \cG^{(1)}_{\Phi,\Gamma}= \cG^{(1)}_{\Phi}/\Gamma, \quad \cG^{(0)}_{\Phi,\Gamma}= \cG^{(0)}_{\Phi}/\Gamma\cong M.
\label{ug.4}\end{equation}

\begin{definition}
The space $\Psi^{m}_{\Phi,\Gamma}(\widetilde{M})$ of $\Gamma$-invariant fibred cusp pseudodifferential operators of order $m$  on $\widetilde{M}$ 
is defined to be the space of conormal distributions 
$K\in I^{m}(\widetilde{M}^{2}_{\Phi}, \Delta_{\Phi}; {}^{\Phi}\Omega_{R}')$ 
such that 
\begin{itemize}
\item[(i)] $K\equiv 0$ at $\pa\widetilde{M}^{2}_{\Phi}\setminus \ff_{\Phi}$, 
that is, $K$ vanishes to infinite order at all hypersurfaces of $\pa M^{2}_{\Phi}$
except possibly at the front face $\ff_{\Phi}$;
\item[(ii)] $K$ is $\Gamma$-invariant with respect to the diagonal action
of $\Gamma$ on $\widetilde{M}^{2}_{\Phi}$,
\[
       R(\gamma)_{*}K= K \; \forall \ \gamma\in \Gamma;
\]  
\item[(iii)] By $(ii)$ $K$ descends to define a distribution on the quotient
$\widetilde{M}_{\Phi}^{2}/\Gamma$ and we require that as an element of 
$\mathcal{C}^{-\infty}(\widetilde{M}^{2}_{\Phi}/\Gamma; {}^{\Phi}\Omega_{R}'/\Gamma)$, it has compact support.
\end{itemize}
The space $\Psi^{m}_{\Phi,\Gamma-\ph}(\widetilde{M})$ of polyhomogeneous (or classical) pseudodifferential $\cF$-operators of order $m$ is defined similarly, but using the space
 $I^{m}_{\ph}(\widetilde{M}^{2}_{\Phi}, \Delta_{\Phi}, {}^{\Phi}\Omega_{R}')$ 
of polyhomogeneous conormal distributions of order $m$.
\label{tfb.13}\end{definition}
Because of condition (iii) in Definition~\ref{tfb.13}, we see that Proposition~3 in \cite{Mazzeo-Melrose} (which says that $\Phi$-operators map $\CI(X)$  to $\CI(X)$ and $\dot{\cC}^{\infty}(X)$ to $\dot{\cC}^{\infty}(X)$)  still holds, so that an operator 
$P\in \Psi^{m}_{\Phi,\Gamma}(\widetilde{M})$ naturally gives continuous linear maps
\begin{equation}
 P: \CI(\widetilde{M})\to \CI(\widetilde{M}), \quad P: \dot{\mathcal{C}}^{\infty}(\widetilde{M})\to
 \dot{\mathcal{C}}^{\infty}(\widetilde{M}).
\label{tfb.14}\end{equation}
Condition (ii) of Definition~\ref{tfb.13} can be reformulated as saying that the action of $P$ on $\CI(\widetilde{M})$ is 
$\Gamma$-equivariant,
\begin{equation}
   R(\gamma)^{*}\circ P \circ R(\gamma^{-1})^{*}= P \; \;
     \forall \ \gamma\in \Gamma.
\label{tfb.15}\end{equation} 
This $\Gamma$-equivariance allows us to define an action of $P$ 
on the smooth functions defined on the quotient $M=\widetilde{M}/\Gamma=
\pa X\times [0,\epsilon)_{x}$.
Indeed, given $f\in \CI(M)$, let $\widetilde{f}= q_{c}^{*}f$ be its pull-back to
$\widetilde{M}$.  Clearly, a function on $\widetilde{M}$ can be written as a pull-back of a function on $M$ if and only if it is $\Gamma$-invariant.  Thus, $\widetilde{f}$ is 
$\Gamma$-invariant and by the $\Gamma$-invariance of $P$, we have
\begin{equation}
     R(\gamma)^{*}(P\widetilde{f})= P(R(\gamma)^{*}\widetilde{f})= P\widetilde{f}, \;
     \forall \ \gamma\in \Gamma.
\label{tfb.16}\end{equation}
This means there exists a unique function $g\in \CI(M)$ such that $P\widetilde{f}=
q_{c}^{*}g$.  We define the action of $P$ on $f$ by $Pf=g$.  Thus,
$Pf\in \CI(M)$ is the unique function such that  $P(q_{c}^{*}f)= q_{c}^{*}Pf$.
This defines continuous linear maps
\begin{equation}
  R_{q_{c}}(P): \CI(M)\to \CI(M), \quad R_{q_{c}}(P): \dot{\mathcal{C}}^{\infty}(M)
    \to \dot{\mathcal{C}}^{\infty}(M).
\label{tfb.17}\end{equation}
In fact, because of condition (iii) in Definition~\ref{tfb.13}, we get more
precisely maps of the form
\begin{equation}
  R_{q_{c}}(P): \CI(M)\to \CI_{c}(M), \quad R_{q_{c}}(P): \dot{\mathcal{C}}^{\infty}(M)
    \to \dot{\mathcal{C}}^{\infty}_{c}(M).
\label{tfb.17}\end{equation}
\begin{definition}
We define the space $\Psi^{m}_{\cF}(M)$ of foliated cusp pseudodifferential 
operators of order $m$ (or $\cF$-operators) on $M$ to be the image of $\Psi^{m}_{\Phi,\Gamma}(\widetilde{M})$ under the representation $R_{q_{c}}:\Psi^{m}_{\Phi,\Gamma}(\widetilde{M})\to
\End(\CI(M))$.  The space $\Psi^{m}_{\cF-\ph}(M)$ of polyhomogeneous $\cF$-operators of order $m$ is defined similarly.   
\label{tfb.17b}\end{definition}

Given $P\in \Psi^{m}_{\cF}(M)$, we can make $P$ act on smooth functions on
$X$ using the collar neighborhood $c: M\hookrightarrow X$,  
\begin{equation}
c_{*}\circ P\circ c^{*} :\CI(X)\to \CI(X), \quad 
c_{*}\circ P \circ c^{*}: \dot{\mathcal{C}}^{\infty}(X)\to  \dot{\mathcal{C}}^{\infty}(X),
\label{tfb.18}\end{equation} 
where $c^{*}: \CI(X)\to \CI(M)$ is the pull-back while $c_{*}:\CI_{c}(M)\to \CI(X)$ is the pushforward.  This defines a map 
\begin{equation}
  \pi_{c}: \Psi^{m}_{\cF}(M)\to \End(\CI(X)).
\label{tfb.19}\end{equation}
On $X$, we can also consider the algebra of pseudodifferential operators
with Schwartz kernels vanishing with all their derivatives on $\pa X^{2}_{b}$,
\begin{equation}
   \dot\Psi^{m}(X)= \{ K \in I^{m}(X^{2}_{b},\Delta_{b};\Omega_{R}) \;  | \; K\equiv 0 \; \mbox{on} \; \pa X^{2}_{b}\},
\label{tfb.20}\end{equation}
where $\Delta_{b}$ is the lift of the diagonal from $X\times X$ to $X^{2}_{b}$.
\begin{definition}
The space of \textbf{foliated cusp pseudodifferential operators}
(or $\cF$-pseudodifferential operators) of order $m$ on $X$ is
\[
     \Psi^{m}_{\cF}(X)= \pi_{c}(\Psi^{m}_{\cF}(M))+ 
\dot{\Psi}^{m}(X).
\]
The space of \textbf{polyhomogeneous} foliated cusp pseudodifferential
operators of order $m$, $\Psi^{m}_{\cF-\ph}(X)$, is defined similarly.  More 
generally, if $E$ and $F$ are complex vector bundles on $X$, we can define
the corresponding spaces $\Psi^{m}_{\cF}(X;E,F)$ and $\Psi^{m}_{\cF-\ph}(X;E,F)$ of $\cF$-operators mapping sections of $E$ to sections of $F$.  
\label{tfb.21}\end{definition}
\begin{remark}
Because of condition (i) in Definition~\ref{tfb.13}, notice that 
\[
    P\in \Psi^{m}_{\cF}(X;E,F) \; \Longrightarrow \; x^{\ell}\circ P \circ x^{-\ell}\in
    \Psi^{m}_{\cF}(X;E,F)
\]
since conjugation by $x^{\ell}$ corresponds to multiplication of  the Schwartz kernel of $P$  by the function $(\frac{x}{x'})^{\ell}$.
\label{conjugation.1}\end{remark}
The underlying Lie groupoid $\cG_{\cF}$ associated to $\cF$-operators is obtained by gluing the Lie groupoid $\cG_{\Phi,\Gamma}$ of \eqref{ug.4} with the pair groupoid $\overset{\circ}{X}\times
\overset{\circ}{X}$ using the quotient map with respect to the action of $\Gamma$ on the left factor
\begin{equation}
 \begin{array}{llcl} 
     q_{L} : & (\overset{\circ}{\tM}\times \overset{\circ}{\tM})/\Gamma & \to & \overset{\circ}{M}\times \overset{\circ}{M} \\
       & [(\widetilde{m},\widetilde{m})]  &\mapsto  & ([\widetilde{m}], [\widetilde{m}'])
 \end{array}
\label{ug.5}\end{equation}   
where the brackets denote equivalence classes modulo the action of the group $\Gamma$.  Thus, $\cG_{\cF}$ is given by
\begin{equation}
    \cG^{(1)}_{\cF}= \cG^{(1)}_{\Phi,\Gamma} \cup_{q_{L}} ( \overset{\circ}{X}\times \overset{\circ}{X}), \quad \cG^{(0)}_{\cF}=X.
\label{ug.6}\end{equation}
One peculiar feature of this groupoid is that $\cG^{(1)}_{\cF}$ is not Hausdorff. 

To show that  $\cF$-operators compose nicely, we need some preparation.

\begin{definition}
Let $\Gamma$ be a discrete group acting smoothly on a manifold $\widetilde{W}$
in such a way that the quotient $W= \widetilde{W}/\Gamma$ is a compact manifold.  A \textbf{partition of unity relative to} $\Gamma$ is then a smooth function 
$\varphi\in \CI_{c}(\widetilde{W})$ such that 
\[
           \sum_{\gamma\in \Gamma} \gamma^{*}\varphi \equiv 1.  
\]
Since $\varphi$ has compact support, notice that for all $w\in W$
the sum $\sum_{\gamma\in \Gamma} \gamma^{*}\varphi(w)$ is finite.
\label{lift.1}\end{definition}
As explained in p.53 of \cite{Atiyah}, a partition of unity relative to
$\Gamma$ is easily constructed.  Indeed, let $\cU_{i}$ be a finite
open covering of $W$ such that $\widetilde{W}\to W$ has a smooth section 
$s_{i}$ over $\cU_{i}$.  Let $\phi_{i}$ be a partition of unity on $W$ with
$\supp(\phi_{i})\subset \cU_{i}$.  Using the section $s_{i}$, we can lift the 
function $\phi_{i}$ to a function $\widetilde{\phi}_{i}\in\CI_{c}(\widetilde{W})$ such
that $\supp(\widetilde{\phi}_{i})\subset s_{i}(\cU_{i})$ and $s^{*}_{i}\widetilde{\phi}_{i}=\phi_{i}$.  Then the function $\varphi= \sum_{i}\widetilde{\phi}_{i}$ is a partition
of unity relative to $\Gamma$. 

Notice that a partition of unity relative to $\Gamma$ on $\pa\widetilde{M}$ can
be used to give an alternative description of the action of a $\Gamma$-invariant operator $P\in \Psi^{m}_{\Phi,\Gamma}(\widetilde{M})$ on $\CI(M)$.  Indeed,
if $\pa \varphi\in \CI(\pa \widetilde{M})$ is such a partition of unity and 
$\varphi\in \CI(\widetilde{M})$ is its pull-back to $\widetilde{M}$, then
$P\in \Psi^{m}_{\Phi,\Gamma}(\widetilde{M})$ acts on $f\in \CI(M)$ 
by 
\[
          Pf= (q_{c})_{*}P (\varphi q_{c}^{*}f)
\]
where for $g\in\CI_{c}(\tM)$, $(q_{c})_{*}g(m)= \sum_{\gamma\in \Gamma} g(\widetilde{m}\cdot\gamma)$ with $\widetilde{m}\in \widetilde{M}$ chosen such that $q_{c}(\widetilde{m})=m$.

On $M$ we can consider the analog of \eqref{tfb.20}, namely
\begin{equation}
\dot{\Psi}^{m}(M)= \{K\in I^{m}(M^{2}_{b},\Delta_{b};\Omega_{R})\; | \;
 K\equiv 0 \; \mbox{on}\; \pa M^{2}_{b}, \; \supp(K) \subset\subset 
 M^{2}_{b}\}.
\label{lift.2}\end{equation}
Similarly, we can define $\dot{\Psi}^{m}_{\cF}(M)$ to be the space of 
operators $Q\in \Psi^{m}_{\cF}(M)$ that can be represented by  Schwartz 
kernels on $\widetilde{M}^{2}_{\Phi}$ vanishing with all their derivatives on 
$\pa\widetilde{M}^{2}_{\Phi}$.  

\begin{lemma}
We have the identification $\dot{\Psi}^{m}(M;E,F)= \dot{\Psi}^{m}_{\cF}(M;E,F)$.\label{lift.3}\end{lemma}
\begin{proof}
Without loss of generality, we can assume $E=F= \underline{\bbC}$.  Since
the inclusion $\dot{\Psi}^{m}_{\cF}(M)\subset \dot{\Psi}^{m}(M)$ is obvious, what
is left to show is that $\dot{\Psi}^{m}(M)\subset\dot{\Psi}_{\cF}^{m}(M)$.  Let
$Q\in \dot{\Psi}^{m}(M)$ be given.   To establish that $Q\in \dot{\Psi}^{m}_{\cF}(M)$, we need to show that its Schwartz kernel can be lifted to a $\Gamma$-invariant distribution in $I^{m}(\widetilde{M}^{2}_{\Phi}, \Delta_{\Phi}, {}^{\Phi}\Omega_{R}')$.  

Let $\cU_{i}$ be a finite open covering of $\pa M$ such that $\pa\widetilde{M}\to \pa M$
has a smooth section $\pa s_{i}$ over $\cU_{i}$.  Let $V_{i}=\cU_{i}\times [0,\epsilon)_{x}$ be the corresponding open covering of $M$ with sections
$s_{i}$ of $\widetilde{M}\to M$ over $V_{i}$.  Decompose the operator $Q$ in 
such a way that 
\begin{equation}
  Q= Q_{1} + Q_{2}, \quad Q_{1}\in \dot{\Psi}^{m}(M), \; Q_{2}\in 
\dot{\Psi}^{-\infty}(M),
\label{lift.4}\end{equation}
with $Q_{1}$ having its Schwartz kernel $K_{Q_{1}}$ supported near the 
diagonal:
\begin{equation}
  \supp K_{Q_{1}}\subset \subset  \bigcup_{i} V_{i}\times V_{i}.
\label{lift.5}\end{equation}
Let $\phi_{i}$ be a smooth partition of unity of $\bigcup_{i} V_{i}\times V_{i}$ with $\supp \phi_{i} \subset V_{i}\times V_{i}$ so that 
\begin{equation}
  K_{Q_{1}}= \sum_{i} \phi_{i} K_{Q_{1}}.
\label{lift.6}\end{equation}
Using the section $s_{i}\times s_{i}:V_{i}\times V_{i}\to \widetilde{M}\times \widetilde{M}$, we can lift $\phi_{i}K_{Q_{1}}$ to a compactly supported Schwartz kernel
$K_{i}$ in $\widetilde{M}\times \widetilde{M}$ and define a corresponding $\Gamma$-invariant Schwartz kernel
\[
    K_{i}^{\Gamma}= \sum_{\gamma\in \Gamma} (\gamma\times \gamma)^{*}K_{i}
\]
under the diagonal action.  Summing over $i$, we get a Schwartz kernel 
\[
   K_{\widetilde{Q}_{1}}= \sum_{i} K^{\Gamma}_{i}
\]
defining a $\Gamma$-invariant  operator $\widetilde{Q}_{1}\in \dot{\Psi}^{m}_{\Gamma}(\widetilde{M})=
\dot{\Psi}^{m}_{\Phi,\Gamma}(\widetilde{M})$.  By construction, $\widetilde{Q}_{1}$ is
such that 
\[
       \widetilde{Q}_{1} q_{c}^{*}f= q_{c}^{*}(Q_{1} f), \quad \forall \ f\in \CI(M).
\]
This means $Q_{1}\in \dot{\Psi}^{m}_{\cF}(M)$.

To find a lift for the operator $Q_{2}\in \dot{\Psi}^{-\infty}(M)$, notice
that there is a sequence of two quotient maps
\begin{equation}
\xymatrix{
   \widetilde{M}\times\widetilde{M} \ar[r]^(.45){q_{D}} & \widetilde{M}\times \widetilde{M}/\Gamma 
\ar[r]^(.55){q_{L}} & M\times M
}
\label{lift.7}\end{equation} 
where $q_{D}$ is the quotient map with respect to the diagonal action of 
$\Gamma $ on $\widetilde{M}\times \widetilde{M}$ and $q_{L}$ is the quotient map
with respect to the action of $\Gamma$ on the left factor.  If 
$\varphi\in \CI(\widetilde{M}\times\widetilde{M}/\Gamma)$ is a choice of partition
of unity relative to $\Gamma$ for the quotient map $q_{L}$, then 
$\varphi q_{L}^{*}K_{Q_{2}}$ is a lift of $K_{Q_{2}}$ to $\widetilde{M}\times\widetilde{M}/\Gamma$.  Since $\varphi q_{L}^{*}K_{Q_{2}}$ vanishes with all its derivatives
at the boundary of $\widetilde{M}\times \widetilde{M}/\Gamma$, this can be further lifted
to a smooth Schwartz kernel on $\widetilde{M}^{2}_{\Phi}/\Gamma$ vanishing with
all its derivatives at the boundary.  This shows that $Q_{2}\in \dot{\Psi}^{-\infty}_{\cF}(M)$. 
\end{proof}

\begin{Theorem}[composition]
The space of foliated cusp pseudodifferential operators is closed under composition by action on $\CI(X)$,
\[
    \Psi^{m}_{\cF}(X;F,G)\circ \Psi^{m'}_{\cF}(X;E,F) \subset \Psi^{m+m'}_{\cF}(X;E,G).
\]
A similar result holds for polyhomogeneous foliated cusp pseudodifferential operators.
\label{tfb.22}\end{Theorem}
\begin{proof}
Without loss of generality, we can assume $E=F=G=\underline{\bbC}$.
Clearly, $\dot{\Psi}^{m}(X)\circ \dot{\Psi}^{m'}(X)\subset
\dot{\Psi}^{m+m'}(X)$.  Because of $(iii)$ in Definition~\ref{tfb.13}, the 
composition result of \cite{Mazzeo-Melrose} applies and we get
\begin{equation}
  \Psi^{m}_{\cF}(M)\circ \Psi^{m'}_{\cF}(M) \subset \Psi^{m+m'}_{\cF}(M)    
\label{tfb.23}\end{equation}
since the $\Gamma$-invariance (condition (ii)) is easily seen to be preserved under composition.  

To complete the proof, we need to show that given $P\in \pi_{c}(\Psi^{m}_{\cF}(M))$ and $Q\in \dot{\Psi}^{m'}(X)$, we have that 
\begin{equation}
 PQ\in \Psi^{m+m'}_{\cF}(X), \quad QP\in \Psi^{m+m'}_{\cF}(X).
\label{lift.8}\end{equation}
First, let us decompose $Q$ as a sum of two operators,
\begin{equation}
  Q= Q_{1}+ Q_{2}, \quad  Q_{1}\in \dot{\Psi}^{m'}(X), \quad Q_{2}\in \dot{\Psi}^{-\infty}(X),
\label{lift.9}\end{equation}
with $Q_{1}$ having its Schwartz kernel supported near the diagonal in
$X\times X$.  By looking at the mapping properties, one concludes immediately
that 
\begin{equation}
    PQ_{2}, Q_{2}P \in \dot{\Psi}^{-\infty}(X)\subset \Psi^{m+m'}_{\cF}(X).
\label{lift.10}\end{equation}
Choosing our decomposition \eqref{lift.9} so that $Q_{1}$ has its Schwartz 
kernel supported sufficiently close to the diagonal, we can further decompose
$Q_{1}$ into a sum of two operators
\begin{equation}
   Q_{1}= Q_{1}' + Q_{1}'', \quad Q_{1}' \in \dot{\Psi}^{m'}(M), \; Q_{1}''
\in \dot{\Psi}^{m'}(X)
\label{lift.11}\end{equation}
in such a way that $K_{Q_{1}''}$ is compactly supported in 
$\inte{X}\times \inte{X}$ and such that 
\begin{equation}
 PQ_{1}''= Q_{1}''P=0.
\label{lift.12}\end{equation}
By Lemma~\ref{lift.3}, we know that $Q_{1}'\in \dot{\Psi}^{m'}_{\cF}(M)$,
so by \eqref{tfb.23}, we have that
\begin{equation}
      PQ_{1}'\in \Psi^{m+m'}_{\cF}(X), \quad Q_{1}'P\in \Psi^{m+m'}_{\cF}(X), 
\label{lift.13}\end{equation}
which completes the proof.

\end{proof}

Let $\nu_{\cF}$ be the density associated to a choice of $\cF$-metric.  If $E$ and $F$ are smooth complex vector bundles over $X$ equipped with Hermitian metrics $h_{E}$ and $h_{F}$ respectively, then we can define the formal adjoint $P^{*}: \dot{\mathcal{C}}^{\infty}(X;F)\to \mathcal{C}^{-\infty}(X;E)$ of an $\cF$-operator $P\in\Psi^{m}_{\cF}(X;E,F)$ by 
\begin{equation}
 \langle P^{*}f, e\rangle_{L^{2}} = \langle f, Pe\rangle_{L^{2}}, \quad e\in\dot{\cC}^{\infty}(X;E), \;
   f\in \dot{\cC}^{\infty}(X;F), 
\label{adj.1}\end{equation}
where the $L^{2}$-inner products are defined using the density $\nu_{\cF}$ and the Hermitian metrics $h_{E}$ and $h_{F}$:
\begin{equation}
\begin{gathered}
\langle e_{1}, e_{2}\rangle_{L^{2}}= \int_{X} h_{E}(e_{1}, e_{2}) \nu_{\cF}, \quad e_{1},e_{2}\in 
\dot{\cC}^{\infty}(X;E),   \\
\langle f_{1}, f_{2}\rangle_{L^{2}}= \int_{X} h_{F}(f_{1}, f_{2}) \nu_{\cF}, \quad f_{1}, f_{2}\in 
\dot{\cC}^{\infty}(X;F).  \end{gathered}
\label{adj.2}\end{equation}
In a similar way, taking $\nu_{\Phi}= q_{c}^{*}(\left. \nu_{\cF}\right|_{M})$ as a choice of $\Gamma$-invariant $\Phi$-density on $\widetilde{M}$ and taking $h_{\widetilde{E}}= q_{c}^{*}( \left. h_{E}\right|_{M})$, $h_{\widetilde{F}}= q_{c}^{*}( \left. h_{F}\right|_{M})$ to be our choices of $\Gamma$-invariant 
Hermitian metrics for the $\Gamma$-equivariant vector bundles $\widetilde{E}= q_{c}^{*}(\left. E\right|_{M})$ and $\widetilde{F}= q_{c}^{*}(\left. F\right|_{M})$, we can define the formal adjoint
$Q^{*}$ of a $\Gamma$-invariant $\Phi$-operator $Q\in \Psi^{m}_{\Phi,\Gamma}(\widetilde{M};\tE,\tF)$ by 
\begin{equation}
\langle Q^{*}\widetilde{f}, \widetilde{e}\rangle_{L^{2}} = \langle \widetilde{f}, Q\widetilde{e}\rangle_{L^{2}}, \quad \widetilde{e}\in\dot{\cC}_{c}^{\infty}(\tM;\tE), \;
   \widetilde{f}\in \dot{\cC}_{c}^{\infty}(\tM;\tF). 
 \label{adj.3}\end{equation}
Here, $\dot{\cC}^{\infty}_{c}(\tM;\tE)$ denotes the space of compactly supported smooth sections of $\tE$ vanishing to all order at $\pa\tM$.
By looking at the definition of $Q$ in terms of its Schwartz kernel and using the fact 
$\nu_{\Phi}$, $h_{\tE}$, $h_{\tF}$ are $\Gamma$-invariant, we see that 
$Q^{*}\in \Psi^{m}_{\Phi,\Gamma}(\widetilde{M};\tF,\tE)$.  
 
 \begin{proposition}
 The formal adjoint of an $\cF$-operator   $P\in \Psi^{m}_{\cF}(X;E,F)$ (defined as in \eqref{adj.1} with respect to a choice of $\cF$-density $\nu_{\cF}$ and Hermitian metrics on $E$ and $F$) is an element of $\Psi^{m}_{\cF}(X;F,E)$. In fact, if 
 \[
      P= \pi_{c}\circ R_{q_{c}}(P_{1}) +P_{2}
 \] 
 with $P_{1}\in \Psi^{m}_{\Phi,\Gamma}(\widetilde{M};\tE,\tF)$ and $P_{2}=\dot{\Psi}^{m}(X;E,F)$,
 then
 \[
   P^{*}= \pi_{c}\circ R_{q_{c}}(P_{1}^{*})+P_{2}^{*}
 \]
 where $P_{1}^{*}\in \Psi^{m}_{\Phi,\Gamma}(\tM;\tF,\tE)$ is defined as in \eqref{adj.3} and 
 $P_{2}^{*}\in \dot{\Psi}^{m}(X;F,E)$.  
 \label{adj.4}\end{proposition}
\begin{proof}
Clearly, the formal adjoint $P_{2}^{*}$ of an operator $P_{2}\in \dot{\Psi}^{m}(X;E,F)$ is an element of $\dot{\Psi}^{m}(X;F,E)$. Thus, to prove the proposition, we may assume that $P$ is of the form
\[
     P= \pi_{c}\circ R_{q_{c}}(P_{1})
\]  
with $P_{1}\in \Psi^{m}_{\Phi,\Gamma}(\tM;\tE,\tF)$.  This means we can think of $P$ as an element of $\Psi_{\cF}^{m}(M;E,F)$.  If $\pa\varphi\in \CI_{c}(\pa \widetilde{M})$ is a choice of partition of unity relative to $\Gamma$ and $\varphi\in \CI(\widetilde{M})$ is its pull-back to $\widetilde{M}$, then the $L^{2}$-inner products of \eqref{adj.2} restricted to $M$ can be rewritten as follows
\begin{equation}
\begin{gathered}
\langle e_{1}, e_{2}\rangle_{L^{2}}= \int_{\tM} h_{\tE}(q_{c}^{*}e_{1}, q_{c}^{*}e_{2}) \varphi\nu_{\Phi}, \quad e_{1},e_{2}\in 
\dot{\cC}^{\infty}(M;E),   \\
\langle f_{1}, f_{2}\rangle_{L^{2}}= \int_{\tM} h_{\tF}(q^{*}_{c}f_{1}, q^{*}_{c}f_{2}) \varphi\nu_{\Phi}, \quad f_{1}, f_{2}\in 
\dot{\cC}^{\infty}(M;F).  \end{gathered}
\label{adj.5}\end{equation}
Since $P_{1}$ is $\Gamma$-invariant and $\varphi$ is a partition of unity relative to $\Gamma$, 
we have for $e\in \dot{\cC}^{\infty}(M;E)$, $f\in \dot{\cC}^{\infty}(M;F)$ with $\widetilde{e}= q^{*}_{c}e$ and $\widetilde{f}=q^{*}_{c}f$, 
\begin{equation}
\begin{aligned}
\langle f, P e\rangle_{L^{2}} &=   \int_{\tM} h_{\tF}(\widetilde{f}, P_{1} \widetilde{e}) \varphi \nu_{\Phi}  = \int_{\tM} h_{\tF}(\widetilde{f}, \varphi P_{1}\widetilde{e})\nu_{\Phi}  \\
      &=  \sum_{\gamma\in\Gamma} \int_{\tM} h_{\tF}(\widetilde{f}, \varphi P_{1}(\gamma^{*}\varphi)\widetilde{e})\nu_{\Phi} =   \sum_{\gamma\in\Gamma} \int_{\tM} h_{\tF}(\widetilde{f}, (\gamma^{*}\varphi) P_{1}\varphi\widetilde{e})\nu_{\Phi}   \\
          &=  \int_{\tM} h_{\tF}(\widetilde{f}, P_{1}\varphi\widetilde{e})\nu_{\Phi}  
      = \int_{\tM} h_{\tF}(P_{1}^{*} \widetilde{f}, \varphi\widetilde{e})\nu_{\Phi}  \\  
      &= \langle R_{q_{c}}(P_{1}^{*})f, e\rangle_{L^{2}},    
\end{aligned}
\label{adj.6}\end{equation}
that is, $P^{*}= \pi_{c}\circ R_{q_{c}}(P^{*}_{1})$.

\end{proof}

\section{Symbol maps} \label{symb.0}

Recall from \cite{Mazzeo-Melrose} that the diagonal $\widetilde{\Delta}_{\Phi} \subset \widetilde{M}^{2}_{\Phi}$ is naturally diffeomorphic to $\widetilde{M}$ with conormal bundle naturally diffeomorphic to ${}^{\Phi}T^{*}\widetilde{M}$.  Since the action of $\Gamma$ on $M$ is smooth and induces  corresponding 
actions on ${}^{\Phi}T\widetilde{M}$ and ${}^{\Phi}T^{*}\widetilde{M}$
with canonical identifications ${}^{\Phi}T\widetilde{M}/\Gamma= {}^{\cF}TM$ and
${}^{\Phi}T^{*}\widetilde{M}/\Gamma= {}^{\cF}T^{*}M$, we see that 
the diagonal $\Delta_{\Phi}= \widetilde{\Delta}_{\Phi}/\Gamma \subset
\widetilde{M}^{2}_{\Phi}/\Gamma$ is naturally diffeomorphic to 
$M$ with conormal bundle naturally diffeomorphic to
${}^{\cF}T^{*}M$.  This means we can proceed as in \cite{Mazzeo-Melrose} to define
a principal symbol map
\begin{equation}
  \sigma_{m}: \Psi^{m}_{\cF}(X) \to S^{[m]}( {}^{\cF}T^{*}X)
\label{symb.1}\end{equation}
where $S^{[m]}({}^{\cF}T^{*}X)= S^{m}({}^{\cF}T^{*}X)/ S^{m-1}({}^{\cF}T^{*}X)$ and $S^{m}({}^{\cF}T^{*}X)$ is the usual space of smooth functions 
$f\in \CI({}^{\cF}T^{*}X)$ such that in a local trivialization 
$\left.{}^{\cF}T^{*}X\right|_{\cU}\cong \cU\times \bbR^{n}_{\xi}$ with local
coordinates $u$ on $\cU$, we have
\[
    \sup_{u,\xi} \frac{ |D^{\alpha}_{u}D^{\beta}_{\xi}f |}{ (1+ |\xi|^{2})^{\frac{m-|\beta|}{2}} } <\infty
    \;\; \forall \alpha,\beta\in \bbN_{0}^{n}.
\]
For polyhomogeneous pseudodifferential operators, the principal symbol
becomes a homogeneous function so that we have the symbol map
\begin{equation}
   \sigma_{m}: \Psi^{m}_{\cF-\ph}(X)\to  \CI(S({}^{\cF}T^{*}X), \Lambda^{m}),
\label{symb.2}\end{equation} 
where $\Lambda$ is the dual of the tautological real line bundle of 
$S({}^{\cF}T^{*}X)$, the sphere bundle of ${}^{\cF}T^{*}X$.

As usual, we can also make our pseudodifferential operators act from
smooth sections of a (complex) vector bundle $E$ to smooth sections of 
another vector bundle $F$ and define the corresponding principal symbol.
This gives a short exact sequence
\begin{equation}
\xymatrix @C=1.5pc{ 0 \ar[r] & \Psi^{m-1}_{\cF}(X;E,F) \ar[r] &
\Psi^{m}_{\cF}(X;E,F) \ar[r]^-{\sigma_{m}} & S^{[m]}({}^{\cF}T^{*}X;
\phi^{*}\hom(E,F))\ar[r] & 0
}  
\label{symb.2b}\end{equation}
where $\phi: {}^{\cF}T^{*}X\to X$ is the bundle projection.

As for fibred cusp operators, there is also a `new' symbol coming from the
boundary.  To describe it, we will first study the corresponding symbol for
the space $\Psi_{\Phi,\Gamma}^{m}(\widetilde{M})$.  In this case, the \textbf{normal operator}
of an operator $P\in \Psi^{m}_{\Phi,\Gamma}(\widetilde{M};E,F)$ is defined to be the 
restriction of its Schwartz kernel $K_{P}\in I^{m}(\widetilde{M}_{\Phi}^{2},
\widetilde{\Delta}_{\Phi}; {}^{\Phi}\Omega_{R}'\otimes \Hom(E,F))$ to the 
front face $\ff_{\Phi}$,
\begin{equation}
  N_{\Phi}(P)= \left. K_{P} \right|_{\ff_{\Phi}}.
\label{symb.3}\end{equation}
\begin{remark}
Since the function $\frac{x}{x'}$ pulls back to a function equal to $1$ on $\ff_{\Phi}$, we see that 
\[
        N_{\Phi}( x^{\ell}\circ P \circ x^{-\ell}) = N_{\Phi}(P) \quad \forall \ell \in \bbR.
\]
\label{conjugation.2}\end{remark}

As described in \cite{Mazzeo-Melrose}, this can be interpreted as
the Schwartz kernel of a ${}^{\Phi}N\pa\widetilde{M}$-suspended operator, where ${}^{\Phi}N\pa\tM\to \pa \tM$ is the kernel bundle of the natural map (\cf \eqref{mfb.3}) 
\[
\rho: \left. {}^{\Phi}T\tM\right|_{\pa\tM}\to \left. T\tM\right|_{\pa\tM}.
\] To see
this, recall that the bundle ${}^{\Phi}N\pa \widetilde{M}$ can naturally be seen
as the pull-back of a bundle on $Y$ that we will conveniently denote 
$NY\to Y$.  Indeed, we have a natural decomposition 
\[
       {}^{\Phi}N\pa\widetilde{M}= T(\pa\widetilde{M})/T(\pa\widetilde{M}/Y) \times \bbR,
\]
and from the short exact sequence
\[
\xymatrix{
0 \ar[r] & T(\pa\widetilde{M}/Y) \ar[r] &T(\pa\widetilde{M}) \ar[r]^{\Phi_{*}} & \Phi^{*}TY \ar[r] & 0, 
}
\]
we have a canonical identification $\Phi^{*}TY= T(\pa\widetilde{M})/T(\pa\widetilde{M}/Y)$.  Thus, ${}^{\Phi}N\pa \widetilde{M}\cong \Phi^{*}NY$ with $NY=TY\times \bbR$.  There is a 
corresponding fibration structure 
\begin{equation}
\xymatrix{
  \widetilde{Z}\ar@{-}[r] &  {}^{\Phi}N\pa \widetilde{M} \ar[d]^{\Phi_{*}}\\
    & NY .
}  
\label{fou.1}\end{equation}
The interior of the front face $\ff_{\Phi}$ is naturally identified with
${}^{\Phi}N^{*}\pa \widetilde{M}\times_{\Phi_{*}}{}^{\Phi}N^{*}\pa \widetilde{M}$.  Under
this identification, $\left. K_{P}\right|_{\ff_{\Phi}}$ can be seen as the
Schwartz kernel of a ${}^{\Phi}N\pa \widetilde{M}$-suspended operator via the action
\begin{equation}
   N_{\Phi}(P)f= (\pr_{1})_{*}( \pi^{*}(\left. K_{P}\right|_{\ff_{\Phi}})\cdot
      \pr_{2}^{*}f), \quad f\in \CI_{c}({}^{\Phi}N\pa \widetilde{M};E)
\label{fou.2}\end{equation}
where $\pr_{i}: {}^{\Phi}N\pa \widetilde{M}\times_{Y}{}^{\Phi}N\pa \widetilde{M}\to{}^{\Phi}N\pa \widetilde{M} $ is the projection on the $i$th factor and 
\begin{equation}
  \pi: {}^{\Phi}N\pa \widetilde{M}\times_{Y} {}^{\Phi}N\pa \widetilde{M}\to 
   {}^{\Phi}N\pa \widetilde{M}\times_{\Phi_{*}} {}^{\Phi}N\pa \widetilde{M}
\label{fou.3}\end{equation}
is the projection which to $(z,z',\nu,\nu')\in \pa\widetilde{M}\times_{Y}\pa\widetilde{M}\times_{Y} NY\times_{Y} NY
\cong {}^{\Phi}N\pa \widetilde{M}\times_{Y} {}^{\Phi}N\pa \widetilde{M}$ associates
$(z,z',\nu-\nu')\in \pa\widetilde{M}\times_{Y} \pa\widetilde{M}\times_{Y} NY\cong {}^{\Phi}N\pa \widetilde{M}\times_{\Phi_{*}} {}^{\Phi}N\pa \widetilde{M}$.

The normal operator $N_{\Phi}(P)$ is therefore a family of pseudodifferential operators parametrized by the base $Y$ of the fibration 
\begin{equation}
\xymatrix{
  N_{y}\pa\widetilde{M}\ar@{-}[r] &  {}^{\Phi}N\pa \widetilde{M} \ar[d]\\
    & Y.
}   
\label{fou.4}\end{equation}
Furthermore, for a given $y\in Y$, the corresponding operator $N_{\Phi}(P)(y)$ is
translation invariant with respect to the $N_{y}Y$-action given by the 
decomposition ${}^{\Phi}N_{y}\pa \widetilde{M}= \widetilde{Z}_{y}\times N_{y}Y$
where $\widetilde{Z}_{y}= \Phi^{-1}(y)$.  The 
Schwartz kernel $K_{N_{\Phi}(P)(y)}$ is a conormal distribution in
$I^{m}(\widetilde{Z}_{y}\times\widetilde{Z}_{y}\times N_{y}Y; \Delta_{\widetilde{Z}_{y}}\times \{0\};
\Hom(E,F)\otimes \Omega_{R}(\widetilde{Z}_{y})\otimes \Omega(N_{y}Y))$ acting on
$f$ by 
\begin{equation}
  N_{\Phi}(P)(y)f(z, \nu)= \int_{\widetilde{Z}_{y}\times N_{y}Y} K_{N_{\Phi}(P)(y)}(z,z',\nu-\nu')f(z',\nu'),  
\label{fou.5}\end{equation}
where the integration is performed in the variables $z'$ and $\nu'$, the 
density factor being included in the Schwartz kernel.  Given $\eta\in
N^{*}_{y}Y$, we can therefore consider the Fourier transform in $N_{y}Y$,
\begin{equation}
  K_{\widehat{N}_{\Phi}(P)(y)}(z,z',\eta)= \int_{N_{y}Y} e^{-i\eta(\nu)}
  K_{N_{\Phi}(P)}(z,z',\nu),
\label{fou.6}\end{equation}
which is the Schwartz kernel of a family of operators $\widehat{N}_{\Phi}(P)(y)$
parametrized by the base of the fibration $ {}^{\Phi}N_{y}^{*}\pa\widetilde{M} \to N^{*}_{y}Y$.  Doing
this for each $y\in Y$, we get a family of operators $\widehat{N}_{\Phi}(P)$ associated
to the fibration 
\begin{equation}
\xymatrix{
  \widetilde{Z}\ar@{-}[r] &  {}^{\Phi}N^{*}\pa \widetilde{M} \ar[d]\\
    & N^{*}Y\cong T^{*}Y\times \bbR.
}   
\label{symb.6}\end{equation} 
The family $\widehat{N}_{\Phi}(P)$ is the Fourier transform of the normal operator $N_{\Phi}(P)$.  
For $Q_{1}\in \Psi^{m}_{\sus-\Phi}(\pa \widetilde{M};E,F)$ and
$Q_{2}\in \Psi^{m'}_{\sus-\Phi}(\pa \widetilde{M};F,G)$, it is such
that 
\begin{equation}
   \widehat{Q_{2}\circ Q_{1}}= \widehat{Q}_{2}\circ \widehat{Q}_{1}.
\label{fou.7}\end{equation}
The normal operator $N_{\Phi}(P)$ and its Fourier transform $\widehat{N}_{\Phi}(P)$ contain
the same information.

Since $\Gamma$ acts smoothly on $\widetilde{M}$ and $Y$, there are induced actions of $\Gamma$ on $NY$ and ${}^{\Phi}N\pa \widetilde{M}$.    These actions are compatible in the sense that the projection $\Phi_{*}: {}^{\Phi}N\pa\widetilde{M}\to NY$ is equivariant with respect to these actions.
Let $\psi: \Gamma \to \Diff({}^{\Phi}N\pa \widetilde{M})$
denote  the action on ${}^{\Phi}N\pa\widetilde{M}$ and let $\widehat{\psi}: \Gamma\to \Diff({}^{\Phi}N^{*}\pa\widetilde{M})$
denote the corresponding dual action on the vector bundle
${}^{\Phi}N^{*}\pa\widetilde{M}\to \pa \widetilde{M}$.  Clearly, the quotient
of ${}^{\Phi}N\pa \widetilde{M}$ by this action is ${}^{\cF}N^{}\pa X$, while the quotient of ${}^{\Phi}N^{*}\pa\widetilde{M}$ by the dual
action is ${}^{\cF}N^{*}\pa X$.

\begin{proposition}
For $P\in \Psi^{m}_{\Phi,\Gamma}(\widetilde{M};E,F)$, 
the normal operator $N_{\Phi}(P)$ and its Fourier transform $\widehat{N}_{\Phi}(P)$ are $\Gamma$-invariant with respect to the action of $\Gamma$ on ${}^{\Phi}N\pa \widetilde{M}$
and ${}^{\Phi}N^{*}\pa \widetilde{M}$,
\[
  \psi(\gamma)\circ N_{\Phi}(P)\circ \psi(\gamma^{-1})= N_{\Phi}(P),
\quad   \widehat{\psi}(\gamma)\circ \widehat{N}_{\Phi}(P)\circ \widehat{\psi}(\gamma^{-1})=
    \widehat{N}_{\Phi}(P) \quad \forall\; \gamma \in \Gamma.
\]
\label{symb.7}\end{proposition}
\begin{proof}
By condition (ii) in Definition~\ref{tfb.13}, we know that the Schwartz kernel
of $P$ is $\Gamma$-invariant with respect to the action 
of $\Gamma$ on $\widetilde{M}^{2}_{\Phi}$.  In particular, since this action
preserves the front face $\ff_{\Phi}$, we see that $N_{\Phi}(P)= \left.K_{P}
\right|_{\ff_{\Phi}}$ is also $\Gamma$-invariant with respect to the
action of $\Gamma$ on $\ff_{\Phi}$.  Under the identification
\begin{equation}
  \ff_{\Phi}\setminus\pa\ff_{\Phi}= {}^{\Phi}N\pa \widetilde{M}\times_{\Phi_{*}} {}^{\Phi}N\pa\widetilde{M},
\label{symb.8}\end{equation}
this action corresponds the action $\psi\times_{\Phi_{*}}\psi$.  Thus, $N_{\Phi}(P)$ is 
$\Gamma$-invariant.    Taking
the Fourier transform in the fibres of $NY\to Y$ then gives the
the corresponding result for $\widehat{N}_{\Phi}(P)$.
\end{proof}
Because of the $\Gamma$-invariance of $\widehat{N}_{\Phi}(P)$, we can make it act on the space
of Schwartz functions on ${}^{\cF}N^{*}\pa X$.  Indeed, let
$q_{\widehat{\psi}}: {}^{\Phi}N^{*}\pa \widetilde{M}\to {}^{\cF}N^{*}\pa X$ denote
the quotient map.  Given $f\in \cS({}^{\cF}N^{*}\pa X)$, we can consider the
$\Gamma$-invariant function $q_{\widehat{\psi}}^{*}f$ on ${}^{\Phi}N^{*}\pa\widetilde{M}$.
Since $\widehat{N}_{\Phi}(P)$ is $\Gamma$-invariant, $\widehat{N}_{\Phi}(P)(q_{\widehat{\psi}}^{*}f)$ is also $\Gamma$-invariant and we can find a unique function
$g\in \CI({}^{\cF}N^{*}\pa X)$ such that 
$\widehat{N}_{\Phi}(P)(q_{\widehat{\psi}}^{*}f)= q_{\widehat{\psi}}^{*}g$.  We define 
the action of $\widehat{N}_{\Phi}(P)$ on $f$ to be this function $g$. 
Similarly, using the quotient map $q_{\psi}:{}^{\Phi}N\pa\tM\to {}^{\cF}N\pa X$, we can make $N_{\Phi}(P)$ act on $\cS({}^{\cF}N\pa X)$.

\begin{definition}
The map 
\[r_{\psi}: N_{\Phi}(\Psi^{\infty}_{\Phi,\Gamma}(\widetilde{M};E,F))\to \Hom(\cS({}^{\cF}N\pa X;E),\cS({}^{\cF}N\pa X;F))
\]
 is defined by requiring that 
\[
     q^{*}_{\psi}( r_{\psi}(N_{\Phi}(P))f)=
N_{\Phi}(P) q^{*}_{\psi}f \quad \forall \; P\in \Psi^{\infty}_{\Phi,\Gamma}(\widetilde{M};E,F), \; \forall \; f\in \cS({}^{\cF}N\pa X;E).
\]
Similarly, the map $r_{\widehat{\psi}}: \widehat{N}_{\Phi}(\Psi^{\infty}_{\Phi,\Gamma}(\widetilde{M};E,F))\to \Hom(\cS({}^{\cF}N^{*}\pa X;E),\cS({}^{\cF}N^{*}\pa X;F))$ is defined by requiring that 
\[
     q^{*}_{\widehat{\psi}}( r_{\widehat{\psi}}(\widehat{N}_{\Phi}(P))f)=
\widehat{N}_{\Phi}(P) q^{*}_{\widehat{\psi}}f \quad \forall \; P\in \Psi^{\infty}_{\Phi,\Gamma}(\widetilde{M};E,F), \; \forall \; f\in \cS({}^{\cF}N^{*}\pa X;E).
\]  
\label{symb.10}\end{definition}
If we choose a metric for the vector bundle ${}^{\cF}N\pa X\to \pa X$, for instance the one induced by a choice of $\cF$-metric, then we have in particular a fibrewise volume density on ${}^{\cF}N\pa X$.  This can be used to define a fibrewise Fourier transform:
\begin{equation}
  \mathfrak{F}_{G}: \cS({}^{\cF}N\pa X;G)\to \cS({}^{\cF}N^{*}\pa X;G)
\label{fourier.1}\end{equation}
for $G$ a complex vector bundle over $\pa X$.  This fibrewise Fourier transform relates $r_{\psi}$ and $r_{\widehat{\psi}}$ in the expected way:
\begin{equation}
  r_{\widehat{\psi}}(\widehat{N}_{\Phi}(P))= \mathfrak{F}_{F}\circ 
  \left( r_{\psi}(N_{\Phi}(P) )  \right)\circ \mathfrak{F}_{E}^{-1}, \quad 
  P\in \Psi^{m}_{\Phi,\Gamma}(\widetilde{M};E,F).
\label{fourier.2}\end{equation}
\begin{lemma}
The maps $r_{\psi}$ and $r_{\widehat{\psi}}$ are injective.  In particular, when $E=F$, this means they induce faithful representations.  
\label{faith.1}\end{lemma}
\begin{proof}
Without loss of generality, we can assume $E=F=\underline{\bbC}$.  Given
$P\in \Psi^{m}_{\Phi,\Gamma}(\widetilde{M})$ such that $\widehat{N}_{\Phi}(P)\ne 0$, we need to show that $r_{\widehat{\psi}}(\widehat{N}_{\Phi}(P))\ne 0$ and $r_{\psi}(N_{\Phi}(P))\ne 0$.  By relation \eqref{fourier.2}, it suffices to show that $r_{\widehat{\psi}}(\widehat{N}_{\Phi}(P))\ne 0$.  Since $\widehat{N}_{\Phi}(P)\ne 0$, there exists $\tf\in \CI_{c}({}^{\Phi}N^{*}\pa \tM)$ such that $\widehat{N}_{\Phi}(P)\tf\ne 0$.

Since $\pa X$ is compact, we can find a finite covering $\{\cU_{i}\}_{i\in \cI}$ of $\pa X$ with $\cI$ a finite set, such that for each $i$, there is a 
diffeomorphism
\begin{equation}
   \phi_{i}: \cU_{i}\to B_{i}\times F_{i}
\label{symb.12}\end{equation}
in such a way that the leaves of $\phi_{i}(\left. \cF\right|_{\cU_{i}})$ are
precisely the fibres of the projection $\pr_{L}:B_{i}\times F_{i}\to B_{i}$
on the left factor, where $B_{i}\subset \bbR^{n-1-\ell}$ and 
$F_{i}\subset \bbR^{\ell}$  are open sets. Taking the $\cU_{i}$ smaller if needed,
we can also assume that the cover $q: \pa \widetilde{M}\to \pa X$ admits a section $s_{i}$ over $\cU_{i}$ inducing a diffeomorphism $s_{i}:\cU_{i}\to \tcU_{i}\subset \pa \tM$.  Then $\{\tcU_{i}\cdot \gamma\}_{i\in\cI, \gamma\in \Gamma}$ is a covering of $\pa \widetilde{M}$.  Let $\{\chi_{i,\gamma}\}$ be a partition of unity subordinate to this covering.  Clearly, we can find $i\in \cI$ and $\gamma\in \Gamma$ such that 
\[
      \widehat{N}_{\Phi}(P)\pi^{*}\chi_{i,\gamma} \tf \ne 0,
\]
where $\pi: {}^{\Phi}N^{*}\pa \tM\to \pa \tM$ is the bundle projection.  Thus, replacing $\tf$ with $\pi^{*}\chi_{i,\gamma}\tf$ if needed, we can assume $\tf$ is supported in $\pi^{*}(\tcU_{i}\cdot \gamma)$.  In fact, without loss of generality, we can assume $\gamma=1$ so that $\tf$ is supported in $\pi^{*}\tcU_{i}$.  Let $\tp\in {}^{\Phi}N^{*}\pa \tM$ be a point such that $\widehat{N}_{\Phi}(P)(\tf)(\tp)\ne 0$.  By Remark~\ref{lf.1}, we can choose the point $\tp$ in such a way that its image $y_{0}\in Y$ under the projection ${}^{\Phi}N\pa\tM\to Y$ is a point where the action of $\Gamma$ is free.  

Let $p$ be the image of $\tp$ under the quotient map $q_{\widehat{\psi}}: {}^{\Phi}N^{*}\pa\tM\to {}^{\cF}N^{*}\pa X$.  Let also $f\in \CI_{c}(\pi^{*}\cU_{i})$
be the pushforward of $\tf$ under the natural identification $\pi^{*}\tcU_{i}\cong \pi^{*}\cU_{i}$ given by the quotient map $q_{\widehat{\psi}}$, where $\pi$ also denotes the bundle projection of ${}^{\cF}N^{*}\pa X$.  Finally, let $\nu_{i}:\pi^{*}\cU_{i}\to B_{i}$ be the fibration induced by the sequence of maps
\[
\xymatrix{
   \pi^{*}\cU_{i} \ar[r] &  \cU_{i}\cong F_{i}\times B_{i} \ar[r]^(0.65){\pr_{R}} & B_{i}.
   }
\] 

Even if $\widehat{N}_{\Phi}(P)(\tf)(\tp)\ne 0$, it is still possible that $r_{\widehat{\psi}}(\widehat{N}_{\Phi}(P))(f)(p)=0$.  This is because the value of  $r_{\widehat{\psi}}(\widehat{N}_{\Phi}(P))(f)$ at $p$ is determined by the restriction of $f$ at the fibres of the fibration $\nu_{i}:\pi^{*}\cU_{i}\to B_{i}$ corresponding to the image of the leaf passing by $p$ of the foliation $\pi^{-1}\cF$ on ${}^{\cF}N^{*}\pa X$ with leaves given by the inverse images of the leaves of $\cF$ under the projection $\pi: {}^{\cF}N^{*}\pa X\to \pa X$.  Still, by the compactness assumption in (iii) of Definition~\ref{tfb.13}, we know at least that the value of $r_{\widehat{\psi}}(\widehat{N}_{\Phi}(P))(f)$ at $p$ is determined by the restriction of $f$ at only finitely many of these fibres, say 
$\nu_{i}^{-1}(b_{1}),\ldots,\nu_{i}^{-1}(b_{k})$.  Moreover, if
$\tnu_{i}: \pi^{*}\tcU_{i}\to B_{i}$ denotes the corresponding fibration on $\pi^{*}\tcU_{i}$, where we regards $B_{i}$ as an open subset of $Y$, then we also know that $\widehat{N}_{\Phi}(P)(\tf)(\tp)$ only depends on the restriction of $\tf$ to the fibre $\tnu_{i}^{-1}(y_{0})$.  Without loss of generality, we can assume that $\tnu^{-1}_{i}(y_{0})$ is mapped to $\nu_{i}^{-1}(b_{1})$ under the quotient map.  Now, using a suitable cut-off function, we can make $\tf$ supported arbitrarily close to the fibre $\tnu^{-1}_{i}(y_{0})$ in such a way that we still have
$\widehat{N}_{\Phi}(P)(\tf)(\tp)\ne 0$.  On the other hand, if the support of $\tf$ is chosen sufficiently close to $\tnu^{-1}_{i}(y_{0})$, then the function $f$ will restrict to zero on $\nu_{i}^{-1}(b_{2}),\ldots, \nu_{i}^{-1}(b_{k})$, which will insures that
\begin{equation}
    r_{\widehat{\psi}}(\widehat{N}_{\Phi}(P))(f)(p)= \widehat{N}_{\Phi}(P)(\tf)(\tp)\ne 0.
\label{faith.2}\end{equation}
In particular, this shows $r_{\widehat{\psi}}(\widehat{N}_{\Phi}(P))\ne 0$.
\end{proof}

The function $f$ in \eqref{faith.2} can also be used to show the following.  

\begin{lemma}
Given $P\in \Psi^{m}_{\Phi,\Gamma}(\widetilde{M};E,F)$, then 
\[
    R_{q_{c}}(P)= 0 \; \Longrightarrow \; \widehat{N}_{\Phi}(P)=0.
\]
\label{symb.11}\end{lemma}
\begin{proof}
Without loss of generality, we can assume $E=F=\underline{\bbC}$.  Suppose that $\widehat{N}_{\Phi}(P)\ne 0$.  We need to show 
that $R_{q_{c}}(P)\ne 0$.  To do this, start with the function $f$ of \eqref{faith.2} constructed in the proof of Lemma~\ref{faith.1} above.  Recall that this function is
supported in $\pi^{*}\cU_{i}$ where $\pi: {}^{\cF}N^{*}\pa X\to \pa X$ is the bundle projection and $\cU_{i}\subset \pa X$ is a small open set as in \eqref{symb.12}, that is, $\left. \cF\right|_{\cU_{i}}$ has its leaves given by the fibre of a trivial fibration.  Moreover, the cover $q:\pa \tM\to \pa X$ admits a section $s_{i}$ over $\cU_{i}$ inducing a diffeomorphism $s_{i}:\cU_{i} \to \tcU_{i}\subset \pa \tM$.  If $\tf\in \CI_{c}(\pi^{*}\tcU_{i})$ is the pull-back of $f$ to $\pi^{*}\tcU_{i}$, then the way we constructed $f$ insures that there exists $\tp\in {}^{\Phi}N^{*}\pa\tM$ with image $p\in {}^{\cF}N^{*}\pa X$ such that \eqref{faith.2} holds.  In particular, if we set $(y_{0},\eta_{0}, \tau_{0})= \Phi_{*}(\tp)\in N^{*}Y\cong T^{*}Y\times \bbR$, then above the point $(y_{0},\eta_{0},\tau_{0})$, we have
\begin{equation}
   (\widehat{N}_{\Phi}(P)(\widetilde{f}))(y_{0},\eta_{0},\tau_{0})\ne 0 \; \mbox{in} \ \CI(\widetilde{Z}_{y_{0}}).
\label{symb.13}\end{equation}
 
 Recall also that the point $\tp$ was chosen in such a way that the action $\Gamma$ on $Y$ is free at $y_{0}$.
Choose coordinates $y$ on $B_{i}$ so that $y_{0}=0$ corresponds to the origin.
Let $\rho\in \CI_{c}(N_{y_{0}}Y)$ be a function supported in a small neighborhood of the origin with Fourier transform $\hat{\rho}\in \cS(N_{y_{0}}^{*}Y)$ such
that $\hat{\rho}(\eta_{0},\tau_{0})\ne 0$.  Let $z$ be coordinates on
$F_{i}$.  On $\widetilde{\cU}_{i}\times (0,\epsilon)$, we can then use the coordinates\begin{equation}
      u=\frac{1}{x}, \quad v= \frac{y}{x}, \quad z=z.
\label{symb.13b}\end{equation}
In these coordinates, consider the function
\[
      \tih=\rho(v,u- \frac{2}{\epsilon})\left(\left.\widetilde{f}\right|_{{}^{\Phi}N^{*}\widetilde{\cU}_{i}}(y_{0},z,\eta_{0},\tau_{0})\right).
\]
This function is well-defined with $\tih\in \CI_{c}(\widetilde{\cU}_{i}\times (0,\epsilon)_{x})$  (that is,  $u>\frac{1}{\epsilon}$ on the support of $\tih$) provided $\rho$ is chosen to have a sufficiently small support.       

Let $\chi$ be the pull-back to $\cU_{i}\times [0,\epsilon)_{x}$ of a cut-off function on $\cU_{i}$ such that $q_{c}^{*}\chi\equiv 1$ on the support of $\tih$ and  consider the operator
$P_{\chi}= P\circ q_{c}^{*}\chi$ obtained by first multiplying with 
$q_{c}^{*}\chi$ and then acting with $P$.  It is 
$\Gamma$-equivariant as well since $q_{c}^{*}\chi$ is a 
$\Gamma$-invariant function.  Since $q_{c}^{*}\chi\equiv 1$ on the support of $\tih$,
we have that $P\tih= P_{\chi}\tih$.  In the coordinates \eqref{symb.13b},
this can be written in terms of the Schwartz kernel of $P_{\chi}$ (\cf formula (3.10) in \cite{Mazzeo-Melrose}),
\begin{equation}
  (P_{\chi}\tih) (u,v,z)= \int K_{P_{\chi}}(u,v,S,Y,z,z')
  \tih(u-S, v-Y, z,z') dSdYdz', 
\label{symb.22}\end{equation}    
where $ S= u-u'$, $Y= v-v'$ and $z$ represents a point on the fibre $\widetilde{Z}$ not necessarily in 
$\widetilde{\cU}_{i}$.  On the other hand, we can make the restriction of the 
 normal operator $N_{\Phi}(P_{\chi})$ (not its Fourier transform)  at $y_{0}$
act on $\tih$ via the inclusion $\widetilde{\cU_{i}}\times(0,\epsilon)_{x}
\subset NY\times F_{i}$ given by the coordinates \eqref{symb.13b},
\begin{equation}
   \left.N_{\Phi}(P_{\chi})\right|_{y_{0}} \tih= 
\int K_{N_{\Phi}(P_{\chi})}(y_{0},S,Y,z,z')\tih(u-S,v-Y,z')dSdYdz'.  
\label{symb.23}\end{equation}
By our choice of $\tih$, notice that this is not zero.  Let $\tp=(u_{0},v_{0},z_{0})\in\tM$ be a point such that $\left.N_{\Phi}(P)\right|_{y_{0}}\tih(\tp)\ne 0$.  
For $k\in \bbN$, consider the new function
\begin{equation}
   \tih_{k}(u,v,z)= \tih(u-k,v,z).
\label{symb.24}\end{equation}
By translation invariance, we will also have that $\left. N_{\Phi}(P_{\chi})\right|_{y_{0}} \tih_{k} (\tp_{k})\ne 0$ with $\tp_{k}=(u_{0}+k,v_{0},z_{0})$.
Since the support of $\tih_{k}$ can be made arbitrarily close to the 
fibre above $y_{0}$ in the boundary of 
$\widetilde{\cU}_{i}\times [0,\epsilon)_{x}$ by taking $k$ large enough,
we see that 
\begin{equation}
         \lim_{k\to \infty} (P_{\chi}\tih_{k}(\tp_{k}) -\left.N_{\Phi}(P_{\chi})\right|_{y_{0}} \tih_{k}(\tp_{k}))=0. 
\label{symb.25}\end{equation}
In particular, taking
$k$ large enough, we have that $P_{\chi}\tih_{k}(\tp_{k})\ne 0$.  If $h_{k}$  denotes the 
corresponding function in $\cU_{i}\times [0,\epsilon)_{x}\subset M$ under
the quotient map, then, since the action of $\Gamma$ is free at $y_{0}\in Y$ and using the compactness property in (iii) of Definition~\ref{tfb.13}, we have that
\begin{equation}
        \lim_{k\to \infty} \left( R_{q_{c}}(P_{\chi})h_{k}(p_{k})- P_{\chi}\tih_{k}(\tp_{k})  \right)=0
\label{symb.24b}\end{equation}
where $p_{k}= q_{c}(\tp_{k})$.  In particular, for $k$ large enough, we have that
$R_{q_{c}}(P)h_{k}=R_{q_{c}}(P_{\chi})h_{k}\ne 0$, 
which shows that  $R_{q_{c}}(P)\ne 0$.

\end{proof}

With this lemma, we can now define the normal operator of a foliated cusp
pseudodifferential operator.

\begin{definition}
Given $P\in \Psi^{m}_{\cF}(X;E,F)$ of the form $P=\pi_{c}(P_{M})+
P_{X}$ with $P_{M}\in \Psi^{m}_{\Phi,\Gamma}(\widetilde{M};E,F)$ and 
$ P_{X}\in \dot{\Psi}^{m}(X;E,F)$, we define its normal
operator by
\[
   N_{\cF}(P)= r_{\psi}(N_{\Phi}(P_{M}))\in \Psi^{m}_{\sus-\cF}(\pa X;E,F),
\]
with Fourier transform $\widehat{N}_{\cF}(P)= r_{\widehat{\psi}}(\widehat{N}_{\Phi}(P_{M}))$,
where $\Psi^{m}_{\sus-\cF}(\pa X;E,F)$ is the image under the
representation $r_{\psi}$ of $\Gamma$-invariant operators in
$\Psi^{m}_{\sus-\Phi}(\pa \widetilde{M};E,F)$.  
By the previous lemma, this does not depend on the choice of 
$P_{M}$ and $P_{X}$.
\label{symb.19}\end{definition}

As for fibred cusp operators, the normal operator induces a short exact 
sequence
\begin{equation}
\xymatrix{
 0 \ar[r] & x\Psi^{m}_{\cF}(X;E,F) \ar[r] & \Psi^{m}_{\cF}(X;E,F) \ar[r]&
\Psi^{m}_{\sus-\cF}(\pa X;E,F) \ar[r] &  0. 
}
\label{symb.20}\end{equation}
This can be seen directly from Lemma~\ref{symb.11} and the corresponding
short exact sequence for $\widetilde{M}$,
\begin{equation}
\xymatrix{
 0 \ar[r] & x\Psi^{m}_{\Phi}(\widetilde{M};E,F) \ar[r] & \Psi^{m}_{\Phi}(\widetilde{M};E,F) \ar[r]&
\Psi^{m}_{\sus-\Phi}(\pa \widetilde{M};E,F) \ar[r] &  0. 
}
\label{symb.21}\end{equation}

\section{$\cF$-Sobolev spaces and a compactness criterion} \label{mp.0}

Let $L^{2}_{\cF}(X)$ be the $L^{2}$-space corresponding to a choice of  
$\cF$-metric $g_{\cF}$.  If $E$ and $F$ are smooth complex vector bundles with a choice of Hermitian metrics $h_{E}$ and $h_{F}$, then there are corresponding
$L^{2}$-spaces of sections $L^{2}_{\cF}(X;E)$ and $L^{2}_{\cF}(X;F)$ for these vector bundles.

One can also define $\cF$-Sobolev spaces.  Let $\cU_{1}, \ldots, \cU_{k}$ be
a finite open covering of $\pa X$ such that $\left. \cF\right|_{\cU_{i}}$ is induced by a trivial fibration 
\[
      \Phi_{i}: \cU_{i}\to B_{i} 
\] 
with fibres and base diffeomorphic to open balls in the Euclidean space.  We also assume the $\cU_{i}$ are chosen so that for each $i$ 
there is a section $\pa s_{i}$ of $\pa \widetilde{M} \to \pa M$ over 
$\cU_{i}$ inducing a diffeomorphism onto its image $\widetilde{\cU}_{i}= \pa s_{i}(\cU_{i})$.  Set $V_{i}= \cU_{i}\times [0,\epsilon)_{x}\subset M\subset X$.
  On each $V_{i}$, the metric $g_{\Phi_{i}}= \left. g_{\cF}\right|_{V_{i}}$ is a fibred cusp metric.  Using the definition of \cite{Mazzeo-Melrose}, there is a corresponding fibred cusp Sobolev space $H^{m}_{\Phi_{i}}(V_{i};E)$ of order $m\in \bbR$.  Choose an open set $V_{0}\subset X\setminus \pa X$ so that 
$V_{0}$ together with $V_{1}, \ldots, V_{k}$ is a finite covering of $X$.  Let $\varphi_{i}\in\CI_{c}(V_{i})$ be a partition of 
unity subordinate to this finite covering.  On $\CI_{c}(X;E)$, we
can consider the norm
\begin{equation}
  \| u\|_{H^{m}_{\cF}}^{2}= \|\varphi_{0}u \|^{2}_{H^{m}_{g_{\cF}}(V_{0};E)}
+ \sum_{i=1}^{k} \| \varphi_{i}u\|^{2}_{H^{m}_{\Phi_{i}}(V_{i};E)}.
\label{sob.1}\end{equation}  
\begin{definition}
We define $H^{m}_{\cF}(X;E)$ to be the closure of $\CI_{c}(X;E)$ with respect
to the norm \eqref{sob.1}.  We refer to the discussion after Proposition~\ref{sob.3} to see this definition does not depend on the choice of covering and partition of unity.
\label{sob.2}\end{definition}
We can also consider the corresponding weighed version.  
There is a continuous inclusion 
\begin{equation}
   x^{\ell}H^{m}_{\cF}(X;E)\subset x^{\ell'}H^{m'}_{\cF}(X;E)
\label{mp.5}\end{equation}
if and only if $\ell\ge \ell'$ and $m\ge m'$.  The inclusion is compact
if and only if $\ell> \ell'$ and $m>m'$.

\begin{proposition}
An $\cF$-operator $P\in \Psi^{m}_{\cF}(X;E,F)$ induces a continuous linear
map
\[
           P: x^{\ell}H^{m+k}_{\cF}(X;E)\to x^{\ell}H^{k}_{\cF}(X;F)
\]
for all real numbers $\ell$ and $k$.
\label{sob.3}\end{proposition}
\begin{proof}
By definition, an operator $P\in \Psi^{m}_{\cF}(X;E,F)$ is of the form
\begin{equation}
 P=\pi_{c}(P_{1})+P_{2}, \quad P_{1}\in \Psi^{m}_{\cF}(M;E,F), \;\;
        P_{2}\in \dot{\Psi}^{m}(X;E,F).
\label{mp.2}\end{equation}
Using a partition of unity and the boundedness result of \cite{Mazzeo-Melrose}, it is easy to see that the operator $P_{2}$ induces a bounded linear map
\[
               P_{2}:  x^{\ell}H^{m+k}_{\cF}(X;E)\to x^{\ell}H^{k}_{\cF}(X;F).
\]
For $\pi_{c}(P_{1})$, let $\widetilde{P}_{1}\in \Psi^{m}_{\Phi,\Gamma}(\widetilde{M};E,F)$  be a corresponding $\Gamma$-invariant operator such that 
$P_{1}= R_{q_{c}}(\widetilde{P}_{1})$.

Let $\varphi\in \CI(\widetilde{M})$ be a partition of unity relative to $\Gamma$ on $\widetilde{M}$ obtained from the pull-back of a corresponding partition
of unity $\pa\varphi\in \CI_{c}(\pa\widetilde{M})$ 
relative to $\Gamma$ on $\pa \widetilde{M}$.  We see from condition $(iii)$ of
Definition~\ref{tfb.13} and the Sobolev-boundedness result of \cite{Mazzeo-Melrose} that  $\widetilde{P}_{1}\circ\varphi$ induces a bounded linear map
\[
        \widetilde{P}_{1}\circ \varphi: x^{\ell}H^{m+k}_{\Phi}(\widetilde{M};E)
\to x^{\ell}H^{k}_{\Phi}(\widetilde{M};F).
\]  
Since for $f\in \dot{\mathcal{C}}^{\infty}(X)$, we have
\[
  \pi_{c}(P_{1})f= c_{*}P_{1}c^{*}f, \quad P_{1}c^{*}f=
(q_{c})_{*}\widetilde{P}_{1}(\varphi q_{c}^{*}c^{*}f),
\]
we see that this implies $\pi_{c}(P_{1})$ induces a bounded linear map
\[
   \pi_{c}(P_{1})=  x^{\ell}H^{m+k}_{\cF}(X;E)\to x^{\ell}H^{k}_{\cF}(X;F).
\]
\end{proof}

To see that the definition of $H^{m}_{\cF}(X;E)$ does not depend of the choice of covering and partition of unity, we provide an alternative description.  For 
$m>0$, let $A_{\frac{m}{2}}\in \Psi^{\frac{m}{2}}_{\cF}(X;E)$ be
an elliptic ${\cF}$-operator and consider the operator 
\begin{equation}
          D_{m}=A^{*}_{\frac{m}{2}}A_{\frac{m}{2}}+1.
\label{sob.4a}\end{equation}
In particular, it induces a continuous linear map $D_{m}: H^{m+k}_{\cF}(X;E)\to H^{k}(X;E)$.  Since it is elliptic, by a standard construction,
there exists $B_{-m}\in \Psi^{-m}_{\cF}(X;E)$ such that 
\begin{equation}
    B_{-m}D_{m}-\Id_{E}= R \in \Psi^{-\infty}_{\cF}(X;E).
\label{sob.4}\end{equation} 
In particular, if $u\in H^{q}(X;E)$ is such
that $D_{m} u=0$, we see from \eqref{sob.4} that $u=Ru$, so that 
\[
    u\in \bigcap_{j} H^{j}_{\cF}(X;E).
\]
In that case, 
\[
  D_{m}u=0 \quad \Longrightarrow \quad \|A_{\frac{m}{2}}u\|^{2}_{L^{2}_{\cF}}
   + \|u\|_{L^{2}_{\cF}}^{2} =0 \quad \Longrightarrow \quad u\equiv 0.
\]

Since $D_{m}$ is formally self-adjoint, this means it induces a bijective
continuous linear map
\[
       D_{m}: H^{m}_{\cF}(X;E)\to L^{2}_{\cF}(X;E).
\]
In particular, instead of \eqref{sob.1}, we can use the equivalent norm $u\mapsto \|D_{m}u\|_{L^{2}_{\cF}}$ for $m>0$ to define $H^{m}_{\cF}(X;E)$.  If we think
of $H^{-m}_{\cF}(X;E)$ as the dual of $H^{m}_{\cF}(X;F)$ seen as a subspace of 
$\mathcal{C}^{-\infty}(X;E)$, then $D^{*}_{m}$ induces a bijective continuous
linear map 
\[
  D_{m}^{*}: L^{2}_{\cF}(X;E)\to H^{-m}_{\cF}(X;E).
\]
Thus, in this case, we can use the equivalent norm $u\mapsto \| (D^{*}_{m})^{-1}u\|_{L^{2}_{\cF}}$ on  $H^{-m}_{\cF}(X;E)$.

\begin{theorem}[Compactness criterion]
For $\delta>0$, an operator $P\in \Psi^{m-\delta}_{\cF}(X;E,F)$ is compact
from $x^{\ell}H^{m+k}_{\cF}(X;E)$ to $x^{\ell}H^{k}_{\cF}(X;F)$ if and only if $N_{\cF}(P)=0$.  In
particular, a polyhomogeneous operator $P\in \Psi^{m}_{\cF}(X;E,F)$ is compact as a map
from $x^{\ell}H^{m+k}_{\cF}(X;E)$ to $x^{\ell}H^{k}_{\cF}(X;F)$ 
if and only if it is in $x\Psi^{m-1}_{\cF}(X;E,F)$.
\label{mp.6}\end{theorem}
\begin{proof}
By considering the operator $\tilde{P}= x^{\ell}Px^{-\ell}$ and using the fact
$N_{\cF}(\tilde{P})= N_{\cF}(P)$ (see Remark~\ref{conjugation.2}), we can reduce to the case where $P\in \Psi^{m-\delta}_{\cF}(X;E,F)$ is seen as a bounded linear
operator 
\begin{equation}
    H^{m+k}_{\cF}(X;E)\to H^{k}_{\cF}(X;F).
\label{sob.5}\end{equation}
Since any operator $Q\in \dot{\Psi}^{m-\delta}(X;E,F)$ induces a compact 
operator
\[
          Q:H^{m+k}_{\cF}(X;E)\to H^{k}_{\cF}(X;F)  ,
\]
we see by the density of $\dot{\Psi}^{m- \delta}(X;E,F)$ in
$x\Psi^{m-\delta}_{\cF}(X;E,F)$ (using the topology of $\Psi^{m}_{\cF}(X;E,F)$)
that any operator in $x\Psi^{m-\delta}_{\cF}(X;E,F)$ is compact as a continuous linear map from $H^{m+k}_{\cF}(X;E)$ to $ H^{k}_{\cF}(X;F)$ .  This shows
that $P\in \Psi^{m-\delta}_{\cF}(X;E,F)$ is compact in that sense
whenever $N_{\cF}(P)=0$.

Conversely, if $P\in \Psi^{m-\delta}_{\cF}(X;E,F)$ is such that 
$N_{\cF}(P)\ne 0$, we need to show that $P$ is not compact as a map from
$H^{m+k}_{\cF}(X;E)$ to $H^{k}_{\cF}(X;F)$.  
Without loss of generality, we can assume that $E=F=\underline{\bbC}$ are
trivial line bundles so that $P\in \Psi^{m-\delta}_{\cF}(X)$.  Since any 
operator in $x\Psi_{\cF}^{m-\delta}(X)$ is compact as a map from 
$H^{m+k}_{\cF}(X)$ to $H^{k}_{\cF}(X)$, we can also assume $P$ is
of the form $\pi_{c}\circ R_{q_{c}}(A)$ with $A\in \Psi^{m-\delta}_{\Phi,\Gamma}(\tM)$.  For the same reason, we can assume $A$ is translation invariant in the $u=\frac{1}{x}$ variable near $\pa \tM$, namely, its Schwartz kernel is of the form 
\[
    K_A(u,u', v,v', z,z')= K_A(u-u', v,v',z,z'), \quad \mbox{for} \; u,u' >>0. 
\]

Since we assume $N_{\cF}(P)\ne 0$, we can apply the construction
in the proof Lemma~\ref{symb.11}.  Thus consider the sequence of functions
$h_{k}$ given in \eqref{symb.24b}.  If we define the $L^{2}$-norm of
$L^{2}_{\cF}(M)$ using an $\cF$-metric which is of the form
\[
          g_{\cF}= du^{2}+ dv^{2}+ g_{F_{i}}
\] 
in $\cU_{i}\times (0,\epsilon)_{x}$, then clearly each element of the sequence
$h_{k}$ has the same $L^{2}$-norm. Similarly, using the Fourier transform in the $u$ and $v$ variables, we can choose the various 
$\cF$-Sobolev norms so that they are 
translation invariant in the variables $u$ and $v$.  
With this choice, each element of the sequence $h_{k}$ has the same
$H^{m+k}_{\cF}$-norm, which we can assume is equal to $1$ by rescaling.  
Thus, since $P$ is translation 
invariant in the $u$ variable near $\pa X$,  each element of the sequence $\{Ph_{k}\}_{k\in\bbN}$ will have the same non-zero $H^{k}_{\cF}$-norm provided we choose the $\cF$-metric $g_{\cF}$ to be translation invariant in the $u$ variable near $\pa X$.
Moreover, the support of $P h_{k}$ goes
to infinity as $k\to\infty$, so $\{Ph_{k}\}_{k\in\bbN}$ has no converging subsequence. 
This means $P$ cannot be compact.   
 
\end{proof}

\section{Sobolev spaces for $\cF$-suspended operators} \label{sss.0}

Let $V$ be a finite dimensional real vector space and $W$ a compact manifold.  For $E, F$ complex vector bundles over $W$, we can consider the space
\begin{equation}
  \Psi^{m}_{\sus(V)}(W;E,F) \subset \Psi^{m}(W\times V;E,F)
\label{sss.1}\end{equation}
of $V$-suspended operators of order $m$ on $W$, see \cite{Melrose_eta}
and \cite{Mazzeo-Melrose} for a definition.  Let $\langle \cdot,\cdot\rangle_{V}$ be a choice of inner product on $V$ and $g_{V}$ be the corresponding Euclidean metric on $V$ seen as a manifold.     If $r\in \CI(V)$ is the distance function from the origin, then as 
described in \cite{MelroseGST}, the function $\rho= \frac{1}{\sqrt{r^{2}+1}}\in \CI(\overline{V})$ can be seen as a boundary defining function for the radial
compactification $\overline{V}$ of $V$.  Thus $\overline{V}$ is a manifold with
boundary diffeomorphic to the closed unit ball and such that $V= \overline{V}\setminus \pa \overline{V}$.  If  $g_{W}$ is a choice of Riemannian metric on $W$, then the product metric
\begin{equation}
   g_{\psi}= g_{V} \oplus g_{W} 
\label{sss.2}\end{equation}
is a fibred cusp metric on $V\times W\subset \overline{V}\times W$ with fibration
on the boundary $\pa (\overline{V}\times W)= \pa \overline{V}\times W$ given by the projection on the left factor
\begin{equation}
\xymatrix{  W  \ar@{-}[r] & \pa \overline{V}\times W \ar[d]^{\psi} \\
                               &    \pa \overline{V}
}
\label{sss.3}\end{equation}
From this point of view, the inclusion \eqref{sss.1} can be refined to an inclusion
\begin{equation}
   \Psi^{m}_{\sus(V)}(W;E,F) \subset \Psi^{m}_{\psi}(\overline{V}\times W;E,F).
\label{sss.4}\end{equation}
In other words, $V$-suspended operators on $W$ can be seen as a particular type of fibred cusp pseudodifferential operators on $\overline{V}\times W$.   In
this particular case, the bundle ${}^{\psi}N\pa (\overline{V}\times W)$ is canonically trivial with its fibres canonically identified with $V$.  Under this
identification, we have
\begin{equation}
  \Psi^{m}_{\sus-\psi}(\pa (\overline{V}\times W);E,F) = \CI(\pa \overline{V};
  \Psi^{m}_{\sus(V)}(W;E,F))
\label{sss.5}\end{equation} 
and the normal operator $N(P)$ of a $V$-suspended operator $P$ seen as a 
$\psi$-operator is just the constant family given by 
\begin{equation}
      N(P)_{q}= P, \quad \forall \ q\in \pa \overline{V}.
\label{sss.6}\end{equation}
This suggests the following definition.

\begin{definition}
The $V$-suspended Sobolev space of order $m$ on $W$ is 
\[
H^{m}_{\sus(V)}(W;E) = H^{m}_{\psi}(\overline{V}\times W;E)
\]
where $H^{m}_{\psi}(\overline{V}\times W;E)$ is the $\psi$-Sobolev space
of order $m$ 
associated to the fibred cusp metric $g_{\psi}= g_{V}\oplus g_{W}$.   
\label{sss.7}\end{definition}
\begin{proposition}
A $V$-suspended operator $P\in \Psi^{m}_{\sus(V)}(W;E,F)$ defines a continuous linear map 
\[
P: H^{m+k}_{\sus(V)}(W;E)\to H^{k}_{\sus(V)}(W;F)
\]
for all $k\in \bbR$.  This map is Fredholm if and only if $P$ is invertible.  
\label{sss.8}\end{proposition}
\begin{proof}
Since $\Psi^{m}_{\sus(V)}(W;E,F)\subset \Psi^{m}_{\psi}(\overline{V}\times W;E,F)$, this mapping property follows from the corresponding one for fibred cusp operators.   Moreover, since $P$ is translation invariant in the $V$ direction, the only way it can be Fredholm is if it is invertible.  Alternatively, this can be shown using \eqref{sss.6} and the Fredholm criterion of Mazzeo and Melrose \cite{Mazzeo-Melrose} for fibred cusp operators.
\end{proof}
 
 If instead we have a family of $V$-suspended operators parametrized by a smooth compact manifold $B$, then the natural Sobolev space one should consider for such a family $P\in \CI(B;\Psi^{m}_{\sus(V)}(W;E,F))$ is given by
 taking the closure of $\cS(B\times V\times W;E) = \dot{\cC}^{\infty}(B\times \overline{V}\times W;E)$ using the norm
 \begin{equation*}
    \| u \| = \sup_{ b\in B} \| u_{b}\|_{H^{m}_{\sus(V)}(W;E)}, \quad u_{b}= 
    \left. u \right|_{\{b\}\times V\times W}.
 \end{equation*}
 Since the normal operator of a fibred cusp operator is a family of suspended operators, we can apply the construction above to this case.  More precisely, suppose now $W$ is a manifold with boundary with 
 $\rho\in \CI(W)$ a choice of boundary defining function and with a fibration on the boundary $\psi: \pa W \to B$.
In this case, an operator $P\in \Psi^{m}_{\sus-\psi}(\pa W;E,F)$ is a family
of operators on ${}^{\psi}N\pa W$ parametrized by $B$.  For each $b\in B$,
we have a corresponding operator 
\[
      P_{b} \in \Psi^{m}_{\sus(N_{b}B)}(\psi^{-1}(b);E,F)
\]
where $NB\to B$ is the bundle leading to the canonical identification
${}^{\psi}N\pa W= \psi^{*}NB$.  
\begin{definition}
On the space of Schwartz sections $\cS({}^{\psi}N\pa W;E)$, consider the norm
\[
   \| u\|_{H^{m}_{\sus-\psi}(\pa W;E)}= \sup_{b\in B} 
    \| u_{b}\|_{H^{m}_{\sus(N_{b}B)}(\psi^{-1}(b);E)}
\]
where $u_{b}= \left. u\right|_{N_{b}B\times \psi^{-1}(b)}$ and we use the identification
  $N_{b}B\times\psi^{-1}(b)= \left.{}^{\psi}N\pa W  \right|_{\psi^{-1}(b)}$.
 We define $H^{m}_{\sus-\psi}(\pa W;E)$ to be the closure of $\cS({}^{\psi}N\pa W;E)$ with respect to this norm. 
\label{sss.9}\end{definition}

The space $H^{m}_{\sus-\psi}(\pa W;E)$ is a Banach space.  It is also naturally
a $\cC^{0}(B)$-Hilbert module.   Proposition~\ref{sss.8} has the following immediate generalization.
\begin{proposition}
An operator $P\in \Psi^{m}_{\sus-\psi}(\pa W;E,F)$ induces a continuous linear
map 
\[
    P: H^{m+k}_{\sus-\psi}(\pa W;E) \to H^{k}_{\sus-\psi}(\pa W;F)
\]
of $\cC^{0}(B)$-Hilbert modules for all $k\in \bbR$. 
\label{sss.10}\end{proposition}

To define the Sobolev spaces associated more generally to $\cF$-suspended
operators, we can proceed as follows.  Let $\varphi\in \CI_{c}(\pa \widetilde{M})$
be a partition of unity relative to $\Gamma$ for the cover $\pa \widetilde{M}\to M$.  If $\pi: {}^{\Phi}N\pa \widetilde{M}\to \pa \widetilde{M}$ is the vector bundle
projection, the function $\pi^{*}\varphi \in \CI({}^{\Phi}N\pa \widetilde{M})$ is a partition of
unity relative to $\Gamma$ for the quotient map
\[
      q_{\psi}: {}^{\Phi}N\pa \widetilde{M}\to {}^{\cF}N\pa M.
\]
\begin{definition}
On the space of smooth sections $\cS({}^{\cF}N\pa X;E)$, consider the norm
\[
   \| u\|_{H^{m}_{\sus-\cF}(\pa X;E)} = \| (\pi^{*}\varphi) q_{\psi}^{*} u\|_{H^{m}_{\sus-\Phi}(\pa \widetilde{M};E)}.
\]
The space $H^{m}_{\sus-\cF}(\pa X;E)$ is the closure  of $\cS({}^{\cF}N\pa X;E)$ 
with respect to this norm.  
\label{sss.11}\end{definition}
Proposition~\ref{sss.10} generalizes immediately to the following.
\begin{proposition}
An $\cF$-suspended operator $P\in \Psi^{m}_{\sus-\cF}(\pa X;E,F)$ induces
a continuous linear map 
\[
   P: H^{m+k}_{\sus-\cF}(\pa X;E) \to H^{k}_{\sus-\cF}(\pa X;F)
\]
for all $k\in \bbR$.
\label{sss.12}\end{proposition}

\section{Fredholm Criterion} \label{fc.0}

If the normal operator $N_{\cF}(P)$ of a fully elliptic $\cF$-operator $P\in \Psi^{m}_{\cF}(X;E,F)$ is invertible, it is not always the case that $N(P)^{-1}$ is in $\Psi^{-m}_{\sus-\cF}(\pa X;F,E)$, but when it is, a parametrix can be constructed using the usual method.

\begin{theorem}[Parametrix construction] 
If $P\in \Psi^{m}_{\cF}(X;E,F)$ is elliptic with an invertible normal operator such that $N_{\cF}(P)^{-1}\in \Psi^{-m}_{\sus-\cF}(\pa X;F,E)$, then there exists $Q\in \Psi^{-m}_{\cF}(X;F,E)$ such that 
\[
      PQ-\Id_{F}\in \dot{\Psi}^{-\infty}_{\cF}(X;F), \quad 
      QP-\Id_{E}\in \dot{\Psi}^{-\infty}_{\cF}(X;E).
\]
In particular, for all $k,\ell\in\bbR$, the operator $P$ is Fredholm as a map
\[
     P: x^{\ell}H^{m+k}_{\cF}(X;E)\to x^{\ell}H^{k}_{\cF}(X;F).
\]
\label{pc.14}\end{theorem}
\begin{proof}
Since $P$ is elliptic and $N_{\cF}(P)^{-1}\in \Psi^{-m}_{\sus-\cF}(\pa X;F,E)$, we can find
$Q_{1}\in \Psi^{-m}_{\cF}(X;F,E)$ such that $N_{\cF}(Q)=N_{\cF}(P)^{-1}$ and 
$\sigma_{-m}(Q)= \sigma_{m}(P)^{-1}$.  Thus,
\[
   Q_{1}P-\Id_{E} = R_{1} \in x\Psi^{-1}_{\cF}(X;E,F).
\] 
Suppose we can also find $Q_{j} \in x^{j-1}\Psi^{-m-j+1}_{\cF}(X;F,E)$ for
$j\le k$ such that 
\begin{equation}
  (Q_{1}+\cdots+ Q_{k}) P -\Id_{E}= R_{k} \in x^{k}\Psi^{-k}_{\cF}(X;E).
\label{pc.15}\end{equation}
This can be improved at the next level provided we can find $Q_{k+1}\in x^{k}\Psi^{-m-k}_{\cF}(X;F,E)$ such that
\[
    Q_{k+1}P \equiv R_{k} \; \mod \; x^{k+1}\Psi^{-k-1}_{\cF}(X;E). 
\]
For this, it suffices to take $Q_{k+1}$ such that 
\[
   N_{\cF}(x^{-k}Q_{k+1})= N_{\cF}(x^{-k}R_{k})N_{\cF}(P)^{-1}, 
   \quad \sigma_{-m-k}(x^{-k}Q_{k})= \sigma_{-k}(x^{-k}R_{k})\sigma_{m}(P)^{-1}.
\]
Proceeding recursively, we can therefore find $Q_{j}\in x^{j-1}\Psi^{-m-j+1}_{\cF}(X;F,E)$ such that \eqref{pc.15} holds for all $k\in \bbN$.  Thus, taking
$Q\in \Psi^{-m}_{\cF}(X;F,E)$ to be some asymptotic sum of the $Q_{j}$, we 
get a left parametrix
\[
  QP-\Id_{E}\in \dot{\Psi}^{-\infty}_{\cF}(X;E).
\]
We can proceed in a similar way to construct a right parametrix $Q'$.  Clearly
we will have that $Q-Q'\in \dot{\Psi}^{-\infty}_{\cF}(X;F,E)$ so that $Q$ is
also a right parametrix.
\end{proof}
\begin{remark}
If the foliation $\cF$ satisfies Assumption~\ref{tfb.2b} except for the requirement that the action of $\Gamma$ on $Y$ is locally free, then the conclusion of Theorem~\ref{pc.14} still holds if instead $P\in \Psi^m_{\cF}(X;E,F)$ is elliptic of the form 
\[
        P = \pi_c \circ R_{q_c}(P_1) + P_2, \quad P_1\in \Psi^m_{\Phi,\Gamma}(\widetilde{M};E,F), \; P_2\in \dot{\Psi}^m(X),
\]
with $N_{\Phi}(P_1)$ invertible with inverse $N_{\Phi}(P_1)^{-1}$ in $\Psi^{-m}_{\sus-\Phi}(\pa\widetilde{M};F,E)$.  The proof is similar but uses $N_{\Phi}(P_1)^{-1}$ instead of $N_{\cF}(P)^{-1}$.   
\label{lcfr.1}\end{remark}

\begin{corollary}
If $P\in \Psi^{m}_{\cF}(X;E,F)$ is elliptic with an invertible normal operator such that $N_{\cF}(P)^{-1}\in \Psi^{-m}_{\sus-\cF}(\pa X;F,E)$  and $P$ is invertible as a map $P: \CI(X;E)\to \CI(X;F)$, then its inverse is an element of $\Psi^{-m}_{\cF}(X;F,E)$.
\label{pc.16}\end{corollary}
\begin{proof}
Let $Q\in \Psi^{-m}(X;F,E)$ be a parametrix for $P$ as in Theorem~\ref{pc.14} so that 
\[
           QP-\Id_{E} = R\in \dot{\Psi}^{-\infty}_{\cF}(X;E).
\]
Composing on the right with $P^{-1}$ we get
\begin{equation}
            P^{-1}= Q -RP^{-1}.
\label{pc.17}\end{equation}
Since $\dot{\Psi}^{-\infty}_{\cF}(X)$ is an ideal, $RP^{-1}$ is in
$\dot{\Psi}^{-\infty}_{\cF}(X;F,E)$ and the result follows from \eqref{pc.17}.
\end{proof}

A particular case where Theorem~\ref{pc.14} can be applied systematically is when the foliation $\cF$ can be resolved by a fibration having compact fibres.

\begin{corollary}
If the foliation $\cF$ on $\pa X$ satisfies Assumption~\ref{tfb.2b} with $\pa\tM$ a compact manifold, then for any fully elliptic operator $P\in \Psi^{m}_{\cF}(X;E,F)$, there exists an operator $Q\in \Psi^{m}_{\cF}(X;E,F)$ such that
  \[
      PQ-\Id_{F}\in \dot{\Psi}^{-\infty}_{\cF}(X;F), \quad 
      QP-\Id_{E}\in \dot{\Psi}^{-\infty}_{\cF}(X;E).
\]\label{pcf.1}\end{corollary}
\begin{proof}
Since $\pa \tM$ is compact and $\Gamma$ acts freely and properly discontinuously on $\pa\tM$, notice that $\Gamma$ has to be a finite group.  
According to Theorem~\ref{pc.14}, it suffices to show that $N_{\cF}(P)$ has an inverse in $\Psi^{-m}_{\sus-\cF}(\pa X;E,F)$.  Replacing $P$ by 
\[
   \left(\begin{array}{cc}
       P & 0 \\
       0 & P^{*} 
   \end{array} \right)
\]
if needed, we can assume that $E=F$.   

To show that   $N_{\cF}(P)^{-1}\in\Psi^{-m}_{\sus-\cF}(\pa X;E)$, it is enough to prove that 
$N_{\Phi}(\tP)$ is invertible, where $\tP\in \Psi^{m}_{\Phi,\Gamma}(\tM;E)$ is such that $r_{\psi}(N_{\Phi}(\tP))=N_{\cF}(P)$, since then $N_{\Phi}(\tP)^{-1}\in \Psi^{-m}_{\sus-\Phi;\Gamma}(\pa \tM;E)$ automatically.  To do this, we can reduce to the case $m=0$ by replacing  $N_{\Phi}(\tP)$ with $\tQ N_{\Phi}(\tP)$, where $\tQ\in \Psi^{-m}_{\sus-\Phi,\Gamma}(\pa \tM;E)$ is elliptic and invertible with $\Psi^{-m}_{\sus-\Phi,\Gamma}(\pa\tM;E)$ denoting the subspace of $\Gamma$-invariant $\Phi$-suspended operators in $\Psi^{-m}_{\sus-\Phi}(\pa \tM;E)$.  Since $\Gamma$ is finite, such an operator $\tQ$ is easy to construct by averaging over $\Gamma$.  

Thus, we can assume $N_{\Phi}(\tP)\in \Psi^{0}_{\sus-\Phi,\Gamma}(\pa\tM;E)$.
By Proposition~\ref{sss.12}, we have injective maps
\begin{equation}
\begin{gathered}
\iota_{\cF}: \Psi_{\sus-\cF}^{0}(\pa X;E) \to \cL(H^{0}_{\sus-\cF}(\pa X;E)), \\
\iota_{\Phi}: \Psi_{\sus-\Phi,\Gamma}^{0}(\pa \tM;E) \to \cL(H^{0}_{\sus-\Phi}(\pa \tM;E)). \end{gathered}
\label{pcf.2}\end{equation}      
These maps are in fact continuous as it can be easily shown.  Indeed, for any $u\in \cS({}^{\cF}N\pa X;E)$, the map
\[
        \Psi^{0}_{\sus-\cF}(\pa X;E) \ni A \mapsto \|Au\|_{H^{0}_{\cF-\sus}(\pa X;E)}
\]
is continuous.   Since the strong operator topology is weaker than the norm topology, this means the graph of $\iota_{\cF}$ is also closed when we use the norm topology on $ \cL(H^{0}_{\sus-\cF}(\pa X;E))$.  By the closed graph theorem, this means the map $\iota_{\cF}$ is continuous.  The proof that $\iota_{\Phi}$ is continuous is similar.

  Let $\cP^{0}_{\sus-\cF}(\pa X;E)$ and $\cP^{0}_{\sus-\Phi,\Gamma}(\pa \tM;E)$ denote the images of these maps and let $\overline{\cP}^{0}_{\sus-\cF}(\pa X;E)$ and $\overline{\cP}^{0}_{\sus-\Phi,\Gamma}(\pa \tM;E)$ be their closure with respect to the norm topology.  The map
$r_{\psi}: \Psi^{0}_{\sus-\Phi,\Gamma}(\pa \tM;E)\to \Psi^{0}_{\sus-\cF}(\pa X, E)$ naturally extends to give a map of $C^{*}$-algebras
\[
     \overline{r}_{\psi}: \overline{\cP}^{0}_{\sus-\Phi,\Gamma}(\pa \tM;E) \to
         \overline{\cP}^{0}_{\sus-\cF}(\pa X;E).
\]  
By Lemma~\ref{faith.2}, the map $r_{\psi}$ is injective.  In fact, the proof of this lemma generalizes directly to show that $\overline{r}_{\psi}$ is an injective map of $C^{*}$-algebras.  By a standard result (see for instance Proposition 1.3.10 in \cite{Dixmier}), we know then that an element in $ \overline{\cP}^{0}_{\sus-\Phi,\Gamma}(\pa \tM;E)$ is invertible if and only if it is invertible in  
$\overline{\cP}^{0}_{\sus-\cF}(\pa X;E)$.  In particular, if $N_{\cF}(P)$ is invertible, then so is $N_{\Phi}(\tP)$, which completes the proof.    
\end{proof}

It is still possible to get a precise Fredholm criterion when Theorem~\ref{pc.14} does not apply.  The proof of the following proposition is  inspired by the approach of Lauter, Monthubert and Nistor (Theorem 4 in \cite{Lauter-Monthubert-Nistor}) to get Fredholm criteria for algebras of pseudodifferential operators defined on groupoids.  
\begin{proposition}
A classical (or polyhomogeneous) $\cF$-operator $P\in \Psi^{m}_{\cF-\ph}(X;E,F)$ induces a Fredholm operator 
\[
      P: H^{m+k}_{\cF}(X;E)\to H^{k}_{\cF}(X;F)
\]
if and only if $P$ is elliptic and $N_{\cF}(P)$ is invertible as a map
\[
   N_{\cF}(P): H^{m+k}_{\sus-\cF}(\pa X;E) \to H^{k}_{\sus-\cF}(\pa X;F).
\]
\label{fc.1}\end{proposition}
\begin{proof}
Considering instead the operator
\[
\left( \begin{array}{cc} 0 & P^{*} \\ P & 0 \end{array}\right): H^{m+k}_{\cF}(X;E\oplus F)\to H^{k}_{\cF}(X;E\oplus F) , 
\]
we can reduce to the case where $E=F$ and $P$ is formally self-adjoint.
By Proposition~\ref{sob.3}, there is an injective map
\begin{equation}
  \iota_{m,k}: \Psi^{m}_{\cF-\ph}(X;E)\hookrightarrow \cL( H^{m+k}_{\cF}(X;E),
   H^{k}_{\cF}(X;E)).
\label{fc.2}\end{equation}
It is in fact continuous with respect to the norm topology.  Indeed, 
for any $u\in \dot{\mathcal{C}}^{\infty}(X;E)$, $v\in \dot{\mathcal{C}}^{\infty}(X;E)$, the map
\[
         \Psi^{m}_{\cF-ph}(X;E)\ni A \mapsto \langle Au, v\rangle_{L^{2}}
\]
is continuous.  Since the topology induced by the semi-norms 
$A\mapsto |\langle Au, v\rangle_{L^{2}}|$ is weaker than the norm topology,
this implies the graph of \eqref{fc.2} is closed for the product topology (using the norm topology for the space of bounded operators).  Thus, by the closed graph theorem, the map \eqref{fc.2} is continuous.

 Let $\cP^{m}_{k}(X;E)$ be the image of this map and $\overline{\cP}^{m}_{k}(X;E)$ its closure in the space $\cL( H^{m+k}_{\cF}(X;E),
   H^{k}_{\cF}(X;E))$.  Notice the principal symbol defines a continuous map 
\begin{equation}
   \sigma_{m}: \Psi^{m}_{\cF-\ph}(X;E) \to \CI( {}^{\cF}S^{*}X; \Lambda^{m}\otimes
 \hom(E))
\label{fc.3}\end{equation}  
where ${}^{\cF}S^{*}X$ is the cosphere bundle of ${}^{\cF}TX$.  If we use the
$\cC^{0}$-norm on $\CI({}^{\cF}S^{*}X;\Lambda^{m}\otimes \hom(E))$, this extends to a continuous map
\begin{equation}
  \overline{\sigma}_{m}: \overline{\cP}^{m}_{k}(X;E)\to \cC^{0}({}^{\cF}S^{*}X;\Lambda^{m}\otimes \hom(E))
\label{fc.4}\end{equation}
 Similarly, the normal operator induces a continuous linear map
 \begin{equation}
  \overline{N}_{\cF}: \overline{\cP}^{m}_{k}(X;E) \to 
   \cL( H^{m+k}_{\sus-\cF}(\pa X;E), H^{k}_{\sus-\cF}(\pa X;E)).
 \label{fc.5}\end{equation}
 By the compactness criterion of Theorem~\ref{mp.6}, these maps induce
 a continuous injective map
 \begin{multline}
  \overline{\cP}^{m}_{k}(X;E)/ \cK^{m}_{k} \to  \\
   \cC^{0}({}^{\cF}S^{*}X;\Lambda^{m}\otimes\hom(E)) 
  \oplus \cL( H^{m+k}_{\sus-\cF}(\pa X;E), H^{k}_{\sus-\cF}(\pa X;E)) 
 \label{fc.6}\end{multline}
 where $\cK_{k}^{m} \subset \overline{\cP}^{m}_{k}(X;E)$ is the subspace of
 compact operators.  When $m=0$, this is an injective map of $C^{*}$-algebras mapping the identity to the identity.  For such a map, it is a standard fact (see for instance Proposition~1.3.10 in \cite{Dixmier}) that an element of the initial $C^{*}$-algebra is invertible if and only if its image is invertible in the other $C^{*}$-algebra.
In particular, for $m=0$, an operator $P\in \Psi^{m}_{\cF-\ph}(X;E)\subset \overline{\cP}^{m}_{k}(X;E)$ has an inverse 
 $Q$ in the space  $\cL(H^{k}_{\cF}(X;E), H^{m+k}_{\cF}(X;E))$ modulo compact operators if 
 and only if $(\overline{\sigma}_{m}\oplus \overline{N}_{\cF})(P)$ has an inverse in 
\[\cC^{0}({}^{\cF}S^{*}X;\Lambda^{-m}\otimes\hom(E))\oplus \cL(H^{k}_{\sus-\cF}(\pa X;E), H^{m+k}_{\sus-\cF}(\pa X;E)).  \]
 As Lemma~\ref{cstar.1} below shows, this is still true when $m\ne 0$.
   Since $P$ induces a Fredholm operator
 \[
      P: H^{m+k}_{\cF}(X;E)\to H^{k}_{\cF}(X;E)
 \]
 if and only if it has an inverse $Q\in \cL(H^{k}_{\cF}(X;E), H^{m+k}_{\cF}(X;E))$ modulo compact operators, the result follows.
  \end{proof}
\begin{lemma}
For $m\in\bbR$, an operator $P\in \Psi^{m}_{\cF-\ph}(X;E)\subset \overline{\cP}^{m}_{k}(X;E)$ has an inverse 
 $Q$ in the space  $\cL(H^{k}_{\cF}(X;E), H^{m+k}_{\cF}(X;E))$ modulo compact operator if 
 and only if $(\overline{\sigma}_{m}\oplus \overline{N}_{\cF})(P)$ has an inverse in 
\[\cC^{0}({}^{\cF}S^{*}X;\Lambda^{-m}\otimes\hom(E))\oplus \cL(H^{k}_{\sus-\cF}(\pa X;E), H^{m+k}_{\sus-\cF}(\pa X;E)).  \]
\label{cstar.1}\end{lemma}  
\begin{proof}
As shown above, the result is true when $m=0$, since \eqref{fc.6} is then an injective map of $C^{*}$-algebras.  When $m\ne 0$, we can still use \eqref{fc.6} to define an injective map between $C^{*}$-algebras
\[
       \overline{\sigma}\oplus \overline{N}_{\cF}: \cA \hookrightarrow \cB
\]   
where an element of $\cA$ is a matrix $\left(\begin{array}{cc} a & b\\ 
c & d \end{array}  \right)$ with entries such that
\[
  a\in \overline{\cP}^{0}_{m+k}(X;E), \; b\in \overline{\cP}^{-m}_{m+k}(X;E), \;
  c\in   \overline{\cP}^{m}_{k}(X;E), \; d\in \overline{\cP}^{0}_{k}(X;E).
    \]
while an element of $\cB$ is a matrix $\left(\begin{array}{cc} a & b\\ 
c & d \end{array}  \right)$ with entries such that
\begin{equation*}
\begin{aligned}
  a &\in \cC^{0}({}^{\cF}S^{*}X;\Lambda^{0}\otimes\hom(E)) 
  \oplus \cL( H^{m+k}_{\sus-\cF}(\pa X;E)), \\
  b &\in \cC^{0}({}^{\cF}S^{*}X;\Lambda^{-m}\otimes\hom(E)) 
  \oplus \cL( H^{k}_{\sus-\cF}(\pa X;E), H^{m+k}_{\sus-\cF}(\pa X;E)) \\
  c &\in \cC^{0}({}^{\cF}S^{*}X;\Lambda^{m}\otimes\hom(E)) 
  \oplus \cL( H^{m+k}_{\sus-\cF}(\pa X;E), H^{k}_{\sus-\cF}(\pa X;E)), \\ 
  d &\in \cC^{0}({}^{\cF}S^{*}X;\Lambda^{0}\otimes\hom(E)) 
  \oplus \cL( H^{k}_{\sus-\cF}(\pa X;E)).    
  \end{aligned} 
 \end{equation*}
Suppose first that $m<0$.  Let $\overline{D}_{-m}$ be the invertible operator of
\eqref{sob.4a}.  Then clearly 
$\left(\begin{array}{cc} 0 & \overline{D}_{-m}\\ 
P & 0 \end{array}  \right)$ is invertible modulo compact operators if and only if $P$ is.  Similarly,
$(\overline{\sigma}\oplus \overline{N}_{\cF})\left(\begin{array}{cc} 0 & \overline{D}_{-m}\\ 
P & 0 \end{array}  \right)$ is invertible if and only if $(\overline{\sigma}\oplus \overline{N}_{\cF})(P)$ is.  Since an element of $\cA$ is invertible if and only if its image in $\cB$ is invertible, the lemma follows in this case.   When $m>0$, we can proceed similarly if we enlarge the $C^{*}$-algebras $\cA$ and $\cB$ to include the elements
\[
\left(\begin{array}{cc} 0 & \overline{D}_{m}^{-1}\\ 
0 & 0 \end{array}  \right), \quad \left(\begin{array}{cc} 0 & (\overline{\sigma}\oplus \overline{N}_{\cF})(\overline{D}_{m})^{-1}\\ 
0 & 0 \end{array}  \right)
\]  
respectively.

\end{proof}  
  
As the next result shows, this criterion can be formulated independently 
of $k\in \bbR$.

\begin{corollary}
If $P\in \Psi^{m}_{\cF-\ph}(X;E,F)$ is elliptic and $N_{\cF}(P)$ induces an
invertible continuous linear map
\[
   N_{\cF}(P): H^{m+k}_{\sus-\cF}(\pa X;E) \to H^{k}_{\sus-\cF}(\pa X;F),
\]
for some $k\in \bbR$, then it induces an invertible map for all $k\in\bbR$.
Moreover, $P$ then induces a Fredholm operator
\[
P: H^{m+k}_{\cF}(X;E)\to H^{k}_{\cF}(X;F)
\]
for all $k\in\bbR$ with nullspace in $H^{\infty}_{\cF}(X;E)$  
and index independent of $k$.  
\label{fc.7}\end{corollary}
\begin{proof}
By Proposition~\ref{fc.1}, we know that $P$ is Fredholm as a map 
\begin{equation}
    P:  H^{m+k}_{\cF}(X;E)\to H^{k}_{\cF}(X;F).
\label{fc.8}\end{equation}
Since $P$ is elliptic, there exists $Q\in \Psi^{-m}_{\cF}(X;\cF,E)$ such
that 
\[
    QP -\Id_{E}= R\in \Psi^{-\infty}_{\cF}(X;E).
\]
In particular, this means that for $u\in H^{p}_{\cF}(E)$ for some $p\in \bbR$, 
\begin{gather}
  Pu \in H^{l}_{\cF}(X;F)\quad \Longrightarrow \quad u +Ru\in H^{m+l}_{\cF}(X;E) \quad \Longrightarrow \quad   \label{fc.9a}
  u\in H^{m+l}_{\cF}(X;E); \\
 Pu=0 \quad \Longrightarrow \quad u \in H^{\infty}_{\cF}(X;E).
\label{fc.9}\end{gather}
Similarly, we have
\begin{gather}
  u\in H^{p}_{\cF}(X;F), \; P^{*}u\in H^{l}_{\cF}(X;E), \quad \Longrightarrow \quad 
  u \in H^{m+l}_{\cF}(X;F);  \\
u\in H^{p}_{\cF}(X;F), \;  P^{*}u=0, \quad \Longrightarrow \quad u\in H^{\infty}_{\cF}(X;F).   
\label{fc.10}\end{gather}
The map \eqref{fc.8} being Fredholm, it has in particular a closed range.  From
\eqref{fc.9a}, it follows that for $k'\ge k$, the induced map
\begin{equation}
  P: H^{m+k'}_{\cF}(X;E) \to H^{k'}_{\cF}(X;F)
\label{fc.11}\end{equation}
has closed range.  By \eqref{fc.9} and \eqref{fc.10}, we also see that it 
is in fact Fredholm with the same nullspace and index as the map \eqref{fc.8}.

On the other hand, by duality, we see that the induced map 
\[
   P^{*}: H^{-k}_{\cF}(X;F) \to H^{-k-m}_{\cF}(X;E)
\]
is also Fredholm.  By the same argument, we can conclude that for $k'\le k$,
the map 
\[
    P^{*}: H^{-k'}_{\cF}(X;F) \to H^{-k'-m}_{\cF}(X;E) 
\]
is Fredholm.  By duality, this means the map 
\[
         P: H^{m+k'}_{\cF}(X;E)\to H^{k'}_{\cF}(X;F)
\]
is also Fredholm for $k'\le k$.   Again, by \eqref{fc.9} and \eqref{fc.10}, 
its nullspace and index are the same as those of \eqref{fc.8}.  Using Proposition~\ref{fc.1}, we can therefore conclude that the induced map
\[
   N_{\cF}(P): H^{m+k'}_{\sus-\cF}(\pa X;E) \to H^{k'}_{\sus-\cF}(\pa X;F),
\]
is bijective for all $k'\in \bbR$. 
\end{proof}
This suggests the following definition.

\begin{definition}
An $\cF$-operator $P\in \Psi^{m}_{\cF}(E;F)$ is said to be \textbf{fully elliptic} if it is elliptic and its normal operator $N_{\cF}(P) \in \Psi^{m}_{\sus-\cF}(\pa X;E,F)$ is invertible as an operator 
\[
       N_{\cF}(P): H^{m+k}_{\sus-\cF}(\pa X) \to H^{k}_{\sus-\cF}(X;F)
\]
for all $k\in \bbR$.  
\label{pc.1}\end{definition}  
As in \cite{Mazzeo-Melrose} we can more generally let an $\cF$-operator acts on weighted Sobolev spaces.

\begin{theorem}[Fredholm criterion]
A classical $\cF$-operator $P\in \Psi^{m}_{\cF-\ph}(X;E,F)$ induces a Fredholm operator 
\[
 P: x^{\ell}H^{m+k}_{\cF}(X;E)\to x^{\ell}H^{k}_{\cF}(X;F)      
\]
if and only if $P$ is fully elliptic.  In this case, the index is independent
of $k$ and $\ell$.  Furthermore, $\ker P \subset x^{\ell}H^{\infty}_{\cF}(X;E)$ and its range is complementary to a subspace of 
$x^{\ell}H^{\infty}_{\cF}(X;F)$.
\label{fc.13}\end{theorem}
\begin{proof}
When $\ell=0$, this follows from Corollary~\ref{fc.7}.  If $\ell\ne 0$, we have  a commutative diagram  
\begin{equation}
\xymatrix{
   H^{m+k}_{\cF}(X;E) \ar[r]^{P_{\ell}}\ar[d]^{x^{\ell}} & H^{k}_{\cF}(X;F) \ar[d]^{x^{\ell}}  \\
 x^{\ell}H^{m+k}_{\cF}(X;E) \ar[r]^{P} & x^{\ell}H^{k}_{\cF}(X;F) 
}
\label{fc.14}\end{equation}
where $P_{\ell}= x^{-\ell}\circ P\circ x^{\ell}$.  Since $P_{\ell}\in \Psi^{m}_{\cF}(X;E,F)$ is elliptic and is such that $N_{\cF}(P_{\ell})= N_{\cF}(P)$ by
Remark~\ref{conjugation.2},
we see that the top horizontal map is Fredholm by Corollary~\ref{fc.7}.  Since
the vertical maps are isometries of Hilbert spaces, we conclude that the bottom
horizontal map is Fredholm with the same index.  To see that the index
does not depend on $\ell$, it suffices to notice $\ell\mapsto P_{\ell}$ is a continuous family of Fredholm operators, which means the index cannot jump.   
\end{proof}  

\begin{remark}
When a fully elliptic operator $P$ admits a parametrix as in Theorem~\ref{pc.14}, its nullspace is automatically a finite dimensional subspace of $\dot{\cC}^{\infty}(X;E)$, in particular, it does not depend on the choice of $\ell$ in the theorem above.  The author does not know if this holds more generally.    
\label{fc.15b}\end{remark}

\section{An index theorem for some Dirac-type operators}\label{it.0}

In this section, we will suppose that $X$ is even dimensional and oriented.   
To get an index formula for Dirac-type operators, we will make a different assumption on the foliation $\cF$.

\begin{assumption}
The foliation $\cF$ on $\pa X$ has \textbf{compact leaves} and can be described as
in Assumption~\ref{tfb.2b} with $\Gamma$ a \textbf{finite group}, but \textbf{without} assuming the action of $\Gamma$ on $Y$ is locally free.  Furthermore,
the fibration $\Phi:\pa\widetilde{M}\to Y$ has oriented fibres and base and the group $\Gamma$ acts on $\pa\widetilde{M}$ and $Y$ by orientation preserving diffeomorphisms.       
\label{fini.1}\end{assumption}

On $X$, we will consider $\cF$-metrics whose restriction to $M\subset X$ can be lifted to a $\Gamma$-invariant product-type $\Phi$-metric of the form
\begin{equation}
  g_{\Phi}= \frac{dx^{2}}{x^{4}}+ \frac{\Phi^{*}h}{x^{2}} +\kappa,
\label{it.1}\end{equation}
where $h$ is a $\Gamma$-invariant metric on $Y$ and $\kappa$ is a $\Gamma$-invariant family of metrics in the fibres of the fibration $\Phi$ that is lifted to a symmetric 2-tensor in the ambient space via a choice of $\Gamma$-invariant connection for the fibration $\Phi$. Both $h$ and $\kappa$ are allowed to depend smoothly on $x$.  Since $\Gamma$ is finite, such metrics are easy to construct: we can insure $h$ and $\kappa$  are $\Gamma$-invariant by averaging over $\Gamma$, while a $\Gamma$-invariant connection is obtained by taking the orthogonal complement of the vertical tangent bundle $T(\pa\tM/Y)$ with respect to a choice of $\Gamma$-invariant metric on $\pa\tM$.

Let also $\cE$ be a Clifford module for the bundle of Clifford algebras defined
by $({}^{\cF}TX,g_{\cF})$.  We assume $\cE$ is equipped with a Clifford connection and
that in $M=\pa M\times [0,\epsilon)_{x}\subset X$, this Clifford module is naturally identified with $\pi_{\pa}^{*}(\left.\cE\right|_{\pa M})$ where 
$\pi_{\pa}: \pa M \times [0,\epsilon)_{x}\to \pa M$ is the projection on the left factor.  Let $\eth_{X}\in \Psi^{1}_{\cF}(X;\cE)$ be the corresponding Dirac-type operator.  Let $\eth_{M}$ be its restriction to $M$ and $\widetilde{\eth}\in \Psi^{1}_{\Phi,\Gamma}(\widetilde{M};\widetilde{\cE})$ the differential operator which is the $\Gamma$-invariant lift to $\widetilde{M}$ where $\widetilde{\cE}$ is the lift of $\left. \cE\right|_{M}$ to $\widetilde{M}$.

 Since $X$ is 
even dimensional, $\cE$ comes with a natural $\bbZ_{2}$-grading which induces a decomposition $\cE= \cE^{+}\oplus \cE^{-}$.  
The Dirac operator $\eth_{X}$ is odd with respect to this grading, so decomposes in two parts,
\begin{equation}
   \eth_{X}= \left( 
   \begin{array}{cc}
     0 & \eth^{-}_{X} \\
     \eth_{X}^{+} & 0 
   \end{array}
  \right),  \quad 
  \eth_{X}^{\pm}: \CI(X;\cE^{\pm}) \to \CI(X; \cE^{\mp}).
\label{it.2}\end{equation}
There is a corresponding decomposition for $\widetilde{\eth}$.  In local coordinates near the boundary, the operator $\widetilde{\eth}^{+}$ takes the form
\begin{equation}
  \widetilde{\eth}^{+}= \gamma\left( x^{2}\frac{\pa}{\pa x} +\widetilde{\eth}_{0}\right)
  + \cl\left(\frac{e^{k}}{x}\right) \nabla_{x e_{k}} +R
\label{it.3}\end{equation}
where $\gamma$ is Clifford multiplication by $\frac{dx}{x^{2}}$, the 
$e_{1},\ldots, e_{p}$, $p=n-\ell-1$ are local orthonormal sections of 
$(TY,h)$ and  $\widetilde{\eth}_{0}\in \Psi^{1}(\pa \widetilde{M}/Y;\widetilde{\cE}_{0})$  is a family
of Dirac operators associated to the Clifford module $\widetilde{\cE}_{0}= 
\left.\widetilde{\cE}^{+}\right|_{\pa \widetilde{M}}$ and the family of metrics
$\kappa$.  Here, the Clifford multiplication on $\cE_{0}$ is given by
\begin{equation}
  T^{*}(\pa\widetilde{M}/Y) \ni \xi \mapsto -\gamma\cdot \xi \in \Cl_{g_{\Phi}}(\widetilde{M}).
\label{it.4}\end{equation}
Finally, $R\in x\Psi^{1}_{\Phi}(\widetilde{M};\widetilde{\cE})$ is a term that will not contribute to the normal operator.  
Via the identification
\begin{equation}
   \left. \widetilde{\cE}^{-}\right|_{\pa \widetilde{M}} \overset{-\gamma}{\longrightarrow} \widetilde{\cE}_{0} 
\label{it.5}\end{equation}
given by Clifford multiplication by $-\gamma$, the (Fourier transform of the)
normal operator of $\widetilde{\eth}^{+}$ can be seen as an element of  $\Psi^{1}_{\Phi-\sus}(\pa \widetilde{M};\widetilde{\cE}_{0})$ taking the form
\begin{equation}
  \widehat{N}_{\Phi}(\widetilde{\eth}^{+})= \widetilde{\eth}_{0}+ i\tau + i\gamma_{Y}.
\label{it.6}\end{equation}
That is, it is a family of operators on $T^{*}Y\times \bbR$ where at $(y,\delta,\tau)\in T^{*}Y\times \bbR_{\tau}$, $\gamma_{Y}$ 
denotes Clifford multiplication by $-\gamma\cdot\frac{\delta}{x}$.


Our choice of connection for the fibration $\Phi$ gives us a decomposition of
\begin{equation}
  {}^{\Phi}T\pa \widetilde{M}:= \{ v\in \left.{}^{\Phi}T\widetilde{M}
\right|_{\pa \widetilde{M}} \; | \; g_{\Phi}(v, x^{2}\frac{\pa}{\pa x})=0 \}
\label{it.9}\end{equation}
as ${}^{\Phi}T\pa \widetilde{M}= \Phi^{*}TY\oplus T(\pa \widetilde{M}/Y)$.  This gives
a corresponding decomposition of the Clifford algebra,
\begin{equation}
   \Cl({}^{\Phi}T\pa \widetilde{M})= \Phi^{*}\Cl_{h}(Y)\hat{\otimes}\Cl_{\kappa}(T(\pa\widetilde{M}/Y))
\label{it.10}\end{equation}
where $\hat{\otimes}$ is the graded tensor product.  

As in \cite{Albin-Rochon1}, we need to make an assumption in order to get a Fredholm operator.  
\begin{assumption}
There exists a $\Gamma$-invariant family of self-adjoint operators $A\in \Psi^{-\infty}(\pa\widetilde{M}/Y;\widetilde{\cE}_{0})$ anti-commuting with Clifford multiplication by odd sections of $\Phi^{*}\Cl_{h}(Y)$ and such that 
$\widetilde{\eth}_{0}+A$ is an invertible family.  A particularly natural example is when $A=0$ and the family $\widetilde{\eth}_{0}$ is itself invertible, see \S~\ref{qmtn.0} below.   
\label{ad.1}\end{assumption}

\begin{remark}
Depending on whether the fibres of the fibration $\Phi$ are even or odd dimensional,  the Clifford module $\widetilde{\cE}_{0}$ may or may not have a $\bbZ_{2}$-grading as a $\Cl_{\kappa}(T\widetilde{M}/Y)$ Clifford module.  When it does, asking the perturbation $A$ to anti-commute with Clifford multiplication by odd sections of $\Phi^{*}\Cl_{h}(Y)$ forces $A$ to be odd with respect to the $\bbZ_{2}$-grading of $\widetilde{\cE}_{0}$ (as a $\Cl_{\kappa}(T\widetilde{M}/Y)$ Clifford module).    
\label{ad.2}\end{remark}

Let $\rho\in \cS(T^{*}Y\times\bbR)$ be a real-valued $\Gamma$-invariant Schwartz function on $T^{*}Y\times \bbR\cong N^{*}Y$ equal to $1$ on the zero section.  Then there exists 
$\widetilde{Q}\in \Psi^{-\infty}_{\Phi,\Gamma}(\widetilde{M},\widetilde{\cE})$ such that 
\begin{equation}
  \widehat{N}_{\Phi}(\widetilde{Q})= \rho A.
\label{ad.3}\end{equation}
Let us denote by $Q$ the corresponding operator in 
\[
\Psi^{-\infty}_{\cF}(M;\cE^{+},\cE^{-})\subset \Psi^{-\infty}_{\cF}(X;\cE^{+},\cE^{-}).
\]

\begin{proposition}
The operator $\eth^{+}+Q$ is Fredholm and there exists $P\in \Psi^{-1}_{\cF}(X;\cE^{-},\cE^{+})$ such that 
\begin{equation*}
\begin{gathered}
     P(\eth^{+}_{X}+Q)-\Id_{\cE^{+}}= R_{+}\in x^{\infty}\Psi^{-\infty}(X;\cE^{+}), \\
     (\eth^{+}_{X}+Q)P- \Id_{\cE^{-}}= R_{-}\in x^{\infty}\Psi^{-\infty}(X;\cE^{-}).
\end{gathered}
\end{equation*}
For this reason, we regard  $\eth^{+}_{X}+Q$ as a \textbf{Fredholm perturbation} of $\eth^{+}_{X}$.\label{ad.4}\end{proposition}
\begin{proof}
We have
\[
   \widehat{N}_{\Phi}(\widetilde{\eth}^{+}+\widetilde{Q})= \widetilde{\eth}_{0}+\rho A +i\tau
   +i\gamma_{Y}.
\]
Using the anti-commuting relations of the Clifford multiplication, we get
\[
  \widehat{N}_{\Phi}(\widetilde{\eth}^{+}+\widetilde{Q})^{*} \widehat{N}_{\Phi}(\widetilde{\eth}^{+}+\widetilde{Q}) = (\widetilde{\eth_{0}}+\rho A)^{2} + \tau^{2} + \|\delta\|^{2}_{h}
\]
at $(y,\delta,\tau)\in T^{*}Y\times \bbR$.  Since $\widetilde{\eth_{0}}+A$ is an invertible self-adjoint family by assumption, this clearly implies $\widehat{N}_{\Phi}(\widetilde{\eth}^{+}+\widetilde{Q})$ is invertible with inverse in $\Psi^{-1}_{\sus-\Phi,\Gamma}(\pa\widetilde{M},\widetilde{\cE}_{0})$.  The result therefore follows from Remark~\ref{lcfr.1}.
\end{proof}

In appendix~C of \cite{Melrose-Rochon06}, a certain adiabatic calculus was introduced to relate fully elliptic $\Phi$-operators with fully elliptic cusp operator, which are fibred cusp operators for which the fibration on the boundary is given by mapping the entire boundary onto a point.  More precisely, in section~8 of
\cite{Melrose-Rochon06}, a natural construction associates to a fully elliptic operator $P\in \Psi_{\Phi}^{m}(\widetilde{M},E,F)$ an element
$P_{\ad}\in\Psi^{m}_{\Phi-\ad,\cusp}(\widetilde{M};E,F)$ of the corresponding 
adiabatic calculus which is fully elliptic.  Essentially, $P_{\ad}$ corresponds to 
a one parameter family 
\begin{equation}
      (0,+\infty)\ni \delta \mapsto P_{\delta}\in\Psi^{m}_{\cusp}(\widetilde{M};E,F)
\label{ad.5}\end{equation}
of cusp operators which are fully elliptic for $\delta$ small enough.  For such 
a small $\delta$,  the index of $P_{\delta}$ is the same as the one of $P$.  The family $P_{\delta}$ is not uniquely defined, but its homotopy class among fully elliptic operators is.  This is because $P_{\ad}$ makes precise the sense in which $P_{\delta}\to P$ as $\delta\to 0^{+}$.  Besides the usual symbol and normal operator, the operator $P_{\ad}$ has a third `symbol', the adiabatic normal operator, whose r\^ole is to relate the normal operator of $P$ with the normal operator of the family $P_{\delta}$.  

Strictly speaking, the discussion in \cite{Melrose-Rochon06} is for compact manifolds with boundary.  What is important however is the behavior near the boundary.  In that sense, it extends immediately to operators $P$ in $\Psi^{m}_{\Phi}(\widetilde{M};E,F)$ with `fully elliptic' meaning in that context that $P$ is elliptic near the boundary $\pa \widetilde{M}$ and has an invertible normal operator.  

We can apply this adiabatic construction to $\widetilde{\eth}^{+}+\widetilde{Q}$.  First
to $\widetilde{\eth}^{+}$ by considering the family of $\Gamma$-invariant metrics
for $x<\frac{\epsilon}{2}$ 
\begin{equation}
      \widetilde{g}_{\cusp}(\delta) = \frac{dx^{2}}{x^{4}}+ \frac{\Phi^{*}h}{(x+\delta)^{2}}+ \kappa
\label{ad.6}\end{equation}
  with the (adiabatic) limit $\delta\to 0^{+}$ giving back the metric $g_{\Phi}$.  This gives a corresponding family $g_{\cusp}(\delta)$ of cusp metrics on
  $M\times (0,\frac{\epsilon}{2})$ which can be extended to give a family of cusp metrics $g_{\cusp}(\delta)$ on $X$.

On $X\times [0,\nu)_{\delta}$, there is a natural vector bundle ${}^{\ad}TX$ such
that $\left. {}^{\ad}TX\right|_{X\times\{0\}}= {}^{\cF}TX$ and 
$\left. {}^{\ad}TX\right|_{X\times\{\delta\}}= {}^{\cusp}TX$ for
$\delta>0$.  If $X$ has a spin structure and $S_{\cF}$ is the corresponding spinor bundle associated to the metric $g_{\cF}$, then $\cE= S_{\cF}\otimes E$ for some smooth complex vector bundle $E\to X$.  If
$S_{\ad}$ is the spinor bundle associated to ${}^{\ad}TX\to X\times [0,\nu)_{\delta}$ with respect to the family of metrics \eqref{ad.6}, then we can define a Clifford module  on $X\times [0,\nu)_{\delta}$ with respect to the Clifford bundle of ${}^{\ad}TX$ by $\cE_{\ad}= S_{\ad}\otimes \pi^{*}E$, where $\pi: X\times [0,\nu)_{\delta}\to X$ is the natural projection.  The Clifford connection of $\cE$ corresponds to a choice of connection $\nabla^{E}$ for the bundle $E$.  Taking the pull-back connection on $\pi^{*}E$, we get in this way a natural choice of Clifford connection on $\cE_{\ad}$.  
When $X$ is not spin, we can define $\cE_{\ad}$ and its Clifford connection locally on $\cU\times [0,\nu)_{\delta}$ by choosing a spin structure on $\cU\subset X$, where $\cU$ is an open set over which $\left. {}^{\cF}TX\right|_{\cU}$ admits a spin structure.  These local definitions fit together to give a global Clifford module $\cE_{\ad}$ with Clifford connection on $X\times [0,\nu)_{\delta}$. 
This gives a family $\eth_{X}(\delta)$ of Dirac-type operators which together with $\eth_{X}$ fit to give, when restricted to $M$ and lifted to $\widetilde{M}$, a $\Gamma$-invariant element 
$\eth_{\ad}\in \Psi^{1}_{\Phi-\ad,\cusp}(\widetilde{M};\widetilde{\cE}_{\ad})$, where $\widetilde{\cE}_{\ad}$ is the lift of $\cE_{\ad}$ to $\widetilde{M}\times [0,\nu)_{\delta}$.   Our choices of Clifford module and Clifford connection insure that the adiabatic normal operator is the same as the one constructed in section~8 of \cite{Melrose-Rochon06} (\cf \cite{LMP}).

Proceeding as in section~8 of \cite{Melrose-Rochon06}, the perturbation
$\widetilde{Q}\in \Psi^{-\infty}_{\Phi,\Gamma}(\widetilde{M};\widetilde{\cE}^{+},\widetilde{\cE}^{-})$ can be extended to give an element 
$\widetilde{Q}_{\ad}\in \Psi^{-\infty}_{\Phi-\ad, \cusp}(\widetilde{M};\widetilde{\cE}^{+}_{\ad},\widetilde{\cE}^{-}_{\ad})$ which we can assume is $\Gamma$-invariant by averaging over $\Gamma$.  In particular, we get in this way 
a family $\widetilde{Q}(\delta)\in \Psi^{-\infty}_{\cusp,\Gamma}(\widetilde{M},\widetilde{\cE}^{+}_{\delta},\widetilde{\cE}^{-}_{\delta})$ of $\Gamma$-invariant
cusp operators that descends to $M$ and extends to $X$ to a family 
$Q(\delta)\in \Psi^{-\infty}_{\cusp}(X;\cE^{+}_{\delta},\cE^{-}_{\delta})$ of  cusp operators.   Here,
$\cE_{\delta}$ and $\widetilde{\cE}_{\delta}$ are the restriction of $\cE_{\ad}$ and $\widetilde{\cE}_{\ad}$ to $M\times\{\delta\}$ and $\widetilde{M}\times\{\delta\}$ respectively.
Notice that, under the identification
\[
      \left. \widetilde{\cE}^{-}_{\delta}\right|_{\pa \widetilde{M}} \overset{-\gamma}{\longrightarrow} 
      \widetilde{\cE}_{\delta,0}:=  \left. \widetilde{\cE}^{+}_{\delta}\right|_{\pa \widetilde{M}},
\]
we can insure $\widehat{N}_{\cusp}(\widetilde{Q}(\delta))$ is self-adjoint as an element of 
$\Psi^{-\infty}_{\sus}(\pa\widetilde{M};\cE_{\delta,0})$ by replacing $\widehat{N}_{cu}(\widetilde{Q}(\delta))$ by  $\frac{\widehat{N}_{cu}(\widetilde{Q}(\delta)) + \widehat{N}_{cu}(\widetilde{Q}(\delta))^{*}}{2}$ and making the corresponding changes for $\widetilde{Q}_{\ad}$.

From the results of section~8 of \cite{Melrose-Rochon06}, for $\delta>0$ sufficiently small, the family $(\eth^{+}_{X}+Q)(\delta)\in\Psi^{1}_{\cusp}(X;\cE(\delta))$ is fully elliptic (so Fredholm) with the same index as $\eth^{+}_{X}+ Q$.  
Let 
\[\widehat{N}_{\cusp}((\widetilde{\eth}^{+}+\widetilde{Q})(\delta))(\tau)=
\widetilde{\eth}_{0}(\delta) + \widehat{N}(\widetilde{Q})(\tau) + i\tau,
\quad \delta>0, \; \tau\in \bbR,
\]
be the corresponding Fourier transform of the normal operator for this family.  Thus, it is a family of $\Gamma$-invariant suspended operators on $\pa\widetilde{M}$.  Evaluating at $\tau=0$, we get a family of self-adjoint invertible operators $\widetilde{\eth}_{0}(\delta)+\widehat{N}_{\cusp}(\widetilde{Q}(\delta))(0)$ with 
a well-defined eta invariant
\begin{equation}
 \widetilde{\eta}(\delta)= \eta(\widetilde{\eth}_{0}(\delta)+ \widehat{N}_{\cusp}(\widetilde{Q}(\delta))(0))
 \label{ad.7}\end{equation}
If we look at the corresponding family of operators $\eth_{0}(\delta)+\widehat{N}_{\cusp}(Q(\delta))(0)$ on $\pa M$, we also get an eta invariant
\begin{equation}
    \eta(\delta)= \eta(\eth_{0}(\delta)+ \widehat{N}_{\cusp}(Q(\delta))(0))
\label{ad.8}\end{equation}
\begin{definition}
The \textbf{rho invariant} of the invertible perturbation $\widetilde{\eth}_{0}+A$ of the family $\widetilde{\eth}_{0}$ in Assumption~\ref{ad.1} is
\[
        \rho_{A}= \frac{\widetilde{\eta}(\delta)}{|\Gamma|}- \eta(\delta)
\]
where $\delta> 0$ is taken small enough so that $\widetilde{\eta}(\delta)$ and
$\eta(\delta)$ are well-defined.
\label{ad.9}\end{definition}
By looking at the local variation of $\widetilde{\eta}(\delta)$ and $\eta(\delta)$, one can check
that $\rho_{A}$ does not depend on the choice of $\delta$.  Moreover, $\rho_{A}$ does not depend on the choices involved in the construction of $\widetilde{\eth}_{0}(\delta)+ \widehat{N}_{\cusp}(\widetilde{Q}(\delta))(0)$ since different choices would lead to operators that could be connected by a smooth path of invertible elliptic self-adjoint operators of the same form.  Thus, $\rho_{A}$ only
depends on the choice of perturbation $A$ in Assumption~\ref{ad.1}.

This rho invariant is the new ingredient needed to get an index formula for the Fredholm operator of Proposition~\ref{ad.4}.

\begin{theorem}
The index of the Fredholm operator $\eth^{+}_{X}+Q\in \Psi^{1}_{\cF}(X;\cE^{+},\cE^{-})$ in Proposition~\ref{ad.4} is given by
\[
   \Ind(\eth^{+}_{X}+Q)=
   \int_{X}\widehat{A}(X;g_{\cF})\Ch_{g_{\cF}}'(\cE)- \frac{1}{|\Gamma|}\int_{Y}\widehat{A}(Y,h)\widehat{\eta}(\widetilde{\eth}_{0}+A) + \frac{\rho_{A}}{2}
\]
where $\Ch_{g_{\cF}}'(\cE)$ denotes the Chern character form associated to the twisting curvature of $\cE$ (see \cite{BGV}) and $\widehat{\eta}(\widetilde{\eth}_{0}+A)$ is the 
eta form of the family of invertible self-adjoint operators $(\widetilde{\eth}_{0}+A)$ as described in \cite{Albin-Rochon1}, but using the convention of \cite{Albin-Rochon2} to avoid $2\pi i$ factors in the formula.  
\label{ad.10}\end{theorem}
\begin{proof}
Our approach is inspired from \cite{LMP} and consists in taking the adiabatic limit of the index formula for $\eth^{+}_{X}(\delta)+Q(\delta)$.  On the one hand, we know by construction that 
\[
        \Ind(\eth^{+}_{X}+Q)= \Ind(\eth^{+}_{X}(\delta) +Q(\delta))
\] 
for $\delta>0$ sufficiently small.   On the other hand, according to
\cite{Albin-Rochon1}, we have the following formula for the index of $\eth^{+}_{X}(\delta) +Q(\delta)$,
\[
  \Ind(\eth^{+}_{X}(\delta)+Q(\delta))= \int_{X} \widehat{A}(X,g_{\cusp}(\delta))\Ch_{g_{\cusp}(\delta)}'(\cE_{\delta}) - \frac{\eta(\delta)}{2}.
\]
This can be rewritten as 
\begin{equation}
  \Ind(\eth^{+}_{X}(\delta)+Q(\delta))= 
\int_{X} \widehat{A}(X,g_{\cusp}(\delta))\Ch'_{g_{\cusp}(\delta)}(\cE_{\delta})
  - \frac{\widetilde{\eta}(\delta)}{2|\Gamma|} + \frac{\rho_{A}}{2}. 
\label{ad.11}\end{equation}
The results then follow by taking the limit as $\delta\to 0^{+}$ and using the following two adiabatic limits,
\begin{equation*}
\begin{gathered}
\lim_{\delta\to 0^{+}} \int_{X}\widehat{A}(X,g_{\cusp}(\delta))\Ch_{g_{\cusp}(\delta)}'(\cE_{\delta})= \int_{X}\widehat{A}(X,g_{\cF})\Ch_{g_{\cF}}'(\cE), \\
\lim_{\delta\to 0^+} \frac{\widetilde{\eta}(\delta)}{2}= 
  \int_{Y}\widehat{A}(Y,h)\widehat{\eta}(\eth_{0}+A),
\end{gathered}
\end{equation*}
which are established in  the next lemma and theorem.
\end{proof}

\begin{remark}
If $\widehat{N}_{\Phi}(\widetilde{\eth}^+)$ is invertible and $Q=0$, then one can use instead the adiabatic limit of the eta invariant recently obtained in \cite{Goette2011} to obtain a slightly different index formula.   
\label{Goette.1}\end{remark}

\begin{lemma}  We have the following adiabatic limit:
\[
\lim_{\delta\to 0^{+}} \int_{X}\widehat{A}(X,g_{\cusp}(\delta))\Ch_{g_{\cusp}(\delta)}'(\cE_{\delta})= \int_{X}\widehat{A}(X,g_{\cF})\Ch_{g_{\cF}}'(\cE).
\]
\label{ad.12}\end{lemma}
\begin{proof}
Clearly, the result will follow if for each $p$ in $\pa\widetilde{M}$, we can
find an open neighborhood $\cU$ of $p$ in $\widetilde{M}$ such that
\[
 \lim_{\delta\to 0^{+}} \int_{\cU}\widehat{A}(\widetilde{M},\widetilde{g}_{\cusp}(\delta))\Ch_{\widetilde{g}_{\cusp}(\delta)}'(\cE_{\delta})= \int_{\cU}\widehat{A}(\widetilde{M},g_{\Phi})\Ch_{g_{\Phi}}'(\cE).
\] 
If we take $\cU$ sufficiently small, we can assume ${}^{\Phi}T\widetilde{M}$ is trivial over $\cU$, so that in particular, over $\cU$, it admits a spin structure.  Let $S_{\cU}$ be the corresponding spinor bundle.  There is a complex vector bundle $E$ over $\cU$ inducing a decomposition
   $\cE= S_{\cU}\otimes E$
such that the Clifford connection of $\cE$ is of the form
\begin{equation}
    \nabla^{\cE}= \nabla^{S_{\cU}}\otimes 1+ 1\otimes \nabla^{E}
\label{ad.12b}\end{equation}
where $\nabla^{S_{\cU}}$ is the Clifford connection of $S_{\cU}$ and 
$\nabla^{E}$ is some connection on $E$.  The twisting curvature appearing in
$\Ch_{g_{\cusp}(\delta)}'(\cE)$ is then precisely the curvature of $\nabla^{E}$ so that 
\[ 
        \Ch_{g_{\cusp}(\delta)}'(\cE)= \Ch(\nabla^{E})= e^{\frac{i}{2\pi}(\nabla^{E})^{2}}\;\; \mbox{on} \;\cU.
\]
Thus, on $\cU$, we have
\begin{equation}
  \widehat{A}(X,\widehat{g}_{\cusp}(\delta))  \Ch_{g_{\cusp}(\delta)}'(\cE)= 
 \widehat{A}(X,\widehat{g}_{\cusp}(\delta))\Ch(\nabla^{E}).
\label{ad.12c}\end{equation}
Since $\Ch(\nabla^{E})$ is uniformly bounded and in fact has been chosen to be independent of $\delta$, the lemma will follow if we can show that 
$ \widehat{A}(X,g_{\cusp}(\delta))$ is uniformly bounded as a section of
$\Lambda^{*}(T^{*}\cU)$ as $x$ and $\delta$ approach zero.  

To do this, we can follow the approach of \cite{LMP}.  Since the adiabatic limit we are taking is only slightly different than the one of \cite{LMP}, the argument will be the same modulo minor changes.  We will include it for the sake of completeness.

Recall that if $\omega^{0},\ldots,\omega^{n-1}$ is any orthonormal set of one-forms, then the connection one-forms are given by the unique solution to
\begin{equation}
      d\omega^{i}= \omega^{j}\wedge \omega^{i}_{j}, \quad \omega^{i}_{j}= -\omega^{j}_{i},
\label{ad.12d}\end{equation}
where summation on repeated indices is intended.
The curvature two-forms are then given by
\begin{equation}
  \Omega^{j}_{i}= d\omega^{j}_{i}- \omega^{k}_{i}\wedge \omega_{k}^{j}.
\label{ad.12e}\end{equation}
We will choose our coframe as follows.  We first pick $\omega^{0}= \frac{dx}{x^{2}}$.  For $1\le \alpha\le k$, where $k=\dim Y$, we choose an orthonormal coframe $\hat{\omega}^{\alpha}$ in a neighborhood of $\Phi(p)$ in $(Y,h)$.  For $k+1\le \mu\le n-1$, we also pick a local vertical orthonormal coframe $\omega^{\mu}$  for the metric $\kappa$ near the point $p$.
It is vertical in the sense that it is orthogonal to horizontal forms with respect to the metric $g_{\cusp}(\delta)$.  As in \cite{LMP}, we will use the convention that Roman indices $i,j, \ldots$ vary between $0$ and $n$, while Greek indices $\alpha,\beta,\ldots$ vary between $1$ and $k$ and Greek indices $\mu,\nu,\ldots$ vary between $k+1$ and $n-1$.  If we set 
\begin{equation}
  \omega^{\alpha}=\frac{\Phi^{*}\hat{\omega}^{\alpha}}{x+\delta},
\label{ad.12f}\end{equation}
then $\omega^{0},\ldots,\omega^{n-1}$ is an orthonormal coframe in some small neighborhood $\cU$ of $p$ with respect to the metric $g_{\cusp}(\delta)$.  We first compute that 
\begin{equation}
\begin{aligned}
 d\omega^{0} &= 0 \\
             &\equiv \omega^{\alpha}\wedge \omega^{0}_{\alpha}+ \omega^{\mu}\wedge\omega^{0}_{\mu} \\
d\omega^{\alpha} &= -\frac{dx}{(x+\delta)^{2}}\wedge \Phi^{*}\hat{\omega}^{\alpha} + \frac{dx}{x+\delta}\wedge \Phi^{*}(\hat{\omega}^{\alpha})'+ \omega^{\beta}\wedge\Phi^{*}\hat{\omega}^{\alpha}_{\beta} \\
   &\equiv \omega^{0}\wedge\omega^{\alpha}_{0} + \omega^{\beta}\wedge\omega^{\alpha}_{\beta}+ \omega^{\mu}\wedge\omega^{\alpha}_{\mu} \\
d\omega^{\mu}&= dx\wedge( \omega^{\mu})' + (x+\delta)\omega^{\alpha}\wedge E^{\mu}_{\alpha} + \omega^{\nu}\wedge E^{\mu}_{\nu}  \\
  &\equiv \omega^{0}\wedge \omega^{\mu}_{0} + \omega^{\alpha}\wedge\omega^{\mu}_{\alpha} + \omega^{\nu}\wedge \omega^{\mu}_{\nu}
\end{aligned}
\label{ad.12g}\end{equation}
where $'$ denotes differentiation with respect to $x$ and $E^{j}_{i}$, here and below, denotes forms which are uniformly bounded in $x$ and $\delta$.   To study the asymptotic behavior of the unique solution to \eqref{ad.12g}, consider first in each slice $\pa\tM\times \{x\}$ the corresponding equation 
\begin{equation}
       du^{i}= u^{j}\wedge u^{i}_{j}, \quad u^{i}_{j}= -u^{j}_{i},
\label{nad.1}\end{equation}
for the orthonormal coframe $u^{1},\ldots,u^{n-1}$ of the metric $\Phi^{*}h(x)+\kappa(x)$ defined by
$u^{\alpha}= \Phi^{*}\hat{\omega}^{\alpha}$ and $u^{\mu}= \omega^{\mu}$.  Clearly, the connection one-forms $u^{i}_{j}$ given by the unique solution of \eqref{nad.1} are uniformly bounded in $x$ and independent of $\delta$.  If now we consider instead the metric 
$\frac{\Phi^{*}h(x)}{(x+\delta)^{2}} +\kappa(x)$ on the slice $\pa\tM\times \{x\}$, then $v^{1},\ldots,v^{n-1}$ with $v^{\alpha}= \frac{\Phi^{*}\omega^{\alpha}}{x+\delta}$, $v^{\mu}= \omega^{\mu}$ is an associated orthonormal coframe.  The connection one-forms $v^{i}_{j}$ obtained by solving the equation
\begin{equation}
   dv^{\alpha}= v^{\beta}\wedge v^{\alpha}_{\beta}+ v^{\mu}\wedge v^{\alpha}_{\mu}, \quad
   dv^{\mu}= v^{\alpha}\wedge v^{\mu}_{\alpha}+ v^{\nu}\wedge v^{\mu}_{\nu}, \quad v^{i}_{j}= -v^{j}_{i}, 
\label{nad.2}\end{equation}   
are then seen to satisfy
\begin{equation}
   v^{\mu}_{\alpha}= (x+\delta)u^{\mu}_{\alpha}, \quad v^{\mu}_{\nu}=u^{\mu}_{\nu}, \quad 
     v^{\alpha}_{\beta}= \Phi^{*}\hat{\omega}^{\alpha}_{\beta}+ (x+\delta)^{2}E^{\alpha}_{\beta}.
\label{nad.3}\end{equation}
From there, one can check that the connection one-forms solving \eqref{ad.12g} take the form
\begin{equation}
\begin{gathered}
  \omega^{\mu}_{0}= x^{2}E^{\mu}_{0}, \quad \omega^{\mu}_{\nu}= v^{\mu}_{\nu}+ e^{\mu}_{\nu}dx,  \quad \omega^{\mu}_{\alpha}= v^{\mu}_{\alpha}+ (x+\delta)e^{\mu}_{\alpha} dx, \\
  \omega^{\alpha}_{0}= \frac{-x^{2}}{(x+\delta)^{2}}\Phi^{*}\hat{\omega}^{\alpha} + \frac{x^{2}E^{\alpha}_{0}}{x+\delta}+ x^{2}(x+\delta)E^{\alpha}_{0}, \quad \omega^{\alpha}_{\beta}= v^{\alpha}_{\beta}+ e^{\alpha}_{\beta}dx,
  \end{gathered} 
\label{nad.4}\end{equation}
where $e^{i}_{j}$ denotes functions uniformly bounded in $x$ and $\delta$.  Thus, we have obtained
\begin{equation}
\begin{gathered}
  \omega^{\alpha}_{0}= - \frac{x^{2}}{(x+\delta)^{2}} \Phi^{*}\hat{\omega}^{\alpha} + \frac{x^{2} E^{\alpha}_{0}}{x+\delta}, \quad \omega^{\mu}_{0}= x^{2}E^{\mu}_{0}, \\
\omega^{\beta}_{\alpha}= \Phi^{*}\hat{\omega}^{\beta}_{\alpha}+(x+\delta)^{2}E^{\beta}_{\alpha}+ e^{\beta}_{\alpha}dx , \quad \omega^{\nu}_{\mu} =E^{\nu}_{\mu}, \quad \omega^{\mu}_{\alpha}= (x+\delta)E^{\mu}_{\alpha}.
\end{gathered}
\label{ad.12h}\end{equation}
Using \eqref{ad.12e}, one can then compute the curvature forms to get
\begin{equation}
\begin{aligned}
\Omega^{\alpha}_{0}&= dx\wedge \left( -\frac{2x\delta}{(x+\delta)^{3}} \Phi^{*}\hat{\omega}^{\alpha} +F^{\alpha}_{0}\right) + xF^{\alpha}_{0}, \\
\Omega^{\mu}_{0}&= (x+\delta)F^{\mu}_{0}, \\
\Omega^{\beta}_{\alpha}&= F^{\beta}_{\alpha}, \\
\Omega^{\mu}_{\nu}&= F^{\mu}_{\nu}, \\
\Omega^{\mu}_{\alpha}&= dx\wedge F^{\mu}_{\alpha} + (x+\delta)F^{\mu}_{\alpha},
\end{aligned}
\label{ad.12i}\end{equation}
where again the notation $F^{j}_{i}$ denotes terms uniformly bounded in $x$ and $\delta$, but not necessarily smooth at $x=\delta=0$.  As in \cite{LMP}, only the computation of the first term is more delicate,
\begin{equation}
\begin{aligned}
\Omega^{\alpha}_{0}&= d\omega^{\alpha}_{0}- \omega^{\beta}_{0}\wedge \omega^{\alpha}_{\beta}- \omega^{\mu}_{0}\wedge\omega^{\alpha}_{\mu}  \\
&= -d\left(\frac{x^{2}}{(x+\delta)^{2}}\right) \Phi^{*}\hat{\omega}^{\alpha} +
 dx\wedge F^{\alpha}_{0}- \frac{x^{2}}{(x+\delta)^{2}} \Phi^{*}(d\hat{\omega}^{\alpha}- \hat{\omega}^{\beta}\wedge \hat{\omega}^{\alpha}_{\beta}) + xF^{\alpha}_{0},
\end{aligned}
\label{ad.12j}\end{equation}
but since  $d\hat{\omega}^{\alpha}- \hat{\omega}^{\beta}\wedge \hat{\omega}^{\alpha}_{\beta}=0$ by the structure equation for the connection one-forms on $(Y,h)$, the third term vanishes so that $\Omega^{\alpha}_{0}$ is of the claimed form.

Now, the form $\widehat{A}(X,g_{\cusp}(\delta))$ is locally a combination of terms of the form 
\begin{equation}
  \sum \Omega_{i_{1}}^{i_{2}}\Omega_{i_{2}}^{i_{3}}\cdots \Omega^{i_{1}}_{i_{m}}.
\label{ad.12k}\end{equation}
According to \eqref{ad.12i}, the only way we could get an unbounded term is if $\Omega^{\alpha}_{0}$ appears, in fact only the first term $-\frac{2x\delta}{(x+\delta)^{3}} dx\wedge \hat{\omega}^{\alpha}$ would potentially create a problem.  Since it involves a $dx$ factor, it can appears at most once in each terms involved in \eqref{ad.12k} and we can assume it appears as the first curvature term there.  In that case, the last term would be $\Omega^{0}_{\beta}$ or $\Omega^{0}_{\mu}$.  If it is $\Omega^{0}_{\beta}$, only $xF^{0}_{\beta}$ will contribute since the singular term of $\Omega^{\alpha}_{0}$ already contains a $dx$ factor.  For $\Omega^{0}_{\mu}$, we already computed that $\Omega^{0}_{\mu}=(x+\delta)F^{\mu}_{0}$.  Thus, in both cases, we have a vanishing factor compensating for the singular term of $\Omega^{0}_{\alpha}$, so that $\widehat{A}(X,g_{\cusp}(\delta))$ is uniformly bounded in $x$ and $\delta$.  This is easily seen to imply the result.

\end{proof}

\begin{theorem}[Adiabatic limit]
For the family $\widetilde{\eth}_{0}(\delta)+ \widetilde{Q}(\delta)$ described above, 
\[
    \lim_{\delta\to 0^{+}}\frac{\widetilde{\eta}(\delta)}{2}=
     \int_{Y}\widehat{A}(Y,h)\widehat{\eta}(\widetilde{\eth}_{0}+A).
\]
\label{ad.13}\end{theorem}
\begin{proof}
Because of Lemma~\ref{ad.12}, or rather, its analog on $\widetilde{M}$, such
a formula for the adiabatic limit of the eta invariant is basically equivalent to the index formula of \cite{Albin-Rochon1}.  Indeed, changing the metrics $\widetilde{g}_{\cusp}(\delta)$ near $x=\frac{\epsilon}{2}$ to be of product-type,
we can then attach an infinite cylindrical end at $x=\frac{\epsilon}{2}$.  One way to achieve this is to extend the metric $\widetilde{g}_{\cusp}(\delta)$ in the region where $\frac{\epsilon}{2}<x<\epsilon$ by the metric
\begin{equation}
   \frac{d\rho^{2}}{\rho^{4}}+ \widetilde{g}_{\frac{\epsilon}{2}}
\label{ad.13b}\end{equation}
where $\widetilde{g}_{\frac{\epsilon}{2}}$, which is independent of $\delta$, is the restriction of $\widetilde{g}_{\cusp}(\delta)$ to the hypersurface $x=\frac{\epsilon}{2}$ and $\rho\in \CI(\pa\widetilde{M}\times [0,\epsilon])$ is a  
boundary defining function equal to $x$ when $x< \frac{3\epsilon}{4}$ and 
$\epsilon-x$ when $x>\frac{7\epsilon}{8}$.   

  With respect to this metric on $\pa\widetilde{M}\times (0,\epsilon)$, the operator $\widetilde{\eth}(\delta)$ can be extended to a cusp Dirac-type operator 
  $\widetilde{D}(\delta)$ on $\pa\widetilde{M}\times[0,\epsilon]$.  Since the normal operator at $x=0$ of $\widetilde{D}^{+}(\delta)$ admits an invertible perturbation 
$\widehat{N}_{\cusp}(\widetilde{D}^{+}(\delta)+ \widetilde{Q}(\delta))$, so does
the normal operator at $x=\epsilon$.  Thus, extending the perturbation $\widetilde{Q}(\delta)$ adequately in the cylindrical end that was added in such a way that the normal operator of $\widetilde{D}^{+}(\delta)+\widetilde{Q}(\delta)$  at $x=\epsilon$ is invertible and independent of $\delta$, we get in this way a family of fully elliptic 
cusp operators $(\widetilde{D}^{+}(\delta)+\widetilde{Q}(\delta))\in \Psi^{1}_{\cusp}(\pa\widetilde{M}\times [0,\epsilon];\cE^{+},\cE^{-})$.  In particular, this is a Fredholm operator and we have

\begin{equation}
 \Ind(\widetilde{D}^{+}(\delta)+ \widetilde{Q}(\delta))=
  \int_{\pa \widetilde{M}\times [0,\epsilon]} \widehat{A}(\widetilde{g}_{\cusp}(\delta))\Ch_{\widetilde{g}_{\cusp}(\delta)}'(\cE_{\delta}) - \frac{\widetilde{\eta}(\delta)}{2}
   - \frac{\widetilde{\eta}_{\epsilon}}{2}
\label{ad.14}\end{equation}
by the index formula of \cite{Albin-Rochon1}, where $\widetilde{\eta}_{\epsilon}$ is the eta invariant coming from the normal operator of $\widetilde{\eth}^{+}(\delta)+\widetilde{Q}(\delta)$ at $x=\epsilon$.  On the other hand,  we have also
\begin{equation}
\begin{aligned}
 \Ind(\widetilde{D}^{+}(\delta)+ \widetilde{Q}(\delta)) &=
\Ind(\widetilde{D}^{+}+\widetilde{Q}) \\
   &= 
\int_{\pa \widetilde{M}\times [0,\epsilon]}
\widehat{A}(g_{\Phi})\Ch_{g_{\Phi}}'(\cE) - \int_{Y}\widehat{A}(Y,h)\widehat{\eta}(\widetilde{\eth}_{0}+A) -
\frac{\widetilde{\eta}_{\epsilon}}{2},
\end{aligned}
\label{ad.15}\end{equation}
using the index formula of \cite{Albin-Rochon1}.  

Thus, since we assume that  the metrics $g_{\cusp}(\delta)$ stay fixed in the region where $x\ge \frac{\epsilon}{2}$  and that the normal operator $\widetilde{D}^{+} + \widetilde{Q}$ at $x=\epsilon$ and its eta invariant $\widetilde{\eta}_{\epsilon}$ are independent of $\delta$,
taking the limit $\delta\to 0^{+}$ in \eqref{ad.14} and using Lemma~\ref{ad.12} and equation~\eqref{ad.15}, the result follows.

\end{proof}

\begin{remark}
The case $\Gamma=\{0\}$ and $A=0$ with $\eth_{0}$ an invertible family of Dirac-type operators gives back the formula of Bismut and Cheeger \cite{Bismut-Cheeger} for the adiabatic limit of the eta invariant.  Moreover, as it is clear from the proof of the theorem, the total space of the fibration does not have to be the total boundary of a compact manifold with boundary.  
\label{ad.16}\end{remark}
\begin{remark}
Theorem~\ref{ad.13} should be compared with the recent result of Sebastian Goette \cite{Goette2011}, where more general foliations are considered, namely Seifert fibrations on orbifolds that are not necessarily good.  On the other hand, invertible perturbations of the family of Dirac-type operators are not considered there.  The formula in \cite{Goette2011} is different from the one in Theorem~\ref{ad.13} and provides an alternative point of view.  
\label{Goette.2}\end{remark}

\section{Some applications} \label{qmtn.0}

Let us now give a way of obtaining Dirac operators satisfying Assumption~\ref{ad.1}.
Assume that the $\cF$-tangent bundle ${}^{\cF}TX$ is spin as well as ${}^{\cF}N\pa X$ and $T\cF$.  Fix an orientation and a spin structure on ${}^{\cF}TX$ and ${}^{\cF}N\pa X$.  A choice of $\cF$-metric
whose restriction to $M$ lift to a $\Gamma$-invariant $\Phi$-metric as in \eqref{it.1} induces a decomposition 
\[
       \left.{}^{\cF}TX\right|_{\pa X}\cong {}^{\cF}N\pa X \oplus T\cF,
\]
which automatically induces a spin structure on $T\cF$.  Let $S_{X}$, $S_{N}$
and $S_{T\cF}$ denote the corresponding spinor bundles.   As a Clifford module, we can take $\cE=S_{X}$ equipped with its canonical  Clifford connection and get in this way the Dirac operator $\eth_{X}\in \Psi^{1}_{\cF}(X;S_{X})$.
In this case, the decomposition \eqref{it.10} also induces a decomposition
of spinor bundles
\begin{equation}
      \widetilde{S}_{0}=\widetilde{\cE}_{0}= \Phi^{*}S_{Y}\hat{\otimes}
S_{\pa\widetilde{M}/Y}.
\label{ad.17b}\end{equation}
Under this decomposition, the family of operators $\widetilde{\eth}_{0}$ on $\pa\widetilde{M}$ takes the form
\begin{equation}
    \widetilde{\eth}_{0}= \Id_{\Phi^{*}S_{Y}}\hat{\otimes} \eth_{\kappa}
\label{ad.17c}\end{equation}
where $\eth_{\kappa}\in \Psi^{1}(\pa\widetilde{M}/Y;S_{\pa\widetilde{M}/Y})$ is the family of Dirac operator associated to the spinor bundle $S_{\pa\widetilde{M}/Y}$ and the family of metrics $\kappa$.  By Lichnerowicz formula \cite{Lichnerowicz}, if we assume the family of metrics $\kappa$ has positive scalar curvature, then the families $\eth_{\kappa}$ and $\widetilde{\eth}_{0}$ will be invertible, so that Assumption~\ref{ad.1} is automatically satisfied with $A=0$.

\begin{example}
Consider the manifold $\tX= \bbR^{2m-1}\times \bbS^3$ with the Euclidean metric on $\bbR^{2m-1}$ and the standard round metric on the $\bbS^3$ factor. Let $k\in \bbN$ be an odd number and let $\bbZ_k\subset \SU(2)=\bbS^3$ act on $\bbS^3$ by right multiplication and on $\bbR^{2k-1}$ by letting its generator act by a rotation $R$ such that $R^k$ is the identity (\cf Example~\ref{sm.2}).  Then $\bbZ_k$ acts freely and properly discontinuously on $\tX$ by isometries.  This induces an $\cF$-metric $g_{\cF}$ on the quotient $X=\tX/\bbZ_k$.  Since $k$ is odd, the map $q^*: H^{*}(X;\bbZ_2)\to H^{*}(\tX;\bbZ_2)$ is injective (see Proposition~3G.1 in \cite{Hatcher}).  In particular, $X$ is spin and has only one spin structure.  The corresponding Dirac operator satisfies Assumption~\ref{ad.1} with $A=0$.  By Lichnerowicz formula, its index vanishes.  In fact all the terms in the index formula of Theorem~\ref{ad.10} are zero.  If $2m+2$ is a multiple of $4$, we can obtain an example with a non-vanishing index by taking the connected sum $W=X\# V$ with a compact spin manifold V with non-vanishing $\hat{A}$-genus. Keeping the same metric near infinity, the index of the Dirac operator on $W$ is then given by the $\hat{A}$-genus of $V$.         
\label{ad.19}\end{example} 

\begin{example}
One can go through the previous example with $\bbS^3$ replaced by $\bbS^1$.  Provided we choose the spin structure on $X= (\bbR^{2m-1}\times \bbS^1)/\bbZ_k$ so that the induced spin structure on $\bbS^1$ is the non-trivial one, we will have that $\eth_{\kappa}$ will be invertible, so that the Dirac operator on $X$ will be Fredholm, in fact invertible.  Again, one can get an example of Dirac operator with non-vanishing index by taking the connected sum with a spin manifold having a non-vanishing $\hat{A}$-genus.     
\label{ad.20}\end{example}


As explained in Example~\ref{multiT.1}, a natural way of obtaining an $\cF$-metric is to look at quotients of multi-Taub-NUT metrics.  Recall that the multi-Taub-NUT metric of type $A_{k-1}$ can be constructed explicitly using the Gibbons-Hawking ansatz \cite{Gibbons-Hawking}.  Let $\pi: \tX_0\to \bbR^3\setminus \{p_1,\ldots,p_k\}$ be the principle $\bbS^1$-bundle whose first Chern class is $-1$ when restricted to small spheres about the `monopole' points $\{p_1,\ldots, p_k\}$.  Consider the function 
\[
      V= 1+ \frac{1}{2}\sum_{i=1}^k \frac{1}{|x-p_i|}
\]
and equip the $\bbS^1$-bundle $\tX_0$ with the connection $1$-form $\alpha$ such that $d\alpha= \pi^*(*dV)$.  Then the metric
\begin{equation}
     \widetilde{g}_{\ALF}= V( dx_1^2+ dx_2^2 + dx^3) + V^{-1} \alpha^2,
\label{qmtn.1}\end{equation}  
on $\tX_0$ can be smoothly extended to a hyperk\"ahler metric on a manifold $\tX$.  The difference $\tX\setminus \tX_0$ consists of $k$ points $\{\tilde{p}_1,\ldots,\tilde{p}_k\}$ and the map $\pi$ on $\tX_0$ can be extended to a map $\pi:\tX\to \bbR^3$ such that $\pi(\tilde{p}_i)=p_i$.  The metric $\widetilde{g}_{\ALF}$ is an example of fibred boundary metric with fibration at infinity given by a circle bundle $\Phi:\pa \tX\to \bbS^2$ of degree $k$.  Here, we also denote by $\tX$  the manifold with boundary obtained from $\tX$ by adding a suitable boundary at infinity.  

The fact $\widetilde{g}_{\ALF}$ is hyperk\"ahler implies in particular that the manifold $\tX$ is spin.  Since $\tX$ is simply-connected, it has in fact a unique spin structure.  Let $\widetilde{\eth}_{\ALF}$ be the Dirac operator specified by this spin structure and the metric $\widetilde{g}_{\ALF}$.  
\begin{lemma}
The unique spin structures on $\tX$ and $\bbS^2$ induce the trivial spin structure on the $S^1$ fibres.  In particular, the Dirac operator $\widetilde{\eth}_{\ALF}$ is not Fredholm.
\label{qmtn.2}\end{lemma}      
\begin{proof}
Suppose for a contradiction that the spin structure induced on the $\bbS^1$ fibres is the non-trivial one.  Since the non-trivial spin structure bounds the disk, this means the spin structure on $\tX$ can be extended to the closed manifold $\tV=W\cup_{\pa\tX} \tX$ where $W$ is the disk bundle of degree $k$ on $\bbS^2$ with $\pa W= \pa \tX$.  In particular, there is an embedded sphere in $\tV$ of self-intersection $-1$  that contains $\tilde{p}_1$ and intersects the base of the disk bundle $W$ at one point.  By the criterion of \cite[Theorem~2.10]{LM}, such a curve cannot exist if $\tV$ has a spin structure.  To avoid a contradiction, we must conclude the induced spin structure on the $\bbS^1$ fibres is the trivial one.   
 
\end{proof}

Even if $\widetilde{\eth}_{\ALF}$ is not Fredholm, it is still possible to perturb it by a smoothing operator to make it Fredholm.  Indeed, the family of Dirac operators $\eth_{\kappa}$ on the $\bbS^1$ fibres of $\Phi: \pa\tX\to \bbS^2$ all have one-dimensional kernels that patch together to form a trivial line bundle over $\bbS^2$.  Let $A$ be the orthogonal projection onto this kernel bundle with respect to the $L^2$-norm specified by the family of metrics $\kappa$.  This can be extended to a perturbation as in Assumption~\ref{ad.1} by asking it to anti-commute with Clifford multiplication by odd sections of $\Phi^{*}\Cl_{h}(Y)$.  If $\tS= \tS^+\oplus \tS^-$ is the spinor bundle on $\tX$, then taking $\tQ\in \Psi^{-\infty}(\tX;\tS^+,\tS^-)$ such that    
\begin{equation}
       \widehat{N}_{\Phi}(\tQ)= \rho A,
\label{qmtn.3}\end{equation}
where $\rho\in \cS(T^{*}\bbS^2\times \bbR)$ is a real-valued Schwartz function equal to $1$ on the zero section, we obtain a Fredholm perturbation $\widetilde{\eth}_{\ALF}^+ +\tQ$ of $\widetilde{\eth}^+_{\ALF}$.  Using the index theorem of \cite{Albin-Rochon1}, we can compute its index.  

\begin{proposition}
The index of $\widetilde{\eth}_{\ALF}^+ +\tQ$ is zero.  
\label{qmtn.4}\end{proposition}
\begin{proof}
A change of metric that does not change the Normal operator of $\widetilde{\eth}_{\ALF}$ will not change the index.  This means that for the purpose of computing the index, we can assume the metric $\widetilde{g}_{\ALF}$ is of product-type, that is, of the form
\begin{equation}
    \widetilde{g}_{\ALF}= \frac{dx^{2}}{x^{4}}+ \frac{\Phi^{*}h}{x^{2}} +\kappa
\label{qmtn.4a}\end{equation}
near the boundary $\pa \tX$.
By \cite{Albin-Rochon1}, its index is given by
\[
     \Ind(\widetilde{\eth}^+_{\ALF}) = \int_{\tX} \widehat{A}(\tX;\widetilde{g}_{\ALF}) - \int_{\bbS^2} \widehat{A}(\bbS^2,h)\widehat{\eta}(\widetilde{\eth}_0 +A).
\]
So it suffices to compute explicitly the two terms on the right-hand side.  For the first term, notice that
\[
           \widehat{A}(\tX,\widetilde{g}_{\ALF})_{[4]}= -\frac{1}{8} L(\tX,\widetilde{g}_{\ALF})_{[4]}.  
\]
By the results of \cite{DW07}, we thus have
\[
     \int_{\tX} \widehat{A}(\tX,\widetilde{g}_{\ALF})= -\frac{1}{8} \int_{\tX} L(\tX, \widetilde{g}_{\ALF}) =-\frac{1}{8}( \tau(\tX) + \frac{1}{2} a-\lim\eta)
\]
where $\tau(\tX)$ is the signature of $\tX$ and the second term is an adiabatic limit given by \cite[5.4]{DW07}, 
\[
a-\lim\eta=   \frac{e}{3} - \sign(e),
\]
where $e$ is the Euler number of the circle bundle $\Phi: \pa \tX \to \bbS^2$.  In our case, $e= k$ and $\tau(X)= 1-k$, so that
\begin{equation}
    \int_{\tX} \widehat{A}(\tX,\widetilde{g}_{\ALF}) = -\frac{1}{8}( 1-k+ \frac{k}{3} -1)= \frac{k}{12}.
\label{qmtn.5}\end{equation}  
On the other hand, using the result of \cite{Zhang94}, we find that
\[
\int_{\bbS^2} \widehat{A}(\bbS^{2},h)\widehat{\eta}(\eth_0 +A)= \int_{\bbS^2} \widehat{A}(\bbS^2,h)\widehat{\eta}(\eth_0)= \frac{e}{12}= \frac{k}{12},
\]
from which the result follows.  
\end{proof}
\begin{remark}
For the conformally related fibred cusp metric $x^2 \widetilde{g}_{\ALF}$, the associated Dirac operator is Fredholm by the criterion of \cite{Vaillant}.  Using the index formula of \cite{Vaillant} and \cite{DW07} as above, one can compute that the index is also zero in this case.  We leave the details to the reader.
\label{qmtn.6}\end{remark}
Suppose now that $\Gamma$ is a finite cyclic subgroup of $\SO(3)$ acting freely on the monopole points $p_1,\ldots,p_k$, where we assume the origin of $\bbR^3$ is the center of mass of the monopole points.  According to \cite[Proposition~2.7]{HV09} and the discussion at the end of \cite[\S~7]{Wright2011}, the action of $\Gamma$ can be lifted to a free action on $\tX$ by isometries.  The metric $\widetilde{g}_{\ALF}$ then descends to the quotient $X= \tX/\Gamma$ to give an $\cF$-metric $g_{\ALF}$.  

If we assume that $|\Gamma |$ is odd, then we know from \cite[Proposition~3G.1]{Hatcher} that the quotient map $q:\tX\to X$ induces an injection
\[
         q^*: H^*(X;\bbZ_2)\to H^*(\tX;\bbZ_2).
\] 
Since $\tX$ is spin with a unique spin structure, this forces $X$ to also be spin with a unique spin structure.  This means the action of $\Gamma$ lifts to an action on the spinor bundle and we get a corresponding Dirac operator $\eth_{\ALF}$ on the quotient $X$.  This operator $\eth_{\ALF}$ is not Fredholm.  However, since the boundary family $\widetilde{\eth}_0$ is $\Gamma$-invariant as well as the family of projections $A$, we see the perturbation $\tQ\in \Psi^{-\infty}_{\Phi}(\tX;\tS^+,\tS^-)$ can be chosen to be $\Gamma$-invariant.  This gives a corresponding operator $Q\in \Psi^{-\infty}_{\cF}(X;S^+,S^-)$ on $X$ such that $\eth_{\ALF}^+ +Q$ is Fredholm.  Its index can be computed using Theorem~\ref{ad.10}.  We will consider the special case where $\Gamma= \bbZ_k$ and $(\tX,\widetilde{g}_{\ALF})$ is a gravitational instanton of type $A_{k-1}$.

\begin{proposition}
Suppose $k$ is odd and $\Gamma= \bbZ_k$ acts freely by isometry on a gravitational instanton $(\tX;\widetilde{g}_{\ALF})$ of type $A_{k-1}$.  Then the index of the Fredholm operator $\eth_{\ALF}^+ + Q$ described above is given by
\[
   \Ind(\eth_{\ALF}^+ +Q) =  \frac{k^2-1}{12k^2}- \frac{1}{4k^2} \sum_{j=1}^{k^2-1} \csc\left( \frac{\pi j(1+k)}{k^2}\right)\csc\left(\frac{\pi j(1-k)}{k^2}\right).
\]
\label{qmtn.7}\end{proposition}
\begin{proof}
Notice first that as in the proof of Proposition~\ref{qmtn.4}, we can replace $g_{\ALF}$ by a product-type metric to compute the index, that is, we can change $\widetilde{g}_{\ALF}$ near $\pa \tX$ to be of the form \eqref{qmtn.4a}
without changing the index.  By Theorem~\ref{ad.10} and Proposition~\ref{qmtn.4}, the index of $\eth_{\ALF}^+ +Q$ is then given by
\[
         \Ind(\eth_{\ALF}^+ +Q)= \rho_A = \frac{\widetilde{\eta}(\delta)}{k}- \eta(\delta).
\]
To compute $\rho_A$, notice that $h$ is a multiple of the round metric on $\bbS^2$.  By rescaling the boundary defining function if needed, we can in fact assume $h$ is precisely the standard round metric on $\bbS^2$.  Now, for $\delta>0$, the metric
\[
      \widetilde{g}(\delta)= \frac{\Phi^*h}{\delta^2}+ \kappa
\]
is a metric on the lens space $\bbS^3/\bbZ_k$.  For an appropriate value $\delta_0$ of $\delta$, it is in fact a multiple of the standard metric on this lens space, that is, the metric coming from the standard round metric on $\bbS^3$.  For other values of $\delta$, the lift $\hat{g}(\delta)$ of $\widetilde{g}(\delta)$ to $\bbS^3$ is (a multiple of) one of the $S^3$ left-invariant metrics considered in \cite[\S3.1]{Hitchin74} that are also right-invariant under $\bbS^1\subset \bbS^3$.  In particular, by \cite[\S~3.3]{Hitchin74}, these metrics have positive scalar curvature.  By Lichnerowicz formula, this means the Dirac operator $\widetilde{\eth}_0(\delta)$ is invertible.  

From the explicit description \cite[p.29]{Hitchin74} of the lift of $\widetilde{\eth}_0(\delta)$ to $\bbS^3$, we see that the family of projections $A$, seen as an operator on $\bbS^3/\bbZ_k$ acting fibrewise in the $\bbS^1$ fibres of $\bbS^3,$ is such that
\[
       (\widetilde{\eth}_0(\delta) +\epsilon A)^2= \widetilde{\eth}_0(\delta)^2 +\epsilon^2 A^2, \quad \epsilon\ge 0.
\]
In particular, the family $\delta\mapsto \widetilde{\eth}_0(\delta)+ (\delta_0-\delta)A$ is invertible for $\delta\in (0,\delta_0]$.  By choosing $\rho\in \cS(T^*\bbS^2\times \bbR)$ to be equal to 1 in a sufficiently large neighborhood of the zero section (so that its inverse Fourier transform approximates a delta function supported on the zero section), we can choose the family $\widetilde{Q}(\delta)$ in \eqref{ad.7} such that 
\[
                  (\delta_0-\delta)\widehat{N}_{\Phi}(\widetilde{Q}(\delta))(0)
\]
is as close as we want to the family $\delta\mapsto (\delta_0-\delta)A$, say in the operator norm of operators acting from 
the $H^1$ Sobolev space to $L^2$.  This means we can choose $\widetilde{Q}(\delta)$ such that 
\[
     \delta\mapsto \widetilde{\eth}_0(\delta) + (\delta_0-\delta)\widehat{N}_{\Phi}(\widetilde{Q}(\delta))(0)
\] 
is invertible for $\delta\in (0,\delta_0]$.  Thus, we can use this family at $\delta=\delta_0$ to compute $\rho_A$.  This gives
\[
     \rho_A= \frac{\eta(\widetilde{\eth}_0(\delta_0))}{k} - \eta(\eth_0(\delta_0)).
\]
These correspond to the eta invariants of Dirac operators on odd dimensional spherical space forms and can be written explicitly in terms of Dedekind sums \cite[Lemma~2.1]{Gilkey84},
\begin{equation}
\begin{gathered}
\eta(\widetilde{\eth}_0(\delta_0))= \frac{1}{4k} \sum_{j=1}^{k-1} \csc^2\left( \frac{\pi j}{k} \right)= \frac{k^2 -1}{12k};  \\
 \eta(\eth_0(\delta_0))= \frac{1}{4k^2} \sum_{j=1}^{k^2-1} \csc\left( \frac{\pi j (1+k)}{k^2}\right) \csc\left( \frac{\pi j (1-k)}{k^2}\right),
\end{gathered}
\label{qmtn.9}\end{equation}
from which the result follows.  
\end{proof}
This gives the following integrality result.
\begin{corollary}
For $k$ odd, the expression 
\[
\frac{k^2-1}{12k^2}- \frac{1}{4k^2} \sum_{j=1}^{k^2-1} \csc\left( \frac{\pi j(1+k)}{k^2}\right)\csc\left(\frac{\pi j(1-k)}{k^2}\right)
\]
is an integer.
\label{qmtn.10}\end{corollary}
\begin{remark}
Numerical calculations show that this integer is 0 for $k\le 200$ odd.  It is tempting to speculate it should be equal to zero for all odd natural numbers.  
\label{qmtn.11}\end{remark}

\bibliography{pomfb}
\bibliographystyle{amsplain}

\end{document}